\documentclass[11pt]{amsart}

\usepackage{a4wide, amsmath, amsfonts, amssymb, amsthm, breqn}
\usepackage[hyperfootnotes=false]{hyperref}
\usepackage{color}
\usepackage{float}
\usepackage{enumerate}

\allowdisplaybreaks[2]

\numberwithin{equation}{section}
\newtheorem{theorem}{Theorem}[section]
\newtheorem{lemma}[theorem]{Lemma}

\newtheorem{remark}[theorem]{Remark}
\newtheorem{proposition}[theorem]{Proposition}
\newtheorem{definition}[theorem]{Definition}
\newtheorem{example}[theorem]{Example}
\newtheorem{assumption}[theorem]{Assumption}
\newtheorem*{property}{Property}

\newcommand{\dd}{\,\mathrm{d}}
\newcommand{\R}{\mathbb{R}}
\newcommand{\N}{\mathbb{N}}
\newcommand{\Z}{\mathbb{Z}}
\newcommand{\1}{\mathbf{1}}
\newcommand{\D}{\mathrm{D}}
\newcommand{\bS}{\mathbf{S}}
\newcommand{\bX}{\mathbf{X}}
\newcommand{\crpXq}{\mathcal{V}^q_X}
\newcommand{\crpXp}{\mathcal{V}^p_X}
\newcommand{\crpSq}{\mathcal{V}^q_S}
\newcommand{\crpSp}{\mathcal{V}^p_S}
\newcommand{\crpZq}{\mathcal{V}^q_Z}

\newcommand{\bZ}{\mathbf{Z}}
\newcommand{\X}{\mathbb{X}}
\renewcommand{\S}{\mathbb{S}}
\renewcommand{\d}{\mathrm{d}}
\renewcommand{\epsilon}{\varepsilon}
\newcommand{\cP}{\mathcal{P}}
\newcommand{\W}{\mathbb{W}}
\newcommand{\bW}{\mathbf{W}}
\newcommand{\E}{\mathbb{E}}

\title[Portfolio theory with rough paths]{Model-free portfolio theory: a rough path approach}

\author[Allan]{Andrew L. Allan}
\address{Andrew L. Allan, Durham University, United Kingdom}
\email{andrew.l.allan@durham.ac.uk}

\author[Cuchiero]{Christa Cuchiero}
\address{Christa Cuchiero, University of Vienna, Austria}
\email{christa.cuchiero@univie.ac.at}

\author[Liu]{Chong Liu}
\address{Chong Liu, ShanghaiTech University, China}
\email{liuchong@shanghaitech.edu.cn}

\author[Pr{\"o}mel]{David J. Pr{\"o}mel}
\address{David J. Pr{\"o}mel, University of Mannheim, Germany}
\email{proemel@uni-mannheim.de}

\date{\today}

\begin{document}

\begin{abstract}
  Based on a rough path foundation, we develop a model-free approach to stochastic portfolio theory (SPT). Our approach allows to handle significantly more general portfolios compared to previous model-free approaches based on F{\"o}llmer integration. Without the assumption of any underlying probabilistic model, we prove a pathwise formula for the relative wealth process which reduces in the special case of functionally generated portfolios to a pathwise version of the so-called master formula of classical SPT. We show that the appropriately scaled asymptotic growth rate of a far reaching generalization of Cover's universal portfolio based on controlled paths coincides with that of the best retrospectively chosen portfolio within this class. We provide several novel results concerning rough integration, and highlight the advantages of the rough path approach by showing that (non-functionally generated) log-optimal portfolios in an ergodic It{\^o} diffusion setting have the same asymptotic growth rate as 
Cover's universal portfolio and the best retrospectively chosen one.  
  
\end{abstract}

\maketitle




\section{Introduction}

Classical approaches to portfolio theory, going back to the seminal work of H.~Markowitz \cite{Markowitz1959} (see also the early work of B.~de~Finetti \cite{deFinetti1940}), are essentially based on simplistic probabilistic models for the asset returns or prices. As a first step  classical portfolio selection thus requires to build and statistically estimate a probabilistic model of the future asset returns. The second step is usually to find an ``optimal'' portfolio with respect to the now fixed model. However, it is well known that the obtained optimal portfolios and their performance are highly sensitive to model misspecifications and estimation errors; see e.g.~\cite{Chopra1993,DeMiguel2007}.

In order to account for model misspecification and model risk, the concept of model ambiguity, also known as Knightian uncertainty, has gained increasing importance in portfolio theory; see e.g.~\cite{Pflug2007,Guidolin2013}. Here the rationale is to accomplish the portfolio selection with respect to a pool of probabilistic models, rather than a specific one. This has been pushed further by adopting completely \emph{model-free} (or pathwise) approaches, where the trajectories of the asset prices are assumed to be deterministic functions of time. That is, no statistical properties of the asset returns or prices are postulated; see e.g.~\cite{Pal2016,Schied2018,Cuchiero2019}. In portfolio theory there are two major approaches which provide such model-free ways of determining ``optimal'' portfolios: universal and stochastic portfolio theory.

The objective of universal portfolio theory is to find general preference-free well performing investment strategies without referring to a probabilistic setting; see \cite{Li2014} for a survey. This theory was initiated by T.~Cover \cite{Cover1991}, who showed that a properly chosen ``universal'' portfolio has the same asymptotic growth rate as the best retrospectively chosen (constantly rebalanced) portfolio in a discrete-time setting. Here, the word ``universal'' indicates the model-free nature of the constructed portfolio.

Stochastic portfolio theory (SPT), initiated by R.~Fernholz \cite{Fernholz1999,Fernholz2001}, constitutes a descriptive theory aiming to construct and analyze portfolios using only properties of observable market quantities; see \cite{Fernholz2002,Karatzas2009} for detailed introductions. While classical SPT still relies on an underlying probabilistic model, its descriptive nature leads to essentially model-free constructions of ``optimal'' portfolios.

A model-free treatment of universal and stochastic portfolio theory in \emph{continuous-time} was recently introduced in \cite{Schied2018, Cuchiero2019}, clarifying the model-free nature of these theories. So far this analysis has been limited to so-called (generalized) \emph{functionally generated portfolios}, cf.~\cite{Fernholz1999,Strong2014,Schied2018}. These are investment strategies based on logarithmic gradients of so-called portfolio generating functions. This limitation is due to the fact that the corresponding portfolio wealth processes can be defined in a purely pathwise manner only for gradient-type strategies, namely, via F{\"o}llmer's probability-free notion of It{\^o} integration; see F{\"o}llmer's pioneering work~\cite{Follmer1981} and its extensions \cite{Cont2010,Cont2019,CC:22}. Even though these  limitations do not occur in discrete time, optimal portfolio selection approaches based on functionally generated portfolios have also gained  attention in discrete time setups; see e.g.~\cite{Campbell2022}. Another strand of research is robust maximization of asymptotic growth within a pool of Markovian models as pursued in \cite{KardarasRobertson2012, Kardaras2018, Itkin2020}. While these approaches clearly account for model uncertainty, a probabilistic structure still enters via a Markovian volatility matrix and an invariant measure for the market weights process. In a similar direction goes the construction of optimal arbitrages under model uncertainty as pioneered in \cite{Fernholz2011}.

The main goal of the present article is to develop an entirely model-free portfolio theory in continuous-time, in the spirit of stochastic and universal portfolio theory, which allows one to work with a significantly larger class of investment strategies and portfolios. For this purpose, we rely on the pathwise (rough) integration offered by rough path theory---as exhibited in e.g.~\cite{Lyons2002, Lyons2007, Friz2010, Friz2020}---and assume that the (deterministic) price trajectories on the underlying financial market satisfy the so-called Property~(RIE), as introduced in \cite{Perkowski2016}; see Section~\ref{subsec: RIE}. While Property (RIE) does not require any probabilistic structure, it is satisfied, for instance, by the sample paths of semimartingale models fulfilling the condition of ``no unbounded profit with bounded risk'' and, furthermore, it ensures that rough integrals are given as limits of suitable Riemann sums. This is essential in view of the financial interpretation of the integral as the wealth process associated to a given portfolio.

In the spirit of stochastic portfolio theory, we are interested in the relative performance of the wealth processes, where the word ``relative'' may be interpreted as ``in comparison with the market portfolio''. In other words, given $d$ assets with associated price process $S = (S_t^1, \ldots, S_t^d)_{t \in [0,\infty)}$ satisfying Property (RIE), we choose the total market capitalization $S^1 + \cdots + S^d$ as num{\'e}raire, so that the primary assets are the market weights $\mu = (\mu^1_t, \ldots, \mu_t^d)_{t \in [0,\infty)}$, given by
\begin{equation*}
  \mu^i_t := \frac{S_t^i}{S_t^1 + \cdots + S_t^d}, \qquad i = 1, \dots, d,
\end{equation*}
which take values in the open unit simplex $\Delta^d_+$. The main contributions of the present work may be summarized by the following.

\begin{itemize}
  \item In Proposition~\ref{prop: log of relative wealth} we establish a pathwise formula for the relative wealth process associated to portfolios belonging to the space of controlled paths, as introduced in Definition~\ref{def: controlled path} below. This includes functionally generated portfolios commonly considered in SPT---as for instance in \cite{Strong2014,Schied2016,Karatzas2017,Ruf2019,Karatzas2020}---as well as the class which we refer to as \emph{functionally controlled portfolios}, which are portfolios of the form
  \begin{equation}\label{eq:genfuncintro}
    (\pi^F_t)^i = \mu_t^i \bigg(F^i(\mu_t) + 1 - \sum_{j=1}^d \mu_t^j F^j(\mu_t)\bigg),
  \end{equation}
  for some $F \in C^2(\overline{\Delta}^d_+;\mathbb{R}^d)$. Here, $(\pi^F)^i$ denotes the proportion of the current wealth invested in asset $i = 1, \dots, d$. In the case of functionally generated portfolios, i.e.~when $F$ is the logarithmic gradient of some real valued function, we also derive in Theorem~\ref{thm: master formula} a purely pathwise version of the classical master formula of SPT, cf.~\cite{Fernholz2002,Strong2014}.
  \item We introduce Cover's universal portfolio defined via a mixture portfolio based on the notion of controlled paths, and show that its appropriately scaled logarithmic relative wealth process converges in the long-run to that of the best retrospectively chosen portfolio; see Theorems~\ref{thm: Cover's theorem} and \ref{thm: Cover's theorem for generalized functionally generated portfolios}. This extends the results of \cite{Cuchiero2019} to a considerably larger class of investment strategies.
\item In Section~\ref{sec:Functionally controlled protfolios in probabilistic models}, we introduce a probabilistic setup where the dynamics of the market weights are described by a stochastic differential equation (SDE) driven by Brownian motion. Using the law of large numbers for the increments of the It\^o-rough path lift of Brownian motion, this setting allows to replace the scaling function of Theorem~\ref{thm: Cover's theorem for generalized functionally generated portfolios} by $1/T$. For this class of models we can thus prove that the asymptotic growth rates of Cover's universal portfolio and the best retrospectively chosen one are the same (see Theorem~\ref{thm:complogopt} (ii)). We also compare these two portfolios with the log-optimal one assuming additionally that the SDE for the market weights is ergodic. In this case the corresponding growth rates are all asymptotically equivalent, as shown in Theorem~\ref{thm:complogopt} (iii). This is analogous to the result in \cite{Cuchiero2019}, however now proved for the significantly larger class of functionally controlled portfolios.
\item We develop novel results in the theory of rough paths to allow for the pathwise treatment of portfolio theory. In particular, these results include an extension of \cite[Theorem~4.19]{Perkowski2016}, stating that the rough integral can be represented as a limit of left-point Riemann sums---see Theorem~\ref{thm: Ito integral for smooth transformed RIE path}---and the associativity of rough integration, exhibited in Section~\ref{subsec: associativity of rough integration}.
\end{itemize}

One important motivation for our work comes from classical considerations of the log-optimal portfolio in ergodic It{\^o} diffusion models for the market weights process. Indeed, this is one prominent example of an ``optimal'' portfolio that does not belong, in general, to the class of (generalized) functionally generated portfolios, but is still a functionally controlled portfolio of the form \eqref{eq:genfuncintro}; see Section~\ref{subsec:assgrowth}. As illustrated numerically in Figure~\ref{fig1}, the log-optimal portfolio (an example of a functionally controlled portfolio) might significantly outperform a corresponding ``best'' functionally generated portfolio. Indeed, the blue line illustrates the expected utility of the log-optimal portfolio over time, whereas the orange line depicts that of a certain best functionally generated portfolio. For the details of this example we refer to Section~\ref{subsec:comp}.
\begin{figure}\label{fig1}
  \centering
  \includegraphics[scale=1.0]{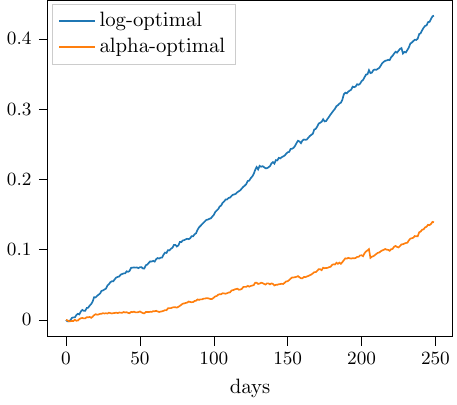}
   \caption{Expected utility of the log-optimal vs.~the alpha-optimal portfolio over time.}
\end{figure}
This indicates that going beyond functionally generated portfolios can have a substantial benefit. This holds true in particular for Cover's universal portfolio when defined as a mixture of portfolios of the form \eqref{eq:genfuncintro}, since in ergodic market models it asymptotically achieves the growth rate of the log-optimal portfolio (see Theorem~\ref{thm:complogopt}). Note that, due to the rough path approach, both the relative wealth processes obtained by investing according to the log-optimal portfolio and according to the universal portfolio make sense for every individual price trajectory. This also gives a theoretical justification for learning a (non-functionally generated) log-optimal portfolio from the observations of a single price path.

\medskip

\noindent\textbf{Outline:} In Section~\ref{sec:rough integration intro} we provide an overview of the essential concepts of rough paths and rough integration relevant for our financial application. In Section~\ref{sec:master formula} we introduce the pathwise description of the underlying financial market and study the growth of wealth processes relative to that of the market portfolio, which leads us to a pathwise master formula analogous to that of classical SPT. Section~\ref{sec:Cover portfolio} is dedicated to Cover's universal portfolio and to proving that its appropriately scaled asymptotic growth rate is equal to that of the best retrospectively chosen portfolio. In Section~\ref{sec:Functionally controlled protfolios in probabilistic models} we introduce a probabilistic setup and show under an ergodicity assumption that the asymptotic growth rate  coincides for Cover's universal portfolio, the best retrospectively chosen one and the log-optimal one. In this setting, we also compare the wealth processes of functionally controlled portfolios and functionally generated ones, illustrating their performance by means of a concrete numerical example. Appendices~\ref{sec:rough integration} and \ref{sec:appendix on RIE} collect findings concerning rough path theory and rough integration needed to establish the aforementioned results.

\medskip

\noindent\textbf{Acknowledgment:} A.~L.~Allan gratefully acknowledges financial support by the Swiss National Science Foundation via Project 200021\textunderscore 184647. C.~Cuchiero gratefully acknowledges financial support from the Vienna Science and Technology Fund (WWTF) under grant MA16-021 and by the Austrian Science Fund (FWF) through grant Y 1235 of the START-program. C.~Liu gratefully acknowledges support from the Early Postdoc.~Mobility Fellowship (No.~P2EZP2\textunderscore 188068) of the Swiss National Science Foundation, and from the G.~H.~Hardy Junior Research Fellowship in Mathematics awarded by New College, Oxford. The authors would also like to thank the anonymous referees for their valuable suggestions, which led to a significant improvement of the present paper.

\section{Rough integration for financial applications}\label{sec:rough integration intro}

In this section we provide the essential concepts from rough path theory for our applications in model-free portfolio theory. Additional results regarding rough integration are developed in the appendices. For more detailed introductions to rough path theory we refer to the books \cite{Lyons2002, Lyons2007, Friz2010, Friz2020}. Let us begin by introducing some basic notation commonly used in the theory of rough paths.

\subsection{Basic notation}

Let $(\R^d,|\,\cdot\,|)$ be standard Euclidean space and let $A \otimes B$ denote the tensor product of two vectors $A, B \in \R^d$, i.e.~the $d \times d$-matrix with $(i,j)$-component given by $[A \otimes B]^{ij} = A^i B^j$ for $1 \leq i, j \leq d$. The space of continuous paths $S \colon [0,T] \to \R^d$ is given by $C([0,T];\R^d)$, and $\|S\|_{\infty,[0,T]}$ denotes the supremum norm of $S$ over the interval $[0,T]$. For the increment of a path $S \colon [0,T] \to \R^d$, we use the standard shorthand notation
\begin{equation*}
  S_{s,t} := S_t - S_s, \qquad \text{for} \quad (s,t) \in \Delta_{[0,T]} := \big\{(u,v) \in [0,T]^2 : u \leq v\big\}.
\end{equation*}
For any partition $\mathcal{P} = \{0 = t_0 < t_1 < \dots < t_{N} = T\}$ of an interval $[0,T]$, we denote the mesh size of $\mathcal{P}$ by $|\mathcal{P}| := \max \{|t_{k+1} - t_k| : k = 0, 1, \dots, N-1\}$. A \emph{control function} is defined as a function $c \colon \Delta_{[0,T]} \to [0, \infty)$ which is superadditive, in the sense that $c(s,u) + c(u,t) \leq c(s,t)$ for all $0 \leq s \leq u \leq t \leq T$. For $p \in [1,\infty)$, the $p$-variation of a path $S \in C([0,T];\R^d)$ over the interval $[s,t]$ is defined by
\begin{equation*}
  \|S\|_{p,[s,t]} := \sup_{\mathcal{P} \subset [s,t]} \bigg(\sum_{[u,v] \in \mathcal{P}} |S_{u,v}| ^p\bigg)^{\hspace{-1pt}\frac{1}{p}},
\end{equation*}
where the supremum is taken over all finite partitions $\mathcal{P}$ of the interval $[s,t]$, and we use the abbreviation $\|S\|_{p} := \|S\|_{p,[0,T]}$. We say that $S$ has finite $p$-variation if $\|S\|_{p} < \infty$, and we denote the space of continuous paths with finite $p$-variation by $C^{p\textup{-var}}([0,T];\R^d)$. Note that $S$ having finite $p$-variation is equivalent to the existence of a control function $c$ such that $|S_{s,t}|^p \leq c(s,t)$ for all $(s,t) \in \Delta_{[0,T]}$. (For instance, one can take $c(s,t) = \|S\|_{p,[s,t]}^p$.) Moreover, for a two-parameter function $\mathbb{S}\colon \Delta_{[0,T]} \to \R^{d\times d}$ we introduce the corresponding notion of $p$-variation by
\begin{equation*}
  \|\mathbb{S}\|_{p,[s,t]} := \sup_{\mathcal{P} \subset [s,t]} \bigg(\sum_{[u,v] \in \mathcal{P}} |\mathbb{S}_{u,v}|^p\bigg)^{\hspace{-1pt}\frac{1}{p}},
\end{equation*}
for $p \in [1,\infty)$.

\smallskip

Given a $k \in \N$ and a domain $A \subseteq \R^d$, we will write $f \in C^k(A;\mathbb{R}^d)$, or sometimes simply $f \in C^k$, to indicate that a function $f$ defined on $A$ with values in $\R^d$ is $k$-times continuously differentiable (seen as restriction of $C^k$-functions on $\R^d$ if $A$ is closed), and we will make use of the associated norm
\begin{equation*}
  \|f\|_{C^k} := \max_{0 \leq n \leq k} \|\D^nf\|_{\infty},
\end{equation*}
where $\D^nf$ denotes the $n$\textsuperscript{th} order derivative of $f$, and $\|\,\cdot\,\|_{\infty}$ denotes the supremum norm.

For a $k \in \N$ and $\gamma \in (0,1]$, we will write $f \in C^{k + \gamma}(A;\R^d)$, or just $f \in C^{k + \gamma}$, to mean that a function $f$ defined on $A$ is $k$-times continuously differentiable (in the Fr\'echet sense), and that its $k$\textsuperscript{th} order derivative $\D^kf$ is locally $\gamma$-H{\"o}lder continuous. In this case we use the norm
\begin{equation*}
  \|f\|_{C^{k + \gamma}} := \max_{0 \leq n \leq k} \|\D^nf\|_{\infty} + \|\D^kf\|_{\gamma\textup{-H\"ol}},
\end{equation*}
where $\|\,\cdot\,\|_{\gamma\textup{-H\"ol}}$ denotes the $\gamma$-H{\"o}lder norm.

Finally, given two vector spaces $U, V$, we write $\mathcal{L}(U;V)$ for the space of linear maps from $U$ to $V$.

\medskip

Let $(E,\|\cdot\|)$ be a normed space and let $f, g \colon E \to \R$ be two functions. We shall write $f \lesssim g$ or $f \leq Cg$ to mean that there exists a constant $C > 0$ such that $f(x) \leq C g(x)$ for all $x \in E$. Note that the value of such a constant may change from line to line, and that the constants may depend on the normed space, e.g.~through its dimension or regularity parameters.

\subsection{Rough path theory and Property (RIE)}\label{subsec: RIE}

Let us briefly recall the fundamental definitions of a rough path and of a controlled path, which allow to set up rough integration.

\begin{definition}
  For $p \in (2,3)$, a \emph{$p$-rough path} is defined as a pair $\bS = (S,\S)$, consisting of a continuous path $S \colon [0,T] \to \R^d$ and a continuous two-parameter function $\S \colon \Delta_{[0,T]} \to \R^{d \times d}$, such that $\|S\|_{p} < \infty$, $\|\S\|_{p/2} < \infty$, and Chen's relation
  \begin{equation}\label{eq:Chens relation}
    \S_{s,t} = \S_{s,u} + \S_{u,t} + S_{s,u} \otimes S_{u,t}
  \end{equation}
  holds for all $0 \leq s \leq u \leq t \leq T$.
\end{definition}

\begin{remark} 
  The success of rough path theory in probability theory is based on the observation that sample paths of many important stochastic processes such as Brownian motion, semimartingales and Markov processes can be enhanced to a rough path, by defining the ``enhancement'' $\S$ via stochastic integration; see e.g.~\cite[Part~III]{Friz2010}.
\end{remark}

\begin{definition}\label{def: controlled path}
  Let $p \in (2,3)$ and $q \geq p$ be such that $2/p + 1/q > 1$, and let $r > 1$ be such that $1/r = 1/p + 1/q$. Let $S \in C^{p\textup{-var}}([0,T];\R^d)$, $F \colon [0,T] \to \R^d$ and $F' \colon [0,T] \to \mathcal{L}(\R^d;\R^d)$ be continuous paths. The pair $(F,F')$ is called a \emph{controlled path} with respect to $S$ (or an \emph{$S$-controlled path}), if the \emph{Gubinelli derivative} $F'$ has finite $q$-variation, and the \emph{remainder} $R^F$ has finite $r$-variation, where $R^F \colon \Delta_{[0,T]} \to \R^d$ is defined implicitly by the relation
  \begin{equation*}
    F_{s,t} = F'_s S_{s,t} + R^F_{s,t} \qquad \text{for} \quad (s,t) \in \Delta_{[0,T]}.
  \end{equation*}
  We denote the space of controlled paths with respect to $S$ by $\crpSq = \crpSq([0,T];\R^d)$, which becomes a Banach space when equipped with the norm
  \begin{equation*}
    \|F,F'\|_{\crpSq,[0,T]} := |F_0| + |F'_0| + \|F'\|_{q,[0,T]} + \|R^F\|_{r,[0,T]}.
  \end{equation*}
\end{definition}

\begin{example}\label{ex: controlled path 1+epsilon}
  For a path $S \in C^{p\textup{-var}}([0,T];\R^d)$ with $p \in (2,3)$, the prototypical example of a controlled path is $(f(S),\D f(S)) \in \crpSq$ for any $f \in C^{1 + \epsilon}$ with $\epsilon \in (p - 2,1]$ and $q = p/\epsilon$. Examples of more general controlled paths are discussed in Remark~\ref{remark: functionally generalized portfolio} and Section~\ref{subsec:base set of universal protfolio} in the context of universal portfolios.
\end{example}


Based on the above definitions, one can establish the existence of the rough integral of a controlled path $(F,F')$ with respect to a $p$-rough path $\bS$. See \cite{Friz2020} for the corresponding theory presented in terms of H{\"o}lder regularity. The following formulation of rough integration in the language of $p$-variation can be found in e.g.~\cite[Theorem~4.9]{Perkowski2016}.

\begin{theorem}[Rough integration]\label{thm: rough integral exists}
  Let $p \in (2,3)$ and $q \geq p$ be such that $2/p + 1/q > 1$, and let $r > 1$ be such that $1/r = 1/p + 1/q$. Let $\bS = (S,\S)$ be a $p$-rough path and let $(F,F') \in \crpSq$ be a controlled path with remainder $R^F$. Then the limit 
  \begin{equation}\label{eq: rough integration}
    \int_0^T F_u \dd \bS_u := \lim_{|\mathcal{P}| \to 0} \sum_{[s,t] \in \mathcal{P}} F_s S_{s,t} + F'_s \S_{s,t}
  \end{equation}
  exists along every sequence of partitions $\mathcal{P}$ of the interval $[0,T]$ with mesh size $|\mathcal{P}|$ tending to zero, and takes values in $\mathbb{R}$. We call this limit the \emph{rough integral} of $(F,F')$ against $\bS$. Here, the product $F_s S_{s,t}$ is understood as the Euclidean inner product, and the product $F'_s \S_{s,t}$ also takes values in $\R$ since the derivative $F'$ takes values in $\mathcal{L}(\R^d;\R^d) \cong \mathcal{L}(\R^{d \times d};\R)$. Moreover, we have the estimate
  \begin{equation}\label{eq:rough integral estimate}
    \bigg|\int_s^t F_u \dd \bS_u - F_s S_{s,t} - F'_s \S_{s,t}\bigg| 
    \leq C\big(\|R^F\|_{r,[s,t]} \|S\|_{p,[s,t]} + \|F'\|_{q,[s,t]} \|\S\|_{\frac{p}{2},[s,t]}\big),
  \end{equation}
  where the constant~$C$ depends only on $p,q$ and~$r$.
\end{theorem}

In Theorem~\ref{thm: rough integral exists} we defined the rough integral of a controlled path $(F,F')$ against a rough path $\bS = (S,\S)$. As noted in \cite[Remark~4.12]{Friz2020}, one can actually define a more general integral of a controlled path $(F,F')$ against another controlled path $(G,G')$.

\begin{lemma}\label{lem: int of controlled paths exists}
  Let $\bS = (S,\S)$ be a $p$-rough path, and let $(F,F'), (G,G') \in \crpSq$ be two controlled paths with remainders $R^F$ and $R^G$, respectively. Then the limit
  \begin{equation}\label{eq:int of controlled paths defn}
    \int_0^T F_u \dd G_u := \lim_{|\mathcal{P}| \to 0} \sum_{[s,t] \in \mathcal{P}} F_s G_{s,t} + F'_s G'_s \S_{s,t}
  \end{equation}
  exists along every sequence of partitions~$\mathcal{P}$ of the interval $[0,T]$ with mesh size $|\mathcal{P}|$ tending to zero, and comes with the estimate
  \begin{align}\label{eq:est int of controlled paths}
    \begin{split}
    &\bigg|\int_s^t F_u \,\d G_u - F_s G_{s,t} - F'_s G'_s \S_{s,t}\bigg|\\
    &\quad \leq C\Big(\|F'\|_\infty (\|G'\|_{q,[s,t]}^q + \|S\|_{p,[s,t]}^p)^{\frac{1}{r}} \|S\|_{p,[s,t]}+ \|F\|_{p,[s,t]} \|R^G\|_{r,[s,t]} \\
    &\quad\quad\quad\quad+ \|R^F\|_{r,[s,t]} \|G'\|_\infty \|S\|_{p,[s,t]} + \|F'G'\|_{q,[s,t]} \|\S\|_{\frac{p}{2},[s,t]}\Big),
    \end{split}
  \end{align}
  where the constant $C$ depends only on $p, q$ and $r$.
\end{lemma}

\begin{proof}
  Set $\Xi_{s,t} := F_s G_{s,t} + F'_s G'_s \S_{s,t}$ and $\delta \Xi_{s,u,t} := \Xi_{s,t} - \Xi_{s,u} - \Xi_{u,t}$ for $0 \leq s \leq u \leq t \leq T$. Using Chen's relation \eqref{eq:Chens relation}, one can show that
  \begin{equation}\label{eq:delta Xi int of controlled paths}
    \delta \Xi_{s,u,t} = -F'_s G'_{s,u} S_{s,u} S_{u,t} - F_{s,u} R^G_{u,t} - R^F_{s,u} G'_u S_{u,t} - (F'G')_{s,u} \S_{u,t}.
  \end{equation}
  Since $1/r = 1/p + 1/q$, Young's inequality gives
  \begin{align*}
    |-F'_s G'_{s,u} S_{s,u} S_{u,t}| 
    &\leq \|F'\|_\infty \|G'\|_{q,[s,u]} \|S\|_{p,[s,u]} \|S\|_{p,[u,t]}\\
    &\lesssim \|F'\|_\infty (\|G'\|_{q,[s,u]}^q + \|S\|_{p,[s,u]}^p)^{\frac{1}{r}} \|S\|_{p,[u,t]} = w_1(s,u)^{\frac{1}{r}} w_2(u,t)^{\frac{1}{p}},
  \end{align*}
  where $w_1(s,u) := \|F'\|_\infty^r (\|G'\|_{q,[s,u]}^q + \|S\|_{p,[s,u]}^p)$ and $w_2(u,t) := \|S\|_{p,[u,t]}^p$ are control functions. Treating the other three terms on the right-hand side of \eqref{eq:delta Xi int of controlled paths} similarly, we deduce the hypotheses of the generalized sewing lemma \cite[Theorem~2.5]{Friz2018}, from which the result follows.
\end{proof}

Rough integration offers strong pathwise stability estimates, and may be viewed as arguably the most general pathwise integration theory, generalizing classical notions of integration such as those of Riemann--Stieltjes, Young and F{\"o}llmer, and allowing one to treat many well-known stochastic processes as integrators; see e.g.~\cite{Friz2020}. However, from the perspective of mathematical finance, rough integration comes with one apparent flaw: the definition of rough integral \eqref{eq: rough integration} is based on so-called ``compensated'' Riemann sums, and thus does not (at first glance) come with the natural interpretation as the capital gain process associated to an investment in a financial market. Indeed, let us suppose that $S$ represents the asset prices on a financial market and $F$ an investment strategy. In this case, neither the associated rough path $\bS = (S,\S)$ nor the controlled path $(F,F^\prime)$, assuming they exist, are uniquely determined by $S$ and $F$, but rather the value of the rough integral $\int_0^T F_u \dd \bS_u$ will depend in general on the choices of $\S$ and $F^\prime$. Moreover, the financial meaning of the term $F^\prime_s \S_{s,t}$ appearing in the compensated Riemann sum in \eqref{eq: rough integration} is far from obvious.

\smallskip

As observed in \cite{Perkowski2016}, the aforementioned drawback of rough integration from a financial perspective can be resolved by introducing the following property of the price path $S$.

\begin{property}[\textbf{RIE}]
  Let $p\in (2,3)$ and let $\mathcal{P}^n = \{0 = t^n_0 < t^n_1 < \cdots < t^n_{N_n} = T\}$, $n \in \N$, be a sequence of partitions of the interval $[0,T]$, such that $|\mathcal{P}^n| \to 0$ as $n \to \infty$. For $S \in C([0,T];\R^d)$, we define $S^n \colon [0,T] \to \R^d$ by
  \begin{equation*}
    S^n_t := S_T \1_{\{T\}}(t) + \sum_{k=0}^{N_n - 1} S_{t^n_k} \1_{[t^n_k,t^n_{k+1})}(t), \qquad t \in [0,T],
  \end{equation*}
  for each $n \in \N$. We assume that:
  \begin{itemize}
    \item the Riemann sums $\int_0^t S^n_u \otimes \d S_u := \sum_{k=0}^{N_n-1} S_{t^n_k} \otimes S_{t^n_k \wedge t,t^n_{k+1} \wedge t}$ converge uniformly as $n \to \infty$ to a limit, which we denote by $\int_0^t S_u \otimes \d S_u$, $t \in [0,T]$,
    \item and that there exists a control function $c$ such that\footnote{Here and throughout, we adopt the convention that $\frac{0}{0} := 0$.}
    \begin{equation*}
      \sup_{(s,t) \in \Delta_{[0,T]}} \frac{|S_{s,t}|^p}{c(s,t)} + \sup_{n \in \N} \, \sup_{0 \leq k < \ell \leq N_n} \frac{\big|\int_{t^n_k}^{t^n_\ell} S^n_u \otimes \d S_u - S_{t^n_k}\otimes S_{t^n_k,t^n_\ell}\big|^{\frac{p}{2}}}{c(t^n_k,t^n_\ell)} \leq 1.
    \end{equation*}
  \end{itemize}
\end{property}

\begin{definition}
  A path $S\in C([0,T];\R^d)$ is said to satisfy \textup{(RIE)} with respect to $p$ and $(\mathcal{P}^n)_{n \in \N}$, if $p$, $(\mathcal{P}^n)_{n \in \N}$ and $S$ together satisfy Property \textup{(RIE)}.
\end{definition}

As discussed in detail in \cite{Perkowski2016}, if a path $S \in C([0,T];\R^d)$ satisfies \textup{(RIE)} with respect to $p$ and $(\mathcal{P}^n)_{n \in \N}$, then $S$ can be enhanced to a $p$-rough path $\bS = (S,\S)$ by setting
\begin{equation}\label{eq:defn A st}
  \S_{s,t} := \int_s^t S_u \otimes \d S_u - S_s \otimes S_{s,t}, \qquad \text{for} \quad (s,t) \in \Delta_{[0,T]}.
\end{equation}
In other words, Property~\textup{(RIE)} ensures the existence of a rough path associated to the path~$S$. The advantage of the (more restrictive) Property~\textup{(RIE)} is that it guarantees that the corresponding rough integrals can be well approximated by classical left-point Riemann sums, as we will see in Section~\ref{subsec: rough integrals as Riemann sums}, thus allowing us to restore the financial interpretation of such integrals as capital processes.

\begin{remark}\label{remark: financial models satisfy RIE}
  The assumption that the underlying price paths satisfy Property~\textup{(RIE)} appears to be rather natural in the context of portfolio theory. Indeed, in stochastic portfolio theory the price processes are commonly modelled as semimartingales fulfilling the condition of ``no unbounded profit with bounded risk'' (NUPBR); see e.g.~\cite{Fernholz2002}. The condition (NUPBR) is also essentially the minimal condition required to ensure that expected utility maximization problems are well-posed; see \cite{Karatzas2007,Imkeller2015}. As established in \cite[Proposition~2.7 and Remark~4.16]{Perkowski2016}, the sample paths of semimartingales fulfilling (NUPBR) almost surely satisfy Property~\textup{(RIE)} with respect to every $p \in (2,3)$ and a suitably chosen sequence of partitions.
\end{remark}

\subsection{The bracket process and a rough It{\^o} formula}

A vital tool in many applications of stochastic calculus is It{\^o}'s formula, and it will also be an important ingredient in our contribution to portfolio theory. Usually, (pathwise) It{\^o} formulae are based on the notion of quadratic variation. In rough path theory, a similar role as that of the quadratic variation is played by the so-called bracket of a rough path, cf.~\cite[Definition~5.5]{Friz2020}.

\begin{definition}\label{def: bracket}
  Let $\bS = (S,\S)$ be a $p$-rough path and let $\textup{Sym}(\S)$ denote the symmetric part of $\S$. The \emph{bracket} of $\bS$ is defined as the path $[\bS] \colon [0,T] \to \R^{d \times d}$ given by
  \begin{equation*}
    [\bS]_t := S_{0,t} \otimes S_{0,t} - 2\textup{Sym}(\S_{0,t}), \qquad t \in [0,T].
  \end{equation*}
\end{definition}

The bracket of a rough path allows one to derive It{\^o} formulae for rough paths. For this purpose, note that $[\bS]$ is a continuous path of finite $p/2$-variation, which can be seen from the observation that
\begin{equation*}
  [\bS]_{s,t} = [\bS]_t - [\bS]_s = S_{s,t} \otimes S_{s,t} - 2\textup{Sym}(\S_{s,t}), \qquad \text{for all} \quad (s,t) \in \Delta_{[0,T]}.
\end{equation*}
The following It{\^o} formula for rough paths can be proven almost exactly as the one in \cite[Theorem~7.7]{Friz2020}, so we will omit its proof here; see also \cite[Theorem~2.12]{Friz2018}.

\begin{proposition}\label{prop: general Ito formula for rough paths}
  Let $\bS = (S,\S)$ be a $p$-rough path and let $\Gamma \in C^{\frac{p}{2}\textup{-var}}([0,T];\R^d)$. Suppose that $F, F'$ and $F''$ are such that $(F,F'), (F',F'') \in \crpSq$, and $F = \int_0^\cdot F'_u \dd \bS_u + \Gamma$. If $g \in C^{p + \epsilon}$ for some $\epsilon > 0$, then, for every $t \in [0,T]$, we have
  \begin{equation*}
    g(F_t) = g(F_0) + \int_0^t \D g(F_u) F'_u \dd \bS_u + \int_0^t \D g(F_u) \dd \Gamma_u + \frac{1}{2} \int_0^t \D^2g(F_u) (F'_u \otimes F'_u) \dd [\bS]_u.
  \end{equation*}
\end{proposition}

Assuming Property~\textup{(RIE)}, it turns out that the bracket $[\bS]$ of a rough path $\bS = (S,\S)$ does coincide precisely with the quadratic variation of the path $S$ in the sense of F{\"o}llmer~\cite{Follmer1981}.

\begin{lemma}\label{lem: property of bracket of rough paths}
  Suppose that $S\in C([0,T];\R^d)$ satisfies \textup{(RIE)} with respect to $p$ and $(\mathcal{P}^n)_{n \in \N}$. Let $\bS = (S,\S)$ be the associated rough path as defined in \eqref{eq:defn A st}. Then, the bracket $[\bS]$ has finite total variation, and is given by
  \begin{equation*}
    [\bS]_t = \lim_{n \to \infty} \sum_{k=0}^{N_n-1} S_{t^n_k \wedge t,t^n_{k+1} \wedge t} \otimes S_{t^n_k \wedge t,t^n_{k+1} \wedge t},
  \end{equation*}
  where the convergence is uniform in $t \in [0,T]$.
\end{lemma}

\begin{proof}
  The $(i,j)$-component of $[\bS]_t$ is given by
  \begin{equation*}
    [\bS]^{ij}_t = S^i_{0,t} S^j_{0,t} - \S^{ij}_{0,t} - \S^{ji}_{0,t} = S^i_t S^j_t - S^i_0 S^j_0 - \int_0^t S^i_u \dd S^j_u - \int_0^t S^j_u \dd S^i_u.
  \end{equation*}
  The result then follows from Lemmas~4.17 and~4.22 in \cite{Perkowski2016}.
\end{proof}

In view of Lemma~\ref{lem: property of bracket of rough paths}, when assuming Property \textup{(RIE)}, we also refer to the bracket~$[\bS]$ as the \emph{quadratic variation} of~$S$.

\subsection{Rough integrals as limits of Riemann sums}\label{subsec: rough integrals as Riemann sums}

As previously mentioned, the main motivation to introduce Property~\textup{(RIE)} is to obtain the rough integral as a limit of left-point Riemann sums, in order to restore the interpretation of the rough integral as the capital process associated with a financial investment. Indeed, we present the following extension of \cite[Theorem~4.19]{Perkowski2016}, which will be another central tool in our pathwise portfolio theory. The proof of Theorem~\ref{thm: Ito integral for smooth transformed RIE path} is postponed to Appendix~\ref{sec:appendix on RIE}.


\begin{theorem}\label{thm: Ito integral for smooth transformed RIE path}
  Suppose that $S \in C([0,T];\R^d)$ satisfies \textup{(RIE)} with respect to $p$ and $(\mathcal{P}^n)_{n \in \N}$. Let $q \geq p$ such that $2/p + 1/q > 1$. Let $f \in C^{p + \epsilon}$ for some $\epsilon > 0$, so that in particular $(f(S),\D f(S)) \in \crpSq$. Then, for any $(Y,Y') \in \crpSq$, the integral of $(Y,Y')$ against $(f(S),\D f(S))$, as defined in Lemma~\ref{lem: int of controlled paths exists}, is given by
  \begin{equation}\label{eq:general rough integral Riemann sums}
    \int_0^t Y_u \dd f(S)_u = \lim_{n \to \infty} \sum_{k=0}^{N_n-1} Y_{t^n_k} f(S)_{t^n_k \wedge t,t^n_{k+1} \wedge t},
  \end{equation}
  where the convergence is uniform in $t \in [0,T]$.
\end{theorem}

As an immediate consequence of Theorem~\ref{thm: Ito integral for smooth transformed RIE path}, assuming Property \textup{(RIE)}, we note that, for $(Y,Y') \in \crpSq$, the rough integral
\begin{equation}\label{eq:int F dS Riemann sum}
  \int_0^t Y_u \,\d \bS_u = \lim_{n \to \infty} \sum_{k=0}^{N_n-1} Y_{t^n_k} S_{t^n_k \wedge t,t^n_{k+1} \wedge t},
\end{equation}
and indeed the more general rough integral in \eqref{eq:general rough integral Riemann sums}, is independent of the Gubinelli derivative $Y^\prime$. However, in the spirit of F{\"o}llmer's pathwise quadratic variation and integration, the right-hand sides of \eqref{eq:general rough integral Riemann sums} and \eqref{eq:int F dS Riemann sum} do in general depend on the sequence of partitions $(\mathcal{P}^n)_{n \in \N}$.


\section{Pathwise (relative) portfolio wealth processes and master formula}\label{sec:master formula}

In this section we consider pathwise portfolio theory on the rough path foundation presented in Section~\ref{sec:rough integration intro}. In particular, we study the growth of wealth processes relative to the market portfolio, and provide an associated pathwise master formula analogous to that of classical stochastic portfolio theory, cf.~\cite{Fernholz1999,Strong2014,Schied2018}. We start by introducing the basic assumptions on the underlying financial market.

\subsection{The financial market}

Since we want to investigate the long-run behaviour of wealth processes, we consider the price trajectories of $d$ assets on the time interval $[0,\infty)$. As is common in stochastic portfolio theory, we do not include default risk---that is, all prices are assumed to be strictly positive---and we do not distinguish between risk-free and risky assets.

\smallskip

A partition $\mathcal{P}$ of the interval $[0,\infty)$ is a strictly increasing sequence of points $(t_i)_{i \geq 0} \subset [0,\infty)$, with $t_0 = 0$ and such that $t_i \to \infty$ as $i \to \infty$. Given any $T > 0$, we denote by $\mathcal{P}([0,T])$ the restriction of the partition $\mathcal{P} \cup \{T\}$ to the interval $[0,T]$, i.e.~$\cP([0,T]) := (\cP \cup \{T\}) \cap [0,T]$. For a path $S \colon [0,\infty) \to \R^d$, we write $S|_{[0,T]}$ for the restriction of $S$ to $[0,T]$, and we set $\R_+ := (0,\infty)$.

\begin{definition}
  For a fixed $p \in (2,3)$, we say that a path $S \in C([0,\infty);\R_+^d)$ is a \emph{price path}, if there exists a sequence of partitions $(\cP_S^n)_{n \in \N}$ of the interval $[0,\infty)$, with vanishing mesh size on compacts, such that, for all $T > 0$, the restriction $S|_{[0,T]}$ satisfies \textup{(RIE)} with respect to $p$ and $(\cP_S^n([0,T]))_{n \in \N}$.

  We denote the family of all such price paths by $\Omega_p$.
\end{definition}



It seems to be natural to allow the partitions $(\mathcal{P}^n_S)_{n \in \N}$ to depend on the price path $S$, since partitions are typically given via stopping times in stochastic frameworks.

\smallskip

Throughout the remainder of the paper, we adopt the following assumption on the regularity parameters.

\begin{assumption}\label{ass: parameters}
  Let $p \in (2,3)$, $q \geq p$ and $r > 1$ be given such that
  \begin{equation*}
    \frac{2}{p} + \frac{1}{q} > 1 \qquad \text{and} \qquad \frac{1}{r} = \frac{1}{p} + \frac{1}{q}.
  \end{equation*}
\end{assumption}

In particular, we note that $1 < p/2 \leq r < p \leq q < \infty$.

\smallskip

By Property (RIE), we can (and do) associate to every price path $S \in \Omega_p$ the $p$-rough path $\bS = (S,\S)$, as defined in~\eqref{eq:defn A st}. We can then define the \emph{market covariance} as the matrix $a = [a^{ij}]_{1 \leq i, j \leq d}$, with $(i,j)$-component given by the measure
\begin{equation}\label{eq:defn covariance matrix a}
  a^{ij}(\d s) := \frac{1}{S^i_s S^j_s} \,\d [\bS]^{ij}_s.
\end{equation}
Although we do not work in a probabilistic setting and thus should not, strictly speaking, talk about covariance in the probabilistic sense, the relation \eqref{eq:defn covariance matrix a} is consistent with classical stochastic portfolio theory (with the bracket process replaced by the quadratic variation), and it turns out to still be a useful quantity in pathwise frameworks, cf.~\cite{Schied2016,Schied2018}.

\subsection{Pathwise portfolio wealth processes}

We now introduce admissible portfolios and the corresponding wealth processes on the market defined above. To this end, we first fix the notation:
\begin{align*}
  \Delta^d := \bigg\{x = (x^1, \ldots, x^d) \in \R^d \, : \, \sum_{i=1}^d x^i = 1\bigg\},
\end{align*}
$\Delta^d_+ := \{x \in \Delta^d : x^i > 0 \ \ \forall i = 1, \ldots, d\}$ and $\overline{\Delta}^d_+ := \{x \in \Delta^d : x^i \geq 0 \ \ \forall i = 1, \ldots, d\}$.

\begin{definition}\label{def: admissible strategies}
  We say that a path $F \colon [0,\infty) \to \R^d$ is an \emph{admissible strategy} if, for every $T > 0$, there exists a path $F' \colon [0,T] \to \mathcal{L}(\R^d;\R^d)$ such that $(F|_{[0,T]},F') \in \crpSq$ is a controlled path with respect to $S$ (in the sense of Definition~\ref{def: controlled path}). We say that an admissible strategy $\pi$ is a \emph{portfolio} for $S$ if additionally $\pi_t \in \Delta^d$ for all $t \in [0,\infty)$.
\end{definition}

\begin{remark}
  As explained in \cite[Remark~4.7]{Friz2020}, if $S$ is sufficiently regular then, given an admissible strategy $F$, there could exist multiple different Gubinelli derivatives $F'$ such that the pair $(F,F')$ defines a valid controlled path with respect to $S$. However, thanks to Property \textup{(RIE)}, Theorem~\ref{thm: Ito integral for smooth transformed RIE path} shows that the rough integral $\int F \,\d \bS$ can be expressed as a limit of Riemann sums which only involve $F$ and $S$, and, therefore, is independent of the choice of $F'$. Thus, the choice of the Gubinelli derivative $F'$ is unimportant, provided that at least one exists. Indeed, one could define an equivalence relation $\sim$ on $\crpSq$ such that $(F,F') \sim (G,G')$ if $F = G$, and define the family of admissible strategies as elements of the quotient space $\crpSq/\sim$. By a slight abuse of notation, we shall therefore sometimes write simply $F \in \crpSq$ instead of $(F, F') \in \crpSq$.
\end{remark}

\begin{remark}\label{remark: functionally generalized portfolio}
  While the admissible class of portfolios introduced in Definition~\ref{def: admissible strategies} allows for a pathwise (model-free) analysis (without notions like filtration or predictability), it also covers the most frequently applied classes of functionally generated portfolios---see \cite{Fernholz1999}---and their generalizations as considered in e.g.~\cite{Strong2014} and \cite{Schied2018}. Indeed, every path-dependent functionally generated portfolio which is sufficiently smooth in the sense of Dupire~\cite{Dupire2019} (see also \cite{Cont2010}), is a controlled path and thus an admissible strategy, as shown in \cite{Ananova2020}.

  In the present work we will principally focus on ``adapted'' strategies $F$, in the sense that $F$ is a controlled path, as in Definition~\ref{def: admissible strategies}, with $F_t$ being a measurable function of $S|_{[0,t]}$ for each $t \in [0,\infty)$. In other words, if $S$ is modelled by a stochastic process then we require $F$ to be adapted to the natural filtration generated by $S$. Clearly, such adapted admissible strategies are reasonable choices in the context of mathematical finance.
\end{remark}

A portfolio $\pi = (\pi^1,\dots,\pi^d)$ represents the ratio of the investor's wealth invested into each of the $d$ assets. As is usual, we normalize the initial wealth to be $1$, since in the following we will only be concerned with the long-run growth. Suppose $S \in \Omega_p$ with corresponding sequence $(\mathcal{P}^n_S)_{n \in \N}$ of partitions. If we restrict the rebalancing according to the portfolio $\pi$ to the discrete times given by $\mathcal{P}^n_S = (t^n_j)_{j \in \N}$, then the corresponding wealth process $W^n$ satisfies
\begin{equation*}
  W^n_t = 1 + \sum_{j=1}^\infty \frac{\pi_{t_j}W^n_{t_j}}{S_{t_j}} S_{t_j \wedge t,t_{j+1} \wedge t} = 1 + \sum_{j=1}^\infty \sum_{i=1}^d \frac{\pi^i_{t_j}W^n_{t_j}}{S^i_{t_j}} S^i_{t_j \wedge t,t_{j+1} \wedge t}
\end{equation*}
with $t_j \wedge t := \min \{t_j,t\}$. Taking the limit to continuous-time (i.e.~$n \to \infty$) and keeping Property (RIE) in mind, we observe that the wealth process $W^\pi$ associated to the portfolio $\pi$ should satisfy
\begin{equation}\label{eq:wealth process}
  W^\pi_t = 1 + \int_0^t \frac{\pi_s W^\pi_s}{S_s} \dd \bS_s, \qquad t \in [0,\infty).
\end{equation}
Analogously to (classical) stochastic portfolio theory (e.g.~\cite{Karatzas2007} or \cite{Schied2018}), the wealth process associated to a portfolio may be expressed as a (rough) exponential.

\begin{lemma}\label{lem: rough exponential gives self financing integrands}
  Let $\pi$ be a portfolio for $S \in \Omega_p$. Then the wealth process $W^\pi$ (with unit initial wealth), given by
  \begin{equation*}
    W^\pi_t := \exp \bigg(\int_0^t \frac{\pi_s}{S_s} \,\d \bS_s - \frac{1}{2} \sum_{i,j=1}^d \int_0^t \frac{\pi^i_s \pi^j_s}{S^i_s S^j_s} \,\d [\bS]^{ij}_s\bigg), \qquad t \in [0,\infty),
  \end{equation*}
  satisfies \eqref{eq:wealth process}, where $\int_0^t \frac{\pi_s}{S_s} \,\d \bS_s$ is the rough integral of the controlled path $\pi/S$ with respect to rough path $\bS$, and $\int_0^t \frac{\pi^i_s \pi^j_s}{S^i_s S^j_s} \,\d [\bS]^{ij}_s$ is the usual Riemann--Stieltjes integral with respect to the $(i,j)$-component of the (finite variation) bracket $[\bS]$.
\end{lemma}

\begin{proof}
  Note that, since $1/S = f(S)$ with the smooth function $f(x) = (1/x^1, \ldots,1/x^d)$ on $\R^d_{+}$, the pair $(1/S,\D f(S)) \in \crpSp \subset \crpSq$ is a controlled path. Therefore, for each portfolio $\pi \in \crpSq$, we can define the quotient $\pi/S = (\pi^1/S^1, \ldots, \pi^d/S^d)$, which gives an element $(\pi/S,(\pi/S)^\prime)$ in $\crpSq$; see Lemma~\ref{lem: product of controlled rough paths}.

  Setting $Z := \int_0^\cdot \frac{\pi_s}{S_s} \,\d \bS_s$, by Lemma~\ref{lem: Ito isometry for rough paths}, we have that
  \begin{equation*}
    [\bZ] = \int_0^\cdot \bigg(\frac{\pi_s}{S_s} \otimes \frac{\pi_s}{S_s}\bigg) \,\d [\bS]_s = \sum_{i,j=1}^d \int_0^\cdot \frac{\pi^i_s \pi^j_s}{S^i_s S^j_s} \,\d [\bS]^{ij}_s,
  \end{equation*}
  where $\bZ$ is the canonical rough path lift of $Z$ (see Section~\ref{sec Rough paths associated to controlled paths}). We then have that $W^\pi_t = \exp(Z_t - \frac{1}{2}[\bZ]_t)$, so that, by Lemma~\ref{lem: dynamics of rough exponential}, $W^\pi$ satisfies
  \begin{equation*}
    W^\pi_t = 1 + \int_0^t W^\pi_s \,\d \bZ_s, \qquad t \in [0,\infty).
  \end{equation*}
  By Lemma~\ref{lem: consistency of rough integrals} and Proposition~\ref{prop: associativity of rough integration} it then follows that $W^\pi$ satisfies~\eqref{eq:wealth process}.
\end{proof}

\begin{remark}
  Every portfolio $\pi$ can be associated to a self-financing admissible strategy $\xi$ by setting $\xi^i_t := \pi^i_t W^\pi_t/S^i_t$ for $i = 1,\ldots,d$. Indeed, we have that $W^\pi_t = \sum_{i=1}^d \xi^i_t S^i_t$, and that
  \begin{equation*}
    W^\pi_t = 1 + \int_0^t \frac{\pi_s W^\pi_s}{S_s} \,\d \bS_s = 1 + \int_0^t \xi_s \,\d \bS_s, \qquad t \in [0,\infty),
  \end{equation*}
  so that $\xi$ is self-financing.
\end{remark}

As in the classical setup of stochastic portfolio theory (e.g.~\cite{Fernholz2002}) we introduce the market portfolio as a reference portfolio.

\begin{lemma}\label{lem: market portfolio}
  The path $\mu \colon [0,\infty) \to \Delta^d_+$, defined by $\mu^i_t := \frac{S^i_t}{S^1_t + \cdots + S^d_t}$ for $i = 1, \ldots, d$, is a portfolio for $S \in \Omega_p$, called the \emph{market portfolio} (or \emph{market weights process}). The corresponding wealth process (with initial wealth $1$) is given by
  \begin{equation*}
    W^\mu_t = \frac{S^1_t + \cdots + S^d_t}{S^1_0 + \cdots + S^d_0}.
  \end{equation*}
\end{lemma}

\begin{proof}
  Since $\mu$ is a smooth function of $S$, it is a controlled path with respect to $S$, and is therefore an admissible strategy. Since $\mu^1_t + \cdots + \mu^d_t = 1$, we see that $\mu$ is indeed a portfolio.

  Let $f(x) := \log (x^1 + \cdots + x^d)$ for $x \in \R^d_+$. By the It{\^o} formula for rough paths (Proposition~\ref{prop: general Ito formula for rough paths}), it follows that
  \begin{align*}
    f(S_t) - f(S_0) &= \int_0^t \bigg(\frac{1}{S^1_s + \cdots + S^d_s}, \, \ldots, \, \frac{1}{S^1_s + \cdots + S^d_s}\bigg) \d \bS_s - \frac{1}{2} \int_0^t \bigg(\frac{\mu_s}{S_s} \otimes \frac{\mu_s}{S_s}\bigg) \,\d [\bS]_s\\
    &= \int_0^t \frac{\mu_s}{S_s} \,\d \bS_s - \frac{1}{2} \sum_{i,j=1}^d \int_0^t \frac{\mu^i_s \mu^j_s}{S^i_s S^j_s} \,\d[\bS]^{ij}_s,
  \end{align*}
  where we used the fact that $\frac{\mu^i_s}{S^i_s} = \frac{1}{S^1_s + \cdots + S^d_s}$. By Lemma~\ref{lem: rough exponential gives self financing integrands}, the right-hand side is equal to $\log W^\mu_t$, so that
  \begin{equation*}
    W^\mu_t = \exp\big(f(S_t) - f(S_0)\big) = \frac{S^1_t + \cdots + S^d_t}{S^1_0 + \cdots + S^d_0}.
  \end{equation*}
\end{proof}

\subsection{Formulae for the growth of wealth processes}

In this subsection we derive pathwise versions of classical formulae of stochastic portfolio theory---see \cite{Fernholz1999}---which describe the dynamics of the relative wealth of a portfolio with respect to the market portfolio; cf.~\cite{Schied2018} for analogous results relying on F{\"o}llmer's pathwise integration.

Given a portfolio $\pi$, we define the relative covariance of $\pi$ by $\tau^\pi = [\tau^\pi_{ij}]_{1 \leq i, j \leq d}$, where
\begin{equation}\label{eq:defn tau^pi_ij}
  \tau^{\pi}_{ij}(\d s) := (\pi_s - e_i)^\top a(\d s)(\pi_s - e_j),
\end{equation}
where $(e_i)_{1 \leq i \leq d}$ denotes the canonical basis of $\R^d$, and we recall $a(\d s)$ as defined in \eqref{eq:defn covariance matrix a}.

Henceforth, we will write
\begin{equation}\label{eq:defn relative wealth process}
  V^\pi := \frac{W^\pi}{W^\mu}
\end{equation}
for the relative wealth of a portfolio $\pi$ with respect to the market portfolio $\mu$.

\begin{proposition}\label{prop: log of relative wealth}
  Let $\pi$ be a portfolio for $S \in \Omega_p$, and let $\mu$ be the market portfolio as above. We then have that
  \begin{equation}\label{eq: formula for log of relative wealth}
    \log V^\pi_t =  \int_0^t \frac{\pi_s}{\mu_s} \,\d \mu_s - \frac{1}{2} \sum_{i,j=1}^d \int_0^t \pi^i_s \pi^j_s \tau^{\mu}_{ij}(\d s), \qquad t \in [0,\infty).
  \end{equation}
\end{proposition}

\begin{remark}\label{remark: about how to express Ito integral wrt mu}
  The integral $\int_0^t \frac{\pi_s}{\mu_s} \,\d \mu_s$ appearing in \eqref{eq: formula for log of relative wealth} is interpreted as the rough integral of the $S$-controlled path $\pi/\mu$ against the $S$-controlled path $\mu$ in the sense of Lemma~\ref{lem: int of controlled paths exists}. By Theorem~\ref{thm: Ito integral for smooth transformed RIE path}, the integral $\int_0^t \frac{\pi_s}{\mu_s} \,\d \mu_s$ can also be expressed as a limit of left-point Riemann sums, which justifies the financial meaning of~\eqref{eq: formula for log of relative wealth}.
\end{remark}

\begin{proof}[Proof of Proposition~\ref{prop: log of relative wealth}]
  \textit{Step~1.} By the It{\^o} formula for rough paths (Proposition~\ref{prop: general Ito formula for rough paths}), with the usual notational convention $\log x = \sum_{i=1}^d \log x^i$, we have
  \begin{equation*}
    \log S_t = \log S_0 + \int_0^t \frac{1}{S_s} \,\d \bS_s - \frac{1}{2} \sum_{i=1}^d \int_0^t \frac{1}{(S^i_s)^2} \,\d [\bS]^{ii}_s, \qquad t \in [0,\infty).
  \end{equation*}
  Since $\pi$ and $\log S$ are $S$-controlled paths, we can define the integral of $\pi$ against $\log S$ in the sense of Lemma~\ref{lem: int of controlled paths exists}. By the associativity of rough integration (Proposition~\ref{prop: associativity of rough integration}), we have
  \begin{equation*}
    \int_0^t \pi_s \,\d \log S_s = \int_0^t \frac{\pi_s}{S_s} \,\d \bS_s  - \frac{1}{2} \sum_{i=1}^d \int_0^t \frac{\pi^i_s}{(S^i_s)^2} \,\d [\bS]^{ii}_s.
  \end{equation*}
  It is convenient to introduce the excess growth rate of the portfolio~$\pi$, given by
  \begin{equation*}
    \gamma_\pi^\ast(\d s) := \frac{1}{2} \bigg(\sum_{i=1}^d \pi^i_s a^{ii}(\d s) - \sum_{i,j=1}^d \pi^i_s \pi^j_s a^{ij}(\d s)\bigg).
  \end{equation*}
  By Lemma~\ref{lem: rough exponential gives self financing integrands}, we have that
  \begin{equation}\label{eq:log V^pi_t espressions}
    \log W^\pi_t = \int_0^t \frac{\pi_s}{S_s} \,\d \bS_s - \frac{1}{2} \sum_{i,j=1}^d \int_0^t \pi^i_s \pi^j_s a^{ij}(\d s) = \int_0^t \pi_s \,\d \log S_s + \gamma_\pi^\ast([0,t]).
  \end{equation}
  In particular, this implies that
  \begin{equation}\label{eq: expression of log relative wealth as rough integral}
    \log V^\pi_t = \int_0^t (\pi_s - \mu_s) \,\d \log S_s + \gamma_\pi^\ast([0,t]) - \gamma_\mu^\ast([0,t]).
  \end{equation}

  \textit{Step~2.} By Lemma~\ref{lem: market portfolio} and \eqref{eq:log V^pi_t espressions}, we have
  \begin{align}
    \log \mu^i_t &= \log \mu^i_0 + \log S^i_t - \log S^i_0 - \log W^\mu_t\nonumber\\
    &= \log \mu^i_0 + \log S^i_t - \log S^i_0 - \int_0^t \mu_s \,\d \log S_s - \gamma_\mu^\ast([0,t])\nonumber\\
    &= \log \mu^i_0 + \int_0^t (e_i - \mu_s) \,\d \log S_s - \gamma_\mu^\ast([0,t]).\label{eq:log mu^i expr}
  \end{align}
  By part (ii) of Proposition~\ref{prop: property of bracket of rough paths} and Lemma~\ref{lem: Ito isometry for rough paths}, we deduce that
  \begin{equation}\label{eq: [log mu] = tau^mu}
    [\log S]_t = a([0,t]), \qquad \text{and} \qquad [\log \mu]_t = \tau^\mu([0,t]).
  \end{equation}
  Applying the It{\^o} formula for rough paths (Proposition~\ref{prop: general Ito formula for rough paths}) to $\exp(\log \mu^i)$, using the associativity of rough integration (Proposition~\ref{prop: associativity of rough integration}), and recalling \eqref{eq:log mu^i expr}, we have
  \begin{equation*}
    \int_0^t \frac{\pi^i_s}{\mu^i_s} \,\d \mu^i_s = \int_0^t \pi^i_s (e_i - \mu_s) \,\d \log S_s - \int_0^t \pi^i_s \,\d \gamma_\mu^\ast(\d s) + \frac{1}{2} \int_0^t \pi^i_s \,\d [\log \mu]^{ii}_s.
  \end{equation*}
  Using \eqref{eq: [log mu] = tau^mu} and summing over $i = 1, \ldots, d$, we obtain
  \begin{equation}\label{eq: int pi/mu d mu expression}
    \int_0^t \frac{\pi_s}{\mu_s} \,\d \mu_s = \int_0^t (\pi_s - \mu_s) \,\d \log S_s - \gamma_\mu^\ast([0,t]) + \frac{1}{2} \sum_{i=1}^d \int_0^t \pi^i_s \tau^\mu_{ii}(\d s).
  \end{equation}

  \textit{Step~3.} Taking the difference of \eqref{eq: expression of log relative wealth as rough integral} and \eqref{eq: int pi/mu d mu expression}, we have
  \begin{equation*}
    \log V^\pi_t = \int_0^t \frac{\pi_s}{\mu_s} \,\d \mu_s + \gamma_\pi^\ast([0,t]) - \frac{1}{2} \sum_{i=1}^d \int_0^t \pi^i_s \tau^\mu_{ii}(\d s).
  \end{equation*}
  It remains to note that
  \begin{equation*}
    \gamma_\pi^\ast([0,t]) = \frac{1}{2} \bigg(\sum_{i=1}^d \int_0^t \pi^i_s \tau^\mu_{ii}(\d s) - \sum_{i,j=1}^d \int_0^t \pi^i_s \pi^j_s \tau^{\mu}_{ij}(\d s)\bigg),
  \end{equation*}
  which follows from a straightforward calculation; see e.g.~\cite[Lemma~1.3.4]{Fernholz2002}.
\end{proof}

While Definition~\ref{def: admissible strategies} allows for rather general portfolios, so-called functionally generated portfolios are the most frequently considered ones in SPT. In a pathwise setting such portfolios and the corresponding master formula were studied previously in \cite{Schied2018} and \cite{Cuchiero2019}. We conclude this section by deriving such a master formula for functionally generated portfolios in the present (rough) pathwise setting.

\smallskip

Let $G$ be a strictly positive function in $C^{p + \epsilon}(\Delta^d_+;\R_+)$ for some $\epsilon > 0$. One can verify that $\nabla \log G(\mu) \in \mathcal{V}^q_\mu$ is a $\mu$-controlled path for a suitable choice of $q$ (see Example~\ref{ex: controlled path 1+epsilon}), and is therefore also an $S$-controlled path by Lemma~\ref{lem: consistency of rough integrals}. Since the product of controlled paths is itself a controlled path (by Lemma~\ref{lem: product of controlled rough paths}), we see that the path $\pi$ defined by
\begin{equation}\label{eq: canonical function generated portfolio}
  \pi^i_t := \mu^i_t\bigg(\frac{\partial}{\partial x_i} \log G(\mu_t) + 1 - \sum_{k=1}^d \mu^k_t \frac{\partial}{\partial x_k} \log G(\mu_t)\bigg), \qquad t \in [0,\infty), \ \ i = 1, \ldots, d,
\end{equation}
is a $\mu$-controlled (and hence also an $S$-controlled) path, and is indeed a portfolio for $S \in \Omega_p$. The function $G$ is called a \emph{portfolio generating function}, and we say that $G$ \emph{generates} $\pi$.

\begin{theorem}[The master formula]\label{thm: master formula}
  Let $G \in C^{p + \epsilon}(\Delta^d_+;\R_+)$ for some $\epsilon > 0$ be a portfolio generating function, and let $\pi$ be the portfolio generated by $G$. The wealth of $\pi$ relative to the market portfolio is given by
  \begin{equation*}
    \log V^\pi_t = \log \bigg(\frac{G(\mu_t)}{G(\mu_0)}\bigg) - \frac{1}{2} \sum_{i,j=1}^d \int_0^t \frac{1}{G(\mu_s)} \frac{\partial^2 G(\mu_s)}{\partial x_i \partial x_j} \mu^i_s \mu^j_s \tau^\mu_{ij}(\d s), \qquad t \in [0,\infty).
  \end{equation*}
\end{theorem}

\begin{proof}
  Let $g = \nabla \log G(\mu)$, so that $g^i = \frac{\partial}{\partial x_i} \log G(\mu) = \frac{1}{G(\mu)} \frac{\partial G}{\partial x_i}(\mu)$ for each $i = 1, \ldots, d$. We can then rewrite $\eqref{eq: canonical function generated portfolio}$ as
  \begin{equation}\label{eq:pi^i in terms of mu^i and g^i}
    \pi^i = \mu^i \bigg(g^i + 1 - \sum_{k=1}^d \mu^k g^k\bigg),
  \end{equation}
  so that $\pi^i/\mu^i = g^i + 1 - \sum_{k=1}^d \mu^k g^k$. Since $\sum_{i=1}^d \mu^i_s = 1$ for all $s \geq 0$, we must have that $\sum_{i=1}^d \mu^i_{s,t} = 0$ for all $s < t$. Thus
  \begin{equation*}
    \int_0^t \frac{\pi_s}{\mu_s} \,\d \mu_s = \lim_{n \to \infty} \sum_{k=0}^{N_n-1} \sum_{i=1}^d \frac{\pi^i_{t^n_k}}{\mu^i_{t^n_k}} \mu^i_{t^n_k \wedge t,t^n_{k+1} \wedge t} = \lim_{n \to \infty} \sum_{k=0}^{N_n-1} \sum_{i=1}^d g^i_{t^n_k} \mu^i_{t^n_k \wedge t,t^n_{k+1} \wedge t} = \int_0^t g_s \,\d \mu_s.
  \end{equation*}
  We have from~\eqref{eq:defn tau^pi_ij} that $\sum_{j=1}^d \mu^j_s \tau^\mu_{ij}(\d s) = (\mu_s - e_i)^\top a(\d s) (\mu_s - \mu_s) = 0$. It follows from this and \eqref{eq:pi^i in terms of mu^i and g^i} that
  \begin{equation}\label{eq:sum pi^i pi^j tau^mu = sum g^i g^j mu^i mu^j tau^mu}
    \sum_{i,j=1}^d \pi^i_s \pi^j_s \tau^\mu_{ij}(\d s) = \sum_{i,j=1}^d g^i_s g^j_s \mu^i_s \mu^j_s \tau^\mu_{ij}(\d s).
  \end{equation}

  Recall from~\eqref{eq: [log mu] = tau^mu} that $[\log \mu]_t = \tau^\mu([0,t])$. By applying the It{\^o} formula for rough paths (Proposition~\ref{prop: general Ito formula for rough paths}) to $\mu^i = \exp(\log \mu^i)$, we see that the path $t \mapsto \mu^i_t - \int_0^t \mu^i_s \,\d \log \mu^i_s$ is of finite variation. By part (ii) of Proposition~\ref{prop: property of bracket of rough paths} and Lemma~\ref{lem: Ito isometry for rough paths}, we therefore have that
  \begin{equation}\label{eq:quad var of mu}
    [\mu]^{ij}_t = \int_0^t \mu^i_s \mu^j_s \,\d [\log \mu]^{ij}_s = \int_0^t \mu^i_s \mu^j_s \tau^\mu_{ij}(\d s).
  \end{equation}
  By the It{\^o} formula for rough paths (Proposition~\ref{prop: general Ito formula for rough paths}), we then have
  \begin{align*}  
    \log \bigg(\frac{G(\mu_t)}{G(\mu_0)}\bigg) &= \int_0^t g_s \,\d \mu_s + \frac{1}{2} \sum_{i,j=1}^d \int_0^t \bigg(\frac{1}{G(\mu_s)} \frac{\partial^2 G(\mu_s)}{\partial x_i \partial x_j} - g^i_s g^j_s\bigg) \,\d [\mu]^{ij}_s\\
    &= \int_0^t \frac{\pi_s}{\mu_s} \,\d \mu_s + \frac{1}{2} \sum_{i,j=1}^d \int_0^t \bigg(\frac{1}{G(\mu_s)} \frac{\partial^2 G(\mu_s)}{\partial x_i \partial x_j} - g^i_s g^j_s\bigg)\mu^i_s \mu^j_s \tau^\mu_{ij}(\d s).
  \end{align*}
  Combining this with~\eqref{eq: formula for log of relative wealth} and \eqref{eq:sum pi^i pi^j tau^mu = sum g^i g^j mu^i mu^j tau^mu}, we deduce the result.
\end{proof}

\section{Cover's universal portfolios and their optimality}\label{sec:Cover portfolio}

Like stochastic portfolio theory, Cover's universal portfolios \cite{Cover1991} aim to give general recipes to construct preference-free asymptotically ``optimal'' portfolios; see also \cite{Jamshidian1992} and \cite{Cover1996}. A first link between SPT and these universal portfolios was established in a pathwise framework based on F{\"o}llmer integration in \cite{Cuchiero2019} (see also \cite{Wong2015}). In this section we shall generalize the pathwise theory regarding Cover's universal portfolios developed in \cite{Cuchiero2019} to the present rough path setting.

Cover's universal portfolio is based on the idea of trading according to a portfolio which is defined as the average over a family $\mathcal{A}$ of admissible portfolios. In the spirit of \cite{Cuchiero2019}, we introduce pathwise versions of Cover's universal portfolios---that is, portfolios of the form
\begin{equation*}
  \pi^\nu_t := \frac{\int_{\mathcal{A}} \pi_t V^\pi_t \dd \nu(\pi)}{\int_{\mathcal{A}} V^\pi_t \dd \nu(\pi)}, \qquad t \in [0,\infty),
\end{equation*}
where $\nu$ is a given probability measure on $\mathcal{A}$. In order to find suitable classes $\mathcal{A}$ of admissible portfolios, we recall Assumption~\ref{ass: parameters} and make the following standing assumption throughout the entire section.

\begin{assumption}
  We fix $q' > q$ and $r' > r$ such that $\frac{2}{p} + \frac{1}{q'} > 1 $ and $\frac{1}{r'} = \frac{1}{p} + \frac{1}{q'}$.
\end{assumption}

\subsection{Admissible portfolios}\label{subsec:base set of universal protfolio}

As a first step to construct Cover's universal portfolios in our rough path setting, we need to find a suitable set of admissible portfolios. To this end, we set
\begin{equation*}
  \mathcal{V}_\mu^{q}([0,\infty);\Delta^d) := \Big\{(\pi,\pi^\prime) : \forall \, T > 0, \, (\pi,\pi^\prime)|_{[0,T]} \in \mathcal{V}_\mu^{q}([0,T];\Delta^d)\Big\}.
\end{equation*}
Then, for some fixed control function $c_\mu$ which controls the $p$-variation norm of the market portfolio $\mu$, and for some $M > 0$, we introduce a class of \emph{admissible portfolios} as the set
\begin{equation}\label{eq: configuration space for mu}
  \mathcal{A}^{M,q}(c_\mu) := \left\{(\pi,\pi') \in \mathcal{V}_\mu^{q}([0,\infty);\Delta^d) \,:\,
  \begin{array}{l}
  \Big|\frac{\pi_0}{\mu_0}\Big| + \Big|\Big(\frac{\pi}{\mu}\Big)^\prime_0\Big| \leq M,\\
  \sup_{s \leq t} \frac{|(\frac{\pi}{\mu})^\prime_{s,t}|^q}{c_\mu(s,t)} + \sup_{s \leq t} \frac{|R^{\frac{\pi}{\mu}}_{s,t}|^r}{c_\mu(s,t)} \leq 1
  \end{array}
  \right\}.
\end{equation}
Here $(\pi/\mu, (\pi/\mu)^\prime)$ denotes the product of the two $\mu$-controlled paths $(\pi, \pi')$ and $(\frac{1}{\mu},(\frac{1}{\mu})')$ (see Lemma~\ref{lem: product of controlled rough paths}). In particular, $(\pi/\mu)^\prime = \pi^\prime/\mu + \pi(1/\mu)^\prime$, and $R^{\frac{\pi}{\mu}}$ is the remainder of the controlled rough path $\pi/\mu$.

\begin{remark}
  We consider here controlled paths with respect to $\mu$, instead of with respect to $S$. As noted in Remark~\ref{remark: about how to express Ito integral wrt mu}, every $S$-controlled path $(\pi,\pi^\prime) \in \mathcal{V}_S^q$ can be used to define the integral $\int \frac{\pi_t}{\mu_t}\dd\mu_t$, and all the results in this section can also be established based on $\mathcal{V}_S^q$ with appropriate modifications. We choose to consider $(\pi,\pi^\prime) \in \mathcal{V}_\mu^q$ as a $\mu$-controlled path in order to slightly simplify the notation. It is straightforward to check that $\mathcal{V}_\mu^q \subseteq \mathcal{V}_S^q$.
\end{remark}

Let us recall from Definition~\ref{def: controlled path} that, for any $T > 0$,
\begin{equation*}
  \|(Y,Y^\prime)\|_{\mathcal{V}^q_\mu,[0,T]} = |Y_0| + |Y^\prime_0| + \|Y^\prime\|_{q,[0,T]} + \|R^Y\|_{r,[0,T]}
\end{equation*}
defines a complete norm on $\mathcal{V}_\mu^{q}([0,T];\Delta^d)$. We endow $\mathcal{A}^{M,q}(c_\mu) \subset \mathcal{V}_\mu^{q^\prime}([0,\infty);\Delta^d)$ with the seminorms
\begin{equation}\label{eq: seminorms}
  p_T^{\mu,q^\prime}((\pi,\pi^\prime)) := \Big\|\frac{\pi}{\mu}, \Big(\frac{\pi}{\mu}\Big)^\prime\Big\|_{\mathcal{V}_\mu^{q^\prime},[0,T]}, \qquad T > 0.
\end{equation}
The reason for taking $q^\prime > q$ is that it will allow us to obtain a compact embedding of $\mathcal{A}^{M,q}(c_\mu)$ into $\mathcal{V}^{q^\prime}_{\mu}$. This compactness of the set of admissible portfolios plays a crucial role in obtaining optimality of universal portfolios.

\smallskip

Let us discuss some examples of admissible portfolios. We first check that the functionally generated portfolios treated in \cite{Cuchiero2019} belong to $\mathcal{A}^{M,q}(c_\mu)$ provided that the control function $c_\mu$ is chosen appropriately. Recall that $C^k(\overline{\Delta}^d_+;\R_+)$ denotes the space of $k$-times continuously differentiable $\R_+$-valued functions on the closed (non-negative) simplex $\overline \Delta^d_+$, and that $\|G\|_{C^k} := \max_{0 \leq n \leq k} \|\D^nG\|_{\infty}$.

\begin{lemma}\label{lem: functionally generated portfolios are in base set}
  Let $K > 0$ be a constant, and let
  \begin{equation*}
    \mathcal{G}^K = \Big\{G \in C^3(\overline{\Delta}^d_+;\R_+) : \|G\|_{C^3} \leq K, \, G \geq \frac{1}{K}\Big\}.
  \end{equation*}
  Then the portfolio $\pi$ generated by $G$, as defined in \eqref{eq: canonical function generated portfolio}, belongs to $\mathcal{A}^{M,p}(c_\mu)$ for a suitable control function $c_\mu$ and constant $M$. More precisely, there exists a control function of the form $c_\mu(\hspace{1pt}\cdot\hspace{1pt}, \hspace{1pt}\cdot\hspace{1pt}) = C \|\mu\|_{p,[\hspace{1pt}\cdot\hspace{1pt}, \hspace{1pt}\cdot\hspace{1pt}]}^p$ and a constant $M > 0$, such that $C$ and $M$ only depend on $K$, and
  \begin{equation*}
    \Big\{(\pi^G,(\pi^G)^\prime) : \pi^G \text{ defined in \eqref{eq: canonical function generated portfolio} for some } G \in \mathcal{G}^K\Big\} \subset \mathcal{A}^{M,p}(c_\mu).
  \end{equation*}
  Note that here we take $q = p$ and $r = p/2$.
\end{lemma}

\begin{proof}
  Fix $G \in \mathcal{G}^K$, and let $\pi$ be the associated portfolio as defined in \eqref{eq: canonical function generated portfolio}. Since $\pi$ is defined as a $C^2$ function of $\mu$, we know immediately that it is a $\mu$-controlled path.

  A simple calculation shows that
  \begin{equation*}
    \frac{\pi_t}{\mu_t} = g_t + (1 - \mu_t \cdot g_t)\1,
  \end{equation*}
  where we write $\1 = (1, \ldots, 1)$ and $g_t = \nabla \log G(\mu_t)$, and we use $\cdot$ to denote the standard inner product on $\R^d$. The pair $(\1,0)$ is trivially a $\mu$-controlled path with $\1' = 0$ and $R^{\1} = 0$, and thus clearly satisfies the required bounds in \eqref{eq: configuration space for mu} with an arbitrary control function. It thus suffices to show that $(g,g^\prime)$ and $(\mu \cdot g, (\mu \cdot g)^\prime)$ satisfy the required bounds with control functions $c_\mu^1$ and $c_\mu^2$ respectively, since then $c_\mu := c_\mu^1 + c_\mu^2$ gives the desired control function.

  We begin with $(g,g')$. Let $F := \nabla \log G$, so that $g = F(\mu)$ and $g^\prime = \D F(\mu)$. By Taylor expansion, we can verify that, for all $s \leq t$,
  \begin{equation}\label{eq:bounds on (g,g')}
    |g_{s,t}| \leq \|\D F\|_\infty |\mu_{s,t}|, \qquad |g^\prime_{s,t}| \leq \|\D^2F\|_{\infty} |\mu_{s,t}|, \qquad |R^g_{s,t}| \leq \|\D^2F\|_{\infty} |\mu_{s,t}|^2.
  \end{equation}
  Note that $F$, $\D F$ and $\D^2F$ only depend on $\D G$, $\D^2G$, $\D^3G$ and $1/G$, and therefore, since $\|G\|_{C^3} \leq K$ and $G \geq 1/K$, there exists a constant $C = C(K)$, which only depends on $K$, such that $\|F\|_\infty \leq C$, $\|\D F\|_\infty \leq C$ and $\|\D^2F\|_{\infty} \leq C$. It follows that we can choose $c^1_\mu(s,t) = C \|\mu\|_{p,[s,t]}^p$. Note also that $\|g\|_{\infty} \leq C$ and $\|g'\|_{\infty} \leq C$.

  We now turn to $(\mu \cdot g, (\mu \cdot g)^\prime)$. Noting that $\mu$ is trivially a $\mu$-controlled path with $\mu' = 1$ and $R^\mu = 0$, and that $R^{\mu \cdot g}_{s,t} = \mu_s \cdot R^g_{s,t} + \mu_{s,t} \cdot g_{s,t}$, we deduce that
  \begin{equation*}
    |(\mu \cdot g)^\prime_{s,t}| \leq |g_{s,t}| + \|\mu\|_\infty |g'_{s,t}| + \|g'\|_\infty |\mu_{s,t}|, \qquad |R^{\mu \cdot g}_{s,t}| \leq \|\mu\|_{\infty} |R^g_{s,t}| + |\mu_{s,t}||g_{s,t}|.
  \end{equation*}
  Since $\mu_t$ takes values in the bounded set $\Delta^d_+$, we can use the bounds in \eqref{eq:bounds on (g,g')} to show that there exists a constant $L = L(K)$, depending only on $K$, such that $|(\mu \cdot g)^\prime_{s,t}| \leq L |\mu_{s,t}|$ and $|R^{\mu \cdot g}_{s,t}| \leq L |\mu_{s,t}|^2$. It follows that we may take $c_\mu^2(s,t) := L \|\mu\|_{p,[s,t]}^p$. Finally, we note that the initial values $\pi_0/\mu_0 = g_0 + (1 - \mu_0 \cdot g_0)\1 = F(\mu_0) + (1 - \mu_0 \cdot F(\mu_0))\1$ and $(\pi/\mu)^\prime_0 = \D F(\mu_0) - (F(\mu_0) + \mu_0 \D F(\mu_0))\1 $ are also bounded by a constant $M$ depending only on $K$.
\end{proof}

One particular advantage of rough integration is that the admissible strategies need not be of gradient type, giving us more flexibility in choosing admissible portfolios compared to previous approaches relying on F{\"o}llmer integration.

\begin{example}[Functionally controlled portfolios]\label{ex: functionally controlled portfolios}
  Let
  \begin{equation*}
    \mathcal{F}^{2,K} := \Big\{(\pi^F, \pi^{F,\prime}) : F \in C^2(\overline{\Delta}^d_+;\R^d), \, \|F\|_{C^2} \leq K\Big\}
  \end{equation*}
  for a given constant $K > 0$, where
  \begin{align}\label{eq:portfoliogen}
    (\pi^F_t)^i = \mu^i_t \bigg(F^i(\mu_t) + 1 - \sum_{j=1}^d \mu^j_t F^j(\mu_t)\bigg)
  \end{align}
  for $t \geq 0$ and $i = 1, \ldots, d$. Then $\mathcal{F}^{2,K} \subset \mathcal{A}^{M,p}(c_\mu)$, where we can again take $q = p$. The point here is that we can consider all $C^2$-functions $F$, rather than requiring that $F$ is of the form $F = \nabla \log G$ for some function $G$. One can verify that $\mathcal{F}^{2,K} \subset \mathcal{A}^{M,p}(c_\mu)$ for a suitable control function $c_\mu$ by following the proof of Lemma~\ref{lem: functionally generated portfolios are in base set} almost verbatim.
\end{example}

\begin{example}[Controlled equation generated portfolios]
  Let us define
  \begin{equation*}
    \mathcal{C}^{3,K} := \{f \in C^3(\R^d; \mathcal{L}(\R^d;\R^d)) : \|f\|_{C^3} \leq K\}.
  \end{equation*}
  For a given $f \in \mathcal{C}^{3,K}$, a classical result in rough path theory is that the controlled differential equation with the vector field $f$, driven by $\mu$,
  \begin{equation}\label{eq: controlled equation}
    \d Y^f_t = f(Y^f_t) \dd \mu_t, \qquad Y_0 = \xi \in \Delta^d,
  \end{equation}
  admits a unique solution $(Y^f,(Y^f)^\prime) = (\xi + \int_0^\cdot f(Y^f_u) \dd \mu_u, f(Y^f))$, which is itself a $\mu$-controlled path. Moreover, writing $A^\mu_{s,t} = \int_s^t \mu_{s,u} \otimes d \mu_u$ for the canonical rough path lift of $\mu$ (see Section~\ref{sec Rough paths associated to controlled paths}), and $c_\mu(s,t) := \|\mu\|_{p,[s,t]}^p + \|A^\mu\|_{\frac{p}{2},[s,t]}^{\frac{p}{2}}$, for every $T > 0$, there exists a constant $\Gamma_T$ depending on $p$, $c_\mu([0,T])$ and $K$, such that
  \begin{equation*}
    \sup_{(s,t) \in \Delta_{[0,T]}} \frac{|(Y^f)^\prime_{s,t}|^p}{\Gamma_T c_\mu(s,t)} + \sup_{(s,t) \in \Delta_{[0,T]}} \frac{|R^{Y^f}_{s,t}|^{\frac{p}{2}}}{\Gamma_T c_\mu(s,t)} \leq 1.
  \end{equation*}
  Consequently, as in the proof of Lemma~\ref{lem: functionally generated portfolios are in base set} one can show that there exists an increasing function $\Gamma \colon [0,\infty) \to \R_+$, depending on $p$, $c_\mu$ and $K$  such that
  \begin{equation*}
    \sup_{0 \leq s \leq t < \infty} \frac{\big|(\frac{\pi^f}{\mu})^\prime_{s,t}\big|^p}{\tilde{c}_\mu(s,t)} + \sup_{0 \leq s \leq t < \infty} \frac{\big|R^{\frac{\pi^f}{\mu}}_{s,t}\big|^{\frac{p}{2}}}{\tilde{c}_\mu(s,t)} \leq 1,
  \end{equation*}
  where $\pi^f := \mu(Y^f + (1 - \mu \cdot Y^f)\1)$ and $\tilde{c}_\mu(s,t) := \Gamma_t c_\mu(s,t)$ is again a control function. This implies that the set
  \begin{equation*}
    \big \{\pi^f = \mu(Y^f + (1 - \mu \cdot Y^f)\1) : Y^f \text{ is the solution of \eqref{eq: controlled equation} for some } f \in \mathcal{C}^{3,K}\big\} \subset \mathcal{A}^{M,p}(\tilde{c}_\mu)
  \end{equation*}
  for a suitable constant $M > 0$.
\end{example}

\subsection{Asymptotic growth of universal portfolios}

To investigate the asymptotic growth rates of our pathwise versions of Cover's universal portfolio, we first require some auxiliary results---in particular the compactness of the set of admissible portfolios.

\begin{lemma}
  The set $\mathcal{A}^{M,q}(c_\mu)$ is compact in the topology generated by the family of seminorms $\{p_T^{\mu,q^\prime} : T \in \N\}$ as defined in \eqref{eq: seminorms}, where we recall that $q < q'$.
\end{lemma}

\begin{proof}
  \textit{Step~1}: We first show that the set
  \begin{equation*}
    \mathcal{A} := \bigg\{(Y,Y') \in \mathcal{V}_\mu^{q}([0,\infty);\R^d) \,:\, |Y_0| + |Y^\prime_0| \leq M \quad \text{and} \quad \sup_{s \leq t} \frac{|Y^\prime_{s,t}|^q}{c_\mu(s,t)} + \sup_{s \leq t} \frac{|R^Y_{s,t}|^r}{c_\mu(s,t)} \leq 1\bigg\}
  \end{equation*}
  is compact with respect to the topology generated by the seminorms $\|\hspace{1pt}\cdot\hspace{1pt},\hspace{1pt}\cdot\hspace{1pt}\|_{\mathcal{V}_\mu^{q^\prime},[0,T]}$ for $T \in \N$. It suffices to show that for every fixed $T \in \N$, the set
  \begin{align*}
    \mathcal{A}_T :=  \bigg\{ (Y,Y') \in \mathcal{V}_\mu^{q}([0,T];\R^d) \,:\, &|Y_0| + |Y^\prime_0| \leq M \quad \text{and}\\
    &\sup_{(s,t) \in \Delta_{[0,T]}} \frac{|Y^\prime_{s,t}|^q}{c_\mu(s,t)} + \sup_{(s,t) \in \Delta_{[0,T]}} \frac{|R^Y_{s,t}|^r}{c_\mu(s,t)} \leq 1\bigg\}
  \end{align*}
  is compact with respect to the norm $\|\hspace{1pt}\cdot\hspace{1pt},\hspace{1pt}\cdot\hspace{1pt}\|_{\mathcal{V}_\mu^{q^\prime},[0,T]}$. We first note that, for all $(Y,Y^\prime) \in \mathcal{A}_T$,
  \begin{equation*}
    \|Y^\prime\|_{q,[0,T]} \leq c_\mu(0,T)^{\frac{1}{q}},\quad \|Y^\prime\|_{\infty,[0,T]} \leq M + c_\mu(0,T)^{\frac{1}{q}} \quad \text{and} \quad \|R^Y\|_{r,[0,T]} \leq c_\mu(0,T)^{\frac{1}{r}},
  \end{equation*}
  where the second bound follows from the fact that $|Y^\prime_t| \leq |Y^\prime_0| + |Y^\prime_{0,t}| \leq M + \|Y^\prime\|_{q,[0,T]}$. The $p$-variation of $Y$ can also be controlled as follows. From $Y_{s,t} = Y'_s \mu_{s,t} + R^Y_{s,t}$, we have
  \begin{equation*}
    |Y_{s,t}|^p \leq 2^{p-1} \big(\|Y^\prime\|^p_{\infty,[0,T]} |\mu_{s,t}|^p + |R^Y_{s,t}|^p\big), \qquad (s,t) \in \Delta_{[0,T]},
  \end{equation*}
  and hence
  \begin{equation*}
    \|Y\|_{p,[0,T]} \leq 2^{\frac{p-1}{p}}\big(\|Y^\prime\|_{\infty,[0,T]} \|\mu\|_{p,[0,T]} + \|R^Y\|_{p,[0,T]}\big) \leq 2^{\frac{p-1}{p}} \big(\|Y^\prime\|_{\infty,[0,T]} \|\mu\|_{p,[0,T]} + \|R^Y\|_{r,[0,T]}\big),
  \end{equation*}
  since $r < p$ (see e.g.~\cite[Remark~2.5]{Chistyakov1998}), and thus
  \begin{equation*}
    \|Y\|_{\infty,[0,T]} \leq M + \|Y\|_{p,[0,T]} \leq M + 2^{\frac{p-1}{p}} \Big((M + c_\mu(0,T)^{\frac{1}{q}}) \|\mu\|_{p,[0,T]} + c_\mu(0,T)^{\frac{1}{r}}\Big).
  \end{equation*}
  Therefore, by \cite[Proposition~5.28]{Friz2010}, every sequence $(Y^n,Y^{n,\prime})_{n \geq 1} \subset \mathcal{A}_T$ has a convergent subsequence, which we still denote by $(Y^n,Y^{n,\prime})_{n \geq 1}$, and limits $Y \in C^{p\textup{-var}}([0,T];\mathbb{R}^d)$ and $Y^\prime \in C^{q\textup{-var}}([0,T];\mathbb{R}^d)$, such that $|Y^n_0 - Y_0| + \|Y^n - Y\|_{p',[0,T]} \to 0$ and $|Y^{n,\prime}_0 - Y^\prime_0| + \|Y^{n,\prime} - Y^\prime\|_{q',[0,T]} \to 0$ respectively as $n \to \infty$, for an arbitrary $p^\prime > p$. Since
  \begin{align*}
    |R^{Y^n}_{s,t} - R^{Y^{n+1}}_{s,t}| &\leq |Y^{n,\prime}_s \mu_{s,t} - Y^{n+1,\prime}_s \mu_{s,t}| + |Y^n_{s,t} - Y^{n+1}_{s,t}|\\
    &\leq |Y^{n,\prime}_s - Y^{n+1,\prime}_s| |\mu_{s,t}| + |Y^n_{t} - Y^{n+1}_{t}|+|Y^n_{s} - Y^{n+1}_{s}| \, \longrightarrow \, 0
  \end{align*}
  as $n \to \infty$, uniformly in $(s,t) \in \Delta_{[0,T]}$, we have that
  \begin{align*}
    \|R^{Y^n} - R^{Y^{n+1}}\|_{r^\prime,[0,T]} &\leq \| R^{Y^n} - R^{Y^{n+1}}\|_{r,[0,T]}^{\frac{r}{r^\prime}} \sup_{(s,t) \in \Delta_{[0,T]}}|R^{Y^n}_{s,t} - R^{Y^{n+1}}_{s,t}|^{\frac{r^\prime - r}{r^\prime}}\\
    &\leq 2^{\frac{r}{r^\prime}} c_\mu(0,T)^{\frac{1}{r^\prime}} \sup_{(s,t) \in \Delta_{[0,T]}}|R^{Y^n}_{s,t} - R^{Y^{n+1}}_{s,t}|^{\frac{r^\prime - r}{r^\prime}} \, \longrightarrow \, 0
  \end{align*}
  as $n \to \infty$. Thus, $R^{Y^n}$ also converges to some $R^Y$ in $r'$-variation.

  To see that the limit $(Y,Y') \in \mathcal{A}_T$, we simply note that
  \begin{align*}
    \frac{|Y'_{s,t}|^q}{c_\mu(s,t)} + \frac{|R^Y_{u,v}|^r}{c_\mu(u,v)} &= \lim_{n \to \infty} \bigg(\frac{|Y^{n,\prime}_{s,t}|^q}{c_\mu(s,t)} + \frac{|R^{Y^n}_{u,v}|^r}{c_\mu(u,v)}\bigg) \leq 1,
  \end{align*}
  and then take the supremum over $(s,t) \in \Delta_{[0,T]}$ and $(u,v) \in \Delta_{[0,T]}$ on the left-hand side.

  Thus, $\mathcal{A}_T$ is compact with respect to $p^{\mu,q'}_T$, and $\mathcal{A}$ is then compact in the topology generated by the seminorms $p^{\mu,q'}_T$ for $T \in \N$.

  \textit{Step~2:} Now suppose that $\{(\pi^n,\pi^{n,\prime})\}_{n \in \N}$, is a sequence of portfolios in $\mathcal{A}^{M,q}(c_\mu)$. Correspondingly, $\{(\frac{\pi^n}{\mu},(\frac{\pi^n}{\mu})^{\prime})\}_{n \in \N}$, is then a sequence in $\mathcal{A}$ which, by the result in Step~1 above, admits a convergent subsequence with respect to the seminorms $\|\hspace{1pt}\cdot\hspace{1pt},\hspace{1pt}\cdot\hspace{1pt}\|_{\mathcal{V}_\mu^{q^\prime},[0,T]}$ for $T \in \N$. Since $\|\frac{\pi}{\mu},(\frac{\pi}{\mu})^{\prime}\|_{\mathcal{V}_\mu^{q^\prime},[0,T]} = p^{\mu,q^\prime}_T((\pi,\pi^\prime))$, the convergence also applies to the corresponding subsequence of $\{(\pi^n,\pi^{n,\prime})\}_{n \in \N}$ with respect to the seminorms $\{p^{\mu,q^\prime}_T\}_{T \in \N}$. Let $(\phi,\phi^\prime)$ be the limit of (the convergent subsequence of) $\{(\frac{\pi^n}{\mu},(\frac{\pi^n}{\mu})^{\prime})\}_{n \in \N}$. It is then easy to see that $\phi \mu$, the product of controlled paths $(\phi,\phi^\prime)$ and $(\mu,I)$, is a cluster point of $\{(\pi^n,\pi^{n,\prime})\}_{n \in \N}$ in $\mathcal{A}^{M,q}(c_\mu)$ with respect to the seminorms $\{p^{\mu,q^\prime}_T\}_{T \in \N}$.
\end{proof}

In the next auxiliary result, we establish continuity of the relative wealth of admissible portfolios with respect to the market portfolio. To this end, we recall the family of seminorms $\{p^{\mu,q^\prime}_T\}_{T > 0}$, defined in \eqref{eq: seminorms}, and, for a given sequence $\beta = \{\beta_N\}_{N \in \N}$ with $\beta_N  > 0$ for all $N \in \N$ and $\lim_{N \to \infty} \beta_N = \infty$, we introduce a metric $d_\beta$ on $\mathcal{A}^{M,q}(c_\mu)$ via
\begin{equation*}
  d_\beta((\pi,\pi^\prime), (\phi,\phi^\prime)) := \sup_{N \ge 1} \frac{1}{\beta_N \gamma_N} p^{\mu,q^\prime}_N((\pi,\pi^\prime) - (\phi,\phi^\prime)),
\end{equation*}
where
\begin{equation*}
  \gamma_N := 1 + M + c_\mu(0,N)^{\frac{1}{q}} + c_{\mu}(0,N)^{\frac{1}{r}}.
\end{equation*}
Since $p^{\mu,q^\prime}_N((\pi,\pi^\prime)) \leq \gamma_N$, we have that $d_\beta((\pi,\pi^\prime), (\phi,\phi^\prime)) < \infty$ for all portfolios $(\pi,\pi^\prime), (\phi,\phi^\prime) \in \mathcal{A}^{M,q}(c_\mu)$. The metric $d_\beta$ is thus well-defined on $\mathcal{A}^{M,q}(c_\mu)$. Moreover, it is not hard to see that the topology induced by the metric $d_\beta$ coincides with the topology generated by the family of seminorms $\{p^{\mu,q^\prime}_T\}_{T \in \N}$, so that $(\mathcal{A}^{M,q}(c_\mu), d_\beta)$ is a compact metric space. For $T > 0$, we also denote
\begin{equation}\label{eq:defn xi_T}
  \xi_T := \|\mu\|_{p,[0,T]} + \|A^\mu\|_{\frac{p}{2},[0,T]} + \sum_{i=1}^d [\mu]^{ii}_T.
\end{equation}

\begin{lemma}\label{lem: continuity of relative wealth}
  For any $T > 0$, we have that the estimate
  \begin{equation}\label{eq: Lipschitz estimate of log relative wealth}
    |\log V^\pi_T - \log V^\phi_T| \leq C \beta_N \gamma_N^2 \xi_T \hspace{1pt} d_\beta((\pi,\pi'),(\phi,\phi'))
  \end{equation}
  holds for all $(\pi,\pi'), (\phi,\phi') \in \mathcal{A}^{M,q}(c_\mu)$, for some constant $C$ which depends only on $p, q', r'$ and the dimension $d$, where $N = \lceil T \rceil$, and $V^\pi$ denotes the relative wealth process as defined in \eqref{eq:defn relative wealth process}. In particular, the map from $\mathcal{A}^{M,q}(c_\mu) \to \R$ given by $(\pi,\pi') \mapsto V^\pi_T$ is continuous with respect to the metric $d_{\beta}$.
\end{lemma}

\begin{proof}
  By Proposition~\ref{prop: log of relative wealth} and the relation in \eqref{eq:quad var of mu}, we have that, for any $(\pi,\pi') \in \mathcal{A}^{M,q}(c_\mu)$,
  \begin{equation*}
    \log V^\pi_T = \int_0^T \frac{\pi_s}{\mu_s} \dd \mu_s - \frac{1}{2}\sum_{i,j=1}^d \int_0^T \frac{\pi^i_s\pi^j_s}{\mu^i_s\mu^j_s} \dd [\mu]^{ij}_s,
  \end{equation*}
  which implies that, for $(\pi,\pi'), (\phi,\phi') \in \mathcal{A}^{M,q}(c_\mu)$,
  \begin{equation*}
    |\log V^\pi_T - \log V^\phi_T| \leq \bigg|\int_0^T \frac{\pi_s - \phi_s}{\mu_s} \dd \mu_s\bigg| + \frac{1}{2} \bigg|\sum_{i,j=1}^d \int_0^T \frac{(\pi^i_s - \phi^i_s)(\pi^j_s + \phi^j_s)}{\mu^i_s\mu^j_s} \dd [\mu]^{ij}_s\bigg|.
  \end{equation*}

  We aim to bound the two terms on the right-hand side. Let $A^\mu$ be the canonical rough path lift of $\mu$ (as defined in Section~\ref{sec Rough paths associated to controlled paths}), namely $A^\mu_{s,t} = \int_s^t \mu_{s,u} \otimes \d \mu_u$. Writing $N = \lceil T \rceil$, by the estimate for rough integrals in \eqref{eq:rough integral estimate}, we obtain
  \begin{align*}
    \bigg|\int_0^T \frac{\pi_s - \phi_s}{\mu_s} \dd \mu_s\bigg| &\lesssim \|R^{\frac{\pi - \phi}{\mu}}\|_{r',[0,T]} \|\mu\|_{p,[0,T]} + \Big\|\Big(\frac{\pi - \phi}{\mu}\Big)^\prime\Big\|_{q',[0,T]} \|A^\mu\|_{\frac{p}{2},[0,T]}\\
    &\qquad + \Big|\frac{\pi_0 - \phi_0}{\mu_0}\Big| \|\mu\|_{p,[0,T]} + \Big|\Big(\frac{\pi - \phi}{\mu}\Big)^\prime_0\Big| \|A^\mu\|_{\frac{p}{2},[0,T]}\\
    &\lesssim  p^{\mu,q^\prime}_N((\pi,\pi^\prime) - (\phi,\phi^\prime))(\|\mu\|_{p,[0,T]} + \|A^\mu\|_{\frac{p}{2},[0,T]})\\
    &\leq \beta_N \gamma_N d_{\beta}((\pi,\pi'),(\phi,\phi')) (\|\mu\|_{p,[0,T]} + \|A^\mu\|_{\frac{p}{2},[0,T]}).
  \end{align*}

  For the second term, we note that
  \begin{equation}\label{eq:ugly estimate}
    \bigg|\int_0^T \frac{(\pi^i_s - \phi^i_s)(\pi^j_s + \phi^j_s)}{\mu^i_s \mu^j_s} \dd [\mu]^{ij}_s\bigg| \lesssim \Big\|\frac{\pi - \phi}{\mu}\Big\|_{\infty,[0,T]} \Big\|\frac{\pi + \phi}{\mu}\Big\|_{\infty,[0,T]} \sum_{i=1}^d [\mu]^{ii}_T.
  \end{equation}
  It follows from the relation $\frac{\pi_t}{\mu_t} = \frac{\pi_0}{\mu_0} + (\frac{\pi}{\mu})^\prime_0 \mu_{0,t} + R^{\frac{\pi}{\mu}}_{0,t}$, and the fact that $\mu$ takes values in the bounded set $\Delta^d_+$, that
  \begin{equation*}
    \Big\|\frac{\pi}{\mu}\Big\|_{\infty,[0,T]} \lesssim M + c_{\mu}(0,T)^{\frac{1}{r}} \leq \gamma_N.
  \end{equation*}
  It follows similarly from $\frac{\pi_t - \phi_t}{\mu_t} = \frac{\pi_0 - \phi_0}{\mu_0} + (\frac{\pi - \phi}{\mu})^\prime_0 \mu_{0,t} + R^{\frac{\pi - \phi}{\mu}}_{0,t}$, that
  \begin{equation*}
    \Big\|\frac{\pi - \phi}{\mu}\Big\|_{\infty,[0,T]} \lesssim p^{\mu,q^\prime}_N((\pi,\pi') - (\phi,\phi')) \leq \beta_N \gamma_N d_\beta((\pi,\pi'),(\phi,\phi')).
  \end{equation*}
  Substituting back into \eqref{eq:ugly estimate}, we obtain
  \begin{equation*}
    \bigg|\int_0^T \frac{(\pi^i_s - \phi^i_s)(\pi^j_s + \phi^j_s)}{\mu^i_s \mu^j_s} \dd [\mu]^{ij}_s\bigg| \lesssim \beta_N \gamma_N^2 d_\beta((\pi,\pi'),(\phi,\phi')) \sum_{i=1}^d [\mu]^{ii}_T.
  \end{equation*}
  Combining the inequalities above, we deduce the desired estimate.
\end{proof}

In the following, we will sometimes write simply $\mathcal{A}^{M,q} := \mathcal{A}^{M,q}(c_\mu)$ for brevity.

\medskip

For $(\pi,\pi') \in \mathcal{A}^{M,q}$, we have by definition that $\pi$ is a $\mu$-controlled path. We also have that the relative wealth $V^\pi$ is also a $\mu$-controlled rough path---as can be seen for instance from Proposition~\ref{prop: log of relative wealth}---and hence the product $\pi V^\pi$ is also a controlled path. Let $\nu$ be a fixed probability measure on $(\mathcal{A}^{M,q},d_{\beta})$. Observe that for every $T > 0$ the space $\mathcal{V}_\mu^{q}([0,T];\R^d)$ of controlled paths is a Banach space, and that, as we will see during the proof of Lemma~\ref{lemma wealth of universal portfolio} below, $V^\pi$ is the unique solution to the rough differential equation \eqref{eq:RDE for relative wealth of universal portfolio}, which implies that the mapping $\pi \mapsto V^\pi|_{[0,T]} \in \mathcal{V}_\mu^{q}([0,T];\R^d)$ is continuous by the continuity of the It\^o--Lyons map (see e.g.~\cite[Theorem~1]{Lejay2014a}). Hence, for every $T >0$ we can define the Bochner integral $\int_{\mathcal{A}^{M,q}} (\pi V^\pi)|_{[0,T]} \dd \nu(\pi)$, which is thus itself another controlled path defined on $[0,T]$. The $\mu$-controlled path
\begin{equation}\label{eq: universal portfolio}
  \pi^\nu_t := \frac{\int_{\mathcal{A}^{M,q}} \pi_t V^\pi_t \dd \nu(\pi)}{\int_{\mathcal{A}^{M,q}} V^\pi_t \dd \nu(\pi)}, \qquad t \in [0,\infty),
\end{equation}
is then well-defined, and defines indeed a portfolio in $\mathcal{V}_{\mu}^q$, called the \textit{universal portfolio} associated to the set $\mathcal{A}^{M,q}$ of admissible portfolios.

\begin{lemma}\label{lemma wealth of universal portfolio}
  Let $\pi^\nu$ be the universal portfolio as defined in \eqref{eq: universal portfolio}. Then, for all $T > 0$,
  \begin{equation*}
    V^{\pi^\nu}_T = \int_{\mathcal{A}^{M,q}} V^\pi_T \dd \nu(\pi).
  \end{equation*}
\end{lemma}

\begin{proof}
  By Proposition~\ref{prop: log of relative wealth} and the relation in \eqref{eq:quad var of mu}, we have, for any portfolio $\pi$,
  \begin{equation*}
    V^\pi_t = \exp \bigg(\int_0^t \frac{\pi_s}{\mu_s} \dd \mu_s - \frac{1}{2} \sum_{i,j=1}^d \int_0^t \frac{\pi^{i}_s \pi^{j}_s}{\mu^i_s \mu^j_s} \dd [\mu]^{ij}_s\bigg).
  \end{equation*}
  Setting $Z := \int_0^\cdot \frac{\pi_s}{\mu_s} \dd \mu_s$, by Lemma~\ref{lem: Ito isometry for rough paths}, we can rewrite the relation above as $V^\pi = \exp (Z - \frac{1}{2} [\bZ])$. Thus, by Lemma~\ref{lem: dynamics of rough exponential}, Lemma~\ref{lem: consistency of rough integrals} and Proposition~\ref{prop: associativity of rough integration}, we deduce that $V^\pi$ is the unique solution $Y$ to the linear rough differential equation
  \begin{equation}\label{eq:RDE for relative wealth of universal portfolio}
    Y_t = 1 + \int_0^t Y_s \frac{\pi_s}{\mu_s} \dd \mu_s, \qquad t \geq 0.
  \end{equation}
  It is therefore sufficient to show that the path $t \mapsto \int_{\mathcal{A}^{M,q}} V^\pi_t \dd \nu(\pi)$ also satisfies the RDE~\eqref{eq:RDE for relative wealth of universal portfolio} with $\pi$ replaced by $\pi^\nu$. By the definition of the universal portfolio in \eqref{eq: universal portfolio}, we have
  \begin{equation}\label{eq:universal portfolio simple relation}
    \int_{\mathcal{A}^{M,q}} V^\pi_s \dd \nu(\pi) \frac{\pi^\nu_s}{\mu_s} = \int_{\mathcal{A}^{M,q}} \frac{\pi_s}{\mu_s} V^\pi_s \dd \nu(\pi).
  \end{equation}
  Recalling that $V^\pi$ satisfies \eqref{eq:RDE for relative wealth of universal portfolio}, we know that
  \begin{equation*}
    V^\pi_t = 1 + \int_0^t \frac{\pi_s}{\mu_s} V^\pi_s \dd \mu_s.
  \end{equation*}
By the Fubini theorem for rough integration (Theorem~\ref{thm: rough Fubini}) we then have that
\begin{align*}
\int_{\mathcal{A}^{M,q}} V^\pi_t \dd \nu(\pi) &= 1 + \int_0^t \int_{\mathcal{A}^{M,q}} \frac{\pi_s}{\mu_s} V^\pi_s \dd \nu(\pi) \dd \mu_s = 1 + \int_0^t \int_{\mathcal{A}^{M,q}} V^\pi_s \dd \nu(\pi) \frac{\pi^\nu_s}{\mu_s} \dd \mu_s,
\end{align*}
where we used \eqref{eq:universal portfolio simple relation} to obtain the last equality. Hence, both $V^{\pi^\nu}$ and $\int_{\mathcal{A}^{M,q}} V^\pi \dd \nu(\pi)$ are the unique solution of the same RDE, and thus coincide.
\end{proof}

With these preparations in place, we now aim to compare the growth rates of the universal portfolio \eqref{eq: universal portfolio} and the best retrospectively chosen portfolio. For this purpose, we fix an $M > 0$, and assume that there exists a compact metric space $(\mathcal{K}, d_{\mathcal{K}})$ together with a mapping $\iota : (\mathcal{K}, d_{\mathcal{K}}) \to (\mathcal{A}^{M,q}, d_\beta)$ such that $\iota$ is continuous and injective (and thus a homeomorphism onto its image), and that for every $T > 0$ and $x, y \in \mathcal{K}$, we have that
\begin{equation}\label{eq: lipschitz type estimation for embedding}
  |\log V^{\iota(x)}_T - \log V^{\iota(y)}_T| \leq C \lambda(T) d_{\mathcal{K}}(x,y),
\end{equation}
where $\lambda$ is a positive function of $T$, and $C$ is a universal constant independent of $T$. Here we list some examples of $(\mathcal{K}, d_{\mathcal{K}})$, $\iota$ and $\lambda$:
\begin{enumerate}
  \item $\mathcal{K} = C^{p+\alpha,K}(\overline{\Delta}^d_+; \R^d) = \{G \in C^{p+\alpha}(\overline{\Delta}^d_+;\R^d) : \|G\|_{C^{p+\alpha}} \leq K, \hspace{1pt} G \geq \frac{1}{K}\}$, $d_{\mathcal{K}}(G,\tilde{G}) = \|G - \tilde{G}\|_{C^{2}}$, $\iota(G) = \pi^G$, where $\alpha > 0$ and $\pi^G$ is a classical functionally generated portfolio of the form \eqref{eq: canonical function generated portfolio}. In this case we can take $\lambda(T) = 1 + \max_{i=1,\ldots,d} [\mu]^{ii}_T$; see the proof of \cite[Lemma~4.4]{Cuchiero2019}.
  \item $\mathcal{K} = C^{2+\alpha,K}(\overline{\Delta}^d_+; \R^d) = \{F \in C^{2+\alpha}(\overline{\Delta}^d_+;\R^d): \|F\|_{C^{2+\alpha}} \leq K\}$, $d_{\mathcal{K}}(F,\tilde{F}) = \|F - \tilde{F}\|_{C^{2}}$, $\iota(F) = \pi^F$, where $\alpha \in (0,1]$ and $\pi^F$ is a functionally controlled portfolio defined as in \eqref{eq:portfoliogen}. In this case one may take $\lambda(T) = (1 + \|\mu\|_{p,[0,T]}^2) \xi_T$, where $\xi_T$ is defined in \eqref{eq:defn xi_T}; see Lemma \ref{lemma continuity log wealth func cont port} below.
  \item $\mathcal{K} = \mathcal{A}^{M,q}$, $d_{\mathcal{K}} = d_\beta$, $\iota = \text{Id}_{\mathcal{A}^{M,q}}$.  In view of \eqref{eq: Lipschitz estimate of log relative wealth} we have $\lambda(T) = \beta_{\lceil T \rceil}\gamma_{\lceil T \rceil}^2\xi_T$.
\end{enumerate}
Given such a compact space $(\mathcal{K}, d_{\mathcal{K}})$ equipped with an embedding $\iota$ as above, we define
\begin{equation*}
  V^{\ast,\mathcal{K},\iota}_T = \sup_{x \in \mathcal{K}}V^{\iota(x)}_T = \sup_{\pi \in \iota(\mathcal{K})} V^\pi_T.
\end{equation*}
By the compactness of $\mathcal{K}$ and the continuity provided by the estimate in \eqref{eq: lipschitz type estimation for embedding}, we have that, for each $T > 0$, there exists a portfolio $\pi^{\ast,T} \in \iota(\mathcal{K})$, which can be expressed as $\pi^{\ast,T} = \iota(x^\ast)$ for some $x^\ast \in \mathcal{K}$, known as the \emph{best retrospectively chosen portfolio} associated with $\mathcal{K}$ and $\iota$, such that
\begin{equation}\label{eq:best retro portfolio}
  V^{\ast,\mathcal{K},\iota}_T = V^{\pi^{\ast,T}}.
\end{equation}

The following theorem provides an analogue of \cite[Theorem~4.11]{Cuchiero2019} in our rough path setting.

\begin{theorem}\label{thm: Cover's theorem}
  Let $(\mathcal{K},d_{\mathcal{K}})$ be a compact metric space equipped with a continuous embedding $\iota\colon (\mathcal{K}, d_{\mathcal{K}}) \to (\mathcal{A}^{M,q}, d_\beta)$ which satisfies the bound in \eqref{eq: lipschitz type estimation for embedding} for some positive function $\lambda$. Let $m$ be a probability measure on $\mathcal{K}$ with full support, and let $\nu = \iota_{\ast}(m)$ denote the pushforward measure on $\mathcal{A}^{M,q}$. If $\lim_{T \to \infty} \lambda(T) = \infty$, then
  \begin{equation*}
    \lim_{T \to \infty} \frac{1}{\lambda(T)} \Big(\log V^{\ast,\mathcal{K},\iota}_T - \log V^{\pi^\nu}_T\Big) = 0.
  \end{equation*}
\end{theorem}

In particular, if $\mathcal{K} = C^{p+\alpha,K}(\overline{\Delta}^d_+; \R^d) = \{G \in C^{p+\alpha}(\overline{\Delta}^d_+;\R^d) : \|G\|_{C^{p+\alpha}} \leq K, \hspace{1pt} G \ge \frac{1}{K}\}$, $d_{\mathcal{K}}(G,\tilde{G}) = \|G - \tilde{G}\|_{C^{2}}$, $\iota(G) = \pi^G$, where $\pi^G$ is a classical functionally generated portfolio of the form \eqref{eq: canonical function generated portfolio}, and $\lambda(T) = 1 + \max_{i=1,\ldots,d} [\mu]^{ii}_T$, then one also infers the version of Cover's theorem obtained in \cite[Theorem 4.11]{Cuchiero2019}.

\begin{proof}[Proof of Theorem~\ref{thm: Cover's theorem}]
  As the inequality ``$\geq$'' is trivial, we need only show the reverse inequality. As $\mathcal{K}$ is compact and $m$ has full support, we have that, for any $\eta \in (0,1)$, there exists a $\delta > 0$ such that every $\eta$-ball around a point $x \in \mathcal{K}$ with respect to $d_{\mathcal{K}}$ has $m$-measure bigger than $\delta$.

  Let $T > 0$ be such that $\lambda(T) \geq 1$, and let $\pi^{\ast,T} = \iota(x^*)$ be the best retrospectively chosen portfolio, as in \eqref{eq:best retro portfolio}. For any portfolio $\pi  = \iota(x) \in \iota(\mathcal{K}) \subseteq \mathcal{A}^{M,q}(c_{\mu})$ such that $d_{\mathcal{K}}(x, x^*) \leq \eta$, the estimate in \eqref{eq: lipschitz type estimation for embedding} implies that
  \begin{equation*}
    \frac{1}{\lambda(T)} \Big(\log V^\pi_T - \log V^{\pi^{\ast,T}}_T\Big) \geq -C d_{\mathcal{K}}(x, x^*) \geq -C \eta,
  \end{equation*}
  for some constant~$C$. For any $\varepsilon > 0$, we can therefore choose $\eta$ small enough such that
  \begin{equation}\label{eq:inequality diff log rel wealth epsilon}
    \frac{1}{\lambda(T)} \Big(\log V^\pi_T - \log V^{\pi^{\ast,T}}_T\Big) \geq -\varepsilon.
  \end{equation}
  Let $B_{\eta}(x^*)$ denote the $\eta$-ball in $\mathcal{K}$ around the point $x^*$ with respect to the metric $d_{\mathcal{K}}$, which has $m$-measure $|B_{\eta}(x^*)| \geq \delta$. By Lemma~\ref{lemma wealth of universal portfolio} and Jensen's inequality, we have that
  \begin{equation*}
    (V^{\pi^\nu}_T)^{\frac{1}{\lambda(T)}} \geq \bigg(\int_{B_{\eta}(x^*)} V^{\iota(x)}_T \dd m(x)\bigg)^{\frac{1}{\lambda(T)}} \geq |B_{\eta}(x^*)|^{\frac{1}{\lambda(T)} - 1} \int_{B_{\eta}(x^*)} (V^{\iota(x)}_T)^{\frac{1}{\lambda(T)}} \dd m(x).
  \end{equation*}
  Then, using \eqref{eq:inequality diff log rel wealth epsilon}, we have
  \begin{equation*}
    \bigg(\frac{V^{\pi^\nu}_T}{V^{\pi^{\ast,T}}_T} \bigg)^{\frac{1}{\lambda(T)}} \geq |B_{\eta}(x^*)|^{\frac{1}{\lambda(T)} - 1} \int_{B_{\eta}(x^*)} \bigg(\frac{V^{\iota(x)}_T}{V^{\iota(x^*)}_T} \bigg)^{\frac{1}{\lambda(T)}} \dd m(x) \geq |B_{\eta}(x^*)|^{\frac{1}{\lambda(T)}} e^{-\epsilon} \geq \delta^{\frac{1}{\lambda(T)}} e^{-\epsilon}.
  \end{equation*}
  Taking $\epsilon > 0$ arbitrarily small (which determines $\eta$ and hence also $\delta$) and then $T > 0$ sufficiently large, we deduce the desired inequality.
\end{proof}

\subsection{Universal portfolios based on functionally controlled portfolios}\label{subsection universal portfolios functionally controlled}

The most frequently considered classes of portfolios are those which are generated by functions acting on the underlying price trajectories, such as the functionally generated portfolios in Lemma~\ref{lem: functionally generated portfolios are in base set}. In this section we shall investigate the growth rate of universal portfolios based on the more general class of functionally controlled portfolios, as introduced in Example~\ref{ex: functionally controlled portfolios}. More precisely, we fix constants $\alpha \in (0,1]$ and $K > 0$, and consider the sets
\begin{equation*}
  C^{2+\alpha,K}(\overline{\Delta}^d_+;\R^d) := \{F \in C^{2+\alpha}(\overline{\Delta}^d_+;\R^d): \|F\|_{C^{2+\alpha}} \leq K\}
\end{equation*}
and
\begin{equation*}
  \mathcal{F}^{2+\alpha,K} := \big\{(\pi^F, \pi^{F,\prime}): F \in C^{2+\alpha,K}(\overline  \Delta^d_+;\R^d)\big\},
\end{equation*}
where the portfolio $\pi^F$ is of the form in \eqref{eq:portfoliogen}. Here we recall that $C^{2+\alpha}$ denotes the space of twice continuously differentiable functions whose second derivative is $\alpha$-H{\"o}lder continuous.

\begin{lemma}\label{lem: continuity embedding of function spaces}
  For any $T > 0$ and any $F, G \in C^{2+\alpha,K}(\overline{\Delta}^d_+;\R^d)$, we have that
  \begin{equation}\label{eq: local estimates for generalized functionally generated portfolios}
    p^{\mu,p}_T((\pi^F,\pi^{F,\prime}) - (\pi^G,\pi^{G,\prime})) \leq C \|F - G\|_{C^2} (1 + \|\mu\|_{p,[0,T]}^2),
  \end{equation}
  where the constant $C$ depends only on $p, d$ and $K$. Considering the map $\Phi \colon C^{2+\alpha,K}(\overline{\Delta}^d_+;\R^d) \rightarrow \mathcal{F}^{2+\alpha,K}$ given by\footnote{Note that $\Phi$ plays the role of the embedding $\iota$ in the previous section.}
  \begin{equation*}
    F \, \mapsto \, \Phi(F) := (\pi^F, \pi^{F,\prime}),
  \end{equation*}
  where $\pi^F$ is of the form in \eqref{eq:portfoliogen}, we thus have that $\Phi$ is continuous with respect to the $C^{2}$-distance on $C^{2+\alpha,K}(\overline{\Delta}^d_+;\R^d)$ and each of the seminorms $\{p^{\mu,p}_T\}_{T > 0}$ on $\mathcal{F}^{2+\alpha,K} \subset \mathcal{A}^{M,p}(c_\mu)$. As the notation suggests, here $p^{\mu,p}_T$ is defined as in \eqref{eq: seminorms} with $q'$ replaced by $p$.
\end{lemma}

\begin{proof}
  In the following, for notational simplicity we will omit the Gubinelli derivative in the norms $\|\hspace{1pt}\cdot\hspace{1pt},\hspace{1pt}\cdot\hspace{1pt}\|_{\mathcal{V}^p_\mu,[0,T]}$ and seminorms $p^{\mu,p}_T((\hspace{1pt}\cdot\hspace{1pt},\hspace{1pt}\cdot\hspace{1pt}))$; that is, we will write e.g.~$\|\pi\|_{\mathcal{V}^p_\mu,[0,T]}$ instead of $\|\pi,\pi'\|_{\mathcal{V}^p_\mu,[0,T]}$. Let $F, G \in C^{2+\alpha,K}$ and $s \leq t$. We have
  \begin{align*}
   \Big|(\D F - \D G)(\mu_t) - (\D F - \D G)(\mu_s)\Big| &= \bigg|\int_0^1 (\D^2F - \D^2G)(\mu_s + \lambda \mu_{s,t}) \mu_{s,t} \dd \lambda\bigg|\\
   &\leq \|\D^2F - \D^2G\|_{\infty} |\mu_{s,t}|,
  \end{align*}
  so that
  \begin{equation*}
    \|\D F(\mu)  - \D G(\mu)\|_{p,[0,T]} \leq \|F - G\|_{C^2} \|\mu\|_{p,[0,T]}.
  \end{equation*}
  Similarly, since
  \begin{equation*}
    R^{F(\mu)}_{s,t} = F(\mu_t) - F(\mu_s) - \D F(\mu_s) \mu_{s,t} = \int_0^1 \int_0^1 \D^2F(\mu_s + \lambda_1 \lambda_2 \mu_{s,t}) \mu_{s,t}^{\otimes 2} \lambda_1 \dd \lambda_2 \dd \lambda_1,
  \end{equation*}
  we have
  \begin{equation*}
    \|R^{F(\mu)} - R^{G(\mu)}\|_{\frac{p}{2},[0,T]} \leq \|F - G\|_{C^2} \|\mu\|_{p,[0,T]}^2.
  \end{equation*}
  Thus, for $\mu$-controlled paths $(F(\mu), \D F(\mu))$ and $(G(\mu), \D G(\mu))$, we have that
  \begin{equation}\label{eq: F(mu) - G(mu) cont path norm}
    \|F(\mu) - G(\mu)\|_{\mathcal{V}^p_\mu,[0,T]} \lesssim \|F - G\|_{C^2} (1 + \|\mu\|_{p,[0,T]}^2).
  \end{equation}
  Writing $\pi^F_t/\mu_t = F(\mu_t) + (1 - \mu_t \cdot F(\mu_t))\1$ and $\pi^G_t/\mu_t = G(\mu_t) + (1 - \mu_t \cdot G(\mu_t))\1$, we have that
  \begin{equation*}
    \frac{\pi^F_t - \pi^G_t}{\mu_t} = F(\mu_t) - G(\mu_t) - (\mu_t \cdot (F(\mu_t) - G(\mu_t)))\1,
  \end{equation*}
  so that
  \begin{equation}\label{eq:seminorm of pi^F - pi^G}
    p^{\mu,p}_T(\pi^F - \pi^G) \lesssim \|F(\mu) - G(\mu)\|_{\mathcal{V}^p_\mu,[0,T]} + \|\mu \cdot (F(\mu) - G(\mu))\|_{\mathcal{V}^p_\mu,[0,T]}.
  \end{equation}
  Similarly to the proof of Lemma~\ref{lem: functionally generated portfolios are in base set}, noting that $R^{\mu \cdot (F(\mu) - G(\mu))}_{s,t} = \mu_s \cdot R^{F(\mu) - G(\mu)}_{s,t} + \mu_{s,t} \cdot (F(\mu) - G(\mu))_{s,t}$, we have that
  \begin{equation*}
    |R^{\mu \cdot (F(\mu) - G(\mu))}_{s,t}| \leq \|\mu\|_{\infty,[0,T]} |R^{F(\mu) - G(\mu)}_{s,t}| + |\mu_{s,t}| |(F(\mu) - G(\mu))_{s,t}| \lesssim \|F - G\|_{C^2} |\mu_{s,t}|^2,
  \end{equation*}
  where we used the fact that $\mu$ is bounded, and we deduce that
  \begin{equation*}
    \|\mu \cdot (F(\mu) - G(\mu))\|_{\mathcal{V}^p_\mu,[0,T]} \lesssim \|F - G\|_{C^2} (1 + \|\mu\| _{p,[0,T]}^2).
  \end{equation*}
  Combining this with \eqref{eq: F(mu) - G(mu) cont path norm} and \eqref{eq:seminorm of pi^F - pi^G}, we obtain the estimate in \eqref{eq: local estimates for generalized functionally generated portfolios}, which then implies the desired continuity of $\Phi$.
\end{proof}

\begin{lemma}\label{lemma continuity log wealth func cont port}
  For any $T > 0$ and any $F, G \in C^{2+\alpha,K}(\overline{\Delta}^d_+;\R^d)$, we have that
  \begin{equation}\label{eq: bound for log wealth for generalized functionally generated portfolios}
    |\log V^{\pi^F}_T - \log V^{\pi^G}_T| \leq C \|F - G\|_{C^2} (1 + \|\mu\|_{p,[0,T]}^2) \xi_T,
  \end{equation}
  where $\xi_T$ is defined as in \eqref{eq:defn xi_T}, and the constant $C$ depends only on $p, d$ and $K$.
\end{lemma}

\begin{proof}
  We recall that during the proof of Lemma~\ref{lem: continuity of relative wealth} we showed that
  \begin{equation*}
    |\log V^{\pi^F}_T - \log V^{\pi^G}_T| \leq \bigg|\int_0^T \frac{\pi^F_s - \pi^G_s}{\mu_s} \dd \mu_s\bigg| + \frac{1}{2} \bigg|\sum_{i,j=1}^d \int_0^T \frac{(\pi^{F,i}_s - \pi^{G,i}_s)(\pi^{F,j}_s + \pi^{G,j}_s)}{\mu^i_s \mu^j_s} \dd [\mu]^{ij}_s\bigg|,
  \end{equation*}
  and (in the current setting replacing $q'$ by $p$)
  \begin{equation*}
    \bigg|\int_0^T \frac{\pi^F_s - \pi^G_s}{\mu_s} \dd \mu_s\bigg| \lesssim p^{\mu,p}_T((\pi^F,\pi^{F,\prime}) - (\pi^G,\pi^{G,\prime}))(\|\mu\|_{p,[0,T]} + \|A^\mu\|_{\frac{p}{2},[0,T]}).
  \end{equation*}
  By the estimate in \eqref{eq: local estimates for generalized functionally generated portfolios}, we obtain
  \begin{equation*}
    \bigg|\int_0^T \frac{\pi^F_s - \pi^G_s}{\mu_s} \dd \mu_s\bigg| \lesssim (1 + \|\mu\|_{p,[0,T]}^2) (\|\mu\|_{p,[0,T]} + \|A^\mu\|_{\frac{p}{2},[0,T]}) \|F - G\|_{C^2}.
  \end{equation*}
  Since $\|F\|_{C^{2+\alpha}} \leq K$ and $\|G\|_{C^{2+\alpha}} \leq K$, recalling \eqref{eq:portfoliogen}, we can verify that
  \begin{equation*}
    \bigg|\frac{(\pi^{F,i}_s - \pi^{G,i}_s)(\pi^{F,j}_s + \pi^{G,j}_s)}{\mu^i_s \mu^j_s}\bigg| \lesssim \|F - G\|_{C^2}.
  \end{equation*}
  Hence, we have that
  \begin{equation*}
    \bigg|\sum_{i,j=1}^d \int_0^T \frac{(\pi^{F,i}_s - \pi^{G,i}_s)(\pi^{F,j}_s + \pi^{G,j}_s)}{\mu^i_s \mu^j_s} \dd [\mu]^{ij}_s\bigg| \lesssim \|F - G\|_{C^2} \sum_{i=1}^d [\mu]^{ii}_T.
  \end{equation*}
  Combining the estimates above, we obtain \eqref{eq: bound for log wealth for generalized functionally generated portfolios}.
\end{proof}

As a special case of Theorem~\ref{thm: Cover's theorem}, we can deduce an asymptotic growth rate for the universal portfolio in the case that our portfolios are restricted to the class $\mathcal{F}^{2+\alpha,K}$ of functionally controlled portfolios.

\smallskip

Let $m$ be a fixed probability measure on $C^{2+\alpha, K} = C^{2+\alpha,K}(\overline{\Delta}^d_+;\R^d)$, and define $\nu := \Phi_* m$ as the pushforward measure on $\mathcal{F}^{2+\alpha,K}$ of $m$ under the map $\Phi$ given in Lemma~\ref{lem: continuity embedding of function spaces}. The universal portfolio based on functionally controlled portfolios is then defined by
\begin{equation}\label{eq: universal portfolio1}
  \pi^\nu_t := \frac{\int_{\mathcal{F}^{2+\alpha,K}} \pi_t V^\pi_t \dd \nu(\pi)}{\int_{\mathcal{F}^{2+\alpha,K}} V^\pi_t \dd \nu(\pi)}, \qquad t \in [0,\infty),
\end{equation}
and the wealth process of the best retrospectively chosen portfolio is defined as
\begin{equation}\label{eq: best retrospective portfolio1}
V^{\ast,K,\alpha}_T := \sup_{\pi \in \mathcal{F}^{2+\alpha,K}} V^\pi_T = \sup_{F \in C^{2+\alpha,K}} V^{\pi^F}_T.
\end{equation}

By Lemma~\ref{lemma continuity log wealth func cont port}, the mapping $F \mapsto V^{\pi^F}_T$ is a continuous map on $C^{2+\alpha,K}$ with respect to the $C^2$-norm. We also have that $C^{2+\alpha,K}$ is compact with respect to the $C^2$-norm (see \cite[Lemma~4.1]{Cuchiero2019}). Combining these two facts, we see that, for each $T > 0$, there exists a function $F^\ast_T \in C^{2+\alpha,K}$ such that
\begin{equation*}
V^{\ast,K,\alpha}_T = V^{\pi^{F^\ast_T}}_T.
\end{equation*}


\begin{theorem}\label{thm: Cover's theorem for generalized functionally generated portfolios}
Let $m$ be a probability measure on $C^{2+\alpha,K}$ with full support. Let $\pi^\nu$ be the universal portfolio as defined in \eqref{eq: universal portfolio1}, and define $V^{\ast,K,\alpha}_T$ as in \eqref{eq: best retrospective portfolio1}.

\begin{enumerate}[(i)]
\item If $\lim_{T \to \infty} (1 + \|\mu\|_{p,[0,T]}^2) \xi_T = \infty$, where as usual $\xi_T$ is defined as in \eqref{eq:defn xi_T}, then
\begin{equation}\label{eq:Cover's theorem func cont part i}
\lim_{T \to \infty} \frac{1}{(1 + \|\mu\|_{p,[0,T]}^2) \xi_T} \Big(\log V^{\ast,K,\alpha}_T - \log V^{\pi^\nu}_T\Big) = 0.
\end{equation}
\item With the shorthand notation $\xi_{k,k+1} := \|\mu\|_{p,[k,k+1]} + \|A^\mu\|_{\frac{p}{2},[k,k+1]} + \sum_{i=1}^d [\mu]^{ii}_{k,k+1}$ for each $k \in \N$, if $\lim_{T \to \infty} \sum_{k=0}^{\lceil T \rceil - 1} (1 + \|\mu\|_{p,[k,k+1]}^2) \xi_{k,k+1} = \infty$, then
\begin{equation}\label{eq:Cover's theorem func cont part ii}
\lim_{T \to \infty} \frac{1}{\sum_{k=0}^{\lceil T \rceil - 1} (1 + \|\mu\|_{p,[k,k+1]}^2) \xi_{k,k+1}} \Big(\log V^{\ast,K,\alpha}_T - \log V^{\pi^\nu}_T\Big) = 0.
\end{equation}
\end{enumerate}
\end{theorem}

\begin{proof}
The result of part~(i) follows from Theorem~\ref{thm: Cover's theorem} applied with $\mathcal{K} = C^{2+\alpha,K}$, $d_{\mathcal{K}}(F,G) = \|F - G\|_{C^2}$, $\iota = \Phi$ and $\lambda(T) = (1 + \|\mu\|_{p,[0,T]}^2) \xi_T$, noting from the result of Lemma~\ref{lemma continuity log wealth func cont port} that the bound in \eqref{eq: lipschitz type estimation for embedding} is indeed satisfied in this case.

The result of part~(ii) follows similarly with $\lambda(T) = \sum_{k=0}^{\lceil T \rceil - 1} (1 + \|\mu\|_{p,[k,k+1]}^2) \xi_{k,k+1}$. That the bound in \eqref{eq: lipschitz type estimation for embedding} is satisfied in this case follows from a very straightforward adaptation of the proofs of Lemmas~\ref{lem: continuity embedding of function spaces} and \ref{lemma continuity log wealth func cont port}, whereby the same estimates are applied over the subinterval $[k,k+1]$ for each $k = 0, \ldots, \lceil T \rceil - 1$, and the integrals over $[0,T]$ in the proof of Lemma~\ref{lemma continuity log wealth func cont port} are trivially bounded by the sum of integrals over these subintervals.
\end{proof}

\begin{remark}
The result of Theorem~\ref{thm: Cover's theorem for generalized functionally generated portfolios} is stated for two different ``clocks'', namely $(1 + \|\mu\|_{p,[0,T]}^2) \xi_T$ and $\sum_{k=0}^{\lceil T \rceil - 1} (1 + \|\mu\|^2_{p,[k,k+1]}) \xi_{k,k+1}$. One may wonder whether one of these clocks always dominates the other, making one of the statements superfluous. However, this is not the case.

On the one hand, in Section~\ref{sect: example of nontriviality of growth rate} below we will exhibit a particular scenario which demonstrates the non-triviality of the growth rate established in \eqref{eq:Cover's theorem func cont part i}. In this setting, one may check that $(1 + \|\mu\|_{p,[0,T]}^2) \xi_T$ gives a strictly better asymptotic rate than if one were to use the sum over a partition of subintervals, as in part~(ii) of Theorem~\ref{thm: Cover's theorem for generalized functionally generated portfolios}.

On the other hand, in Section~\ref{sec:Functionally controlled protfolios in probabilistic models} below we will consider a probabilistic model, where the market portfolio $\mu$ is given by the solution of a stochastic differential equation driven by Brownian motion. Using the fact that Brownian motion has independent increments, and the strong law of large numbers, in Theorem~\ref{thm:complogopt} we will use \eqref{eq:Cover's theorem func cont part ii} to improve the asymptotic growth rate to $T$. That is, we will actually show that, almost surely,
\begin{equation*}
\lim_{T \to \infty} \frac{1}{T} \Big(\log V^{\ast,K,\alpha}_T - \log V^{\pi^\nu}_T\Big) = 0.
\end{equation*}
It is therefore valuable to include both parts of Theorem~\ref{thm: Cover's theorem for generalized functionally generated portfolios}.
\end{remark}

\begin{remark}
Strictly speaking, Theorems~\ref{thm: Cover's theorem} (which also recovers the version of Cover's theorem established in \cite{Cuchiero2019}) and \ref{thm: Cover's theorem for generalized functionally generated portfolios} do not say that the universal portfolio $\pi^\nu$ performs asymptotically as well as the best retrospectively chosen one; rather, they provide bounds on how large the gap can become as time increases. For instance, for classical functionally generated portfolios of the form in \eqref{eq: canonical function generated portfolio} the gap is $o(\max_{i = 1, \ldots, d} [\mu]^{ii}_T)$, and for functionally controlled portfolios of the form in \eqref{eq:portfoliogen} the gap is, e.g.~$o((1 + \|\mu\|_{p,[0,T]}^2) \xi_T)$.
\end{remark}


\subsection{The non-triviality of the asymptotic growth rate}\label{sect: example of nontriviality of growth rate}

In this section we will show that the asymptotic growth rate $\lambda(T) = (1 + \|\mu\|_{p,[0,T]}^2) \xi_T$ for functionally controlled portfolios, as established in part~(i) of Theorem~\ref{thm: Cover's theorem for generalized functionally generated portfolios}, is non-trivial, in the sense that there exists an instance of the market portfolio $\mu = (\mu_t)_{t \in [0,\infty)}$ such that
\begin{equation*}
\limsup_{T \to \infty}\frac{\log V^{\ast,K,\alpha}_T}{(1 + \|\mu\|_{p,[0,T]}^2) \xi_T} > 0 \quad \text{and} \quad \lim_{T \to \infty} \frac{\log V^{\ast,K,\alpha}_T - \log V^{\pi^\nu}_T}{(1 + \|\mu\|_{p,[0,T]}^2) \xi_T} = 0,
\end{equation*}
where $\nu = \Phi_{\ast} m$ for an arbitrary probability measure $m$ on $C^{2+\alpha,K}$ with full support.

\begin{lemma}\label{lemma: asymptotic growth rate is nontrivial}
  Let $p \in (2,3)$ as usual, and then fix $\lambda > 0$ such that $\frac{1}{p} < \lambda < \frac{1}{2}$. Let $d = 3$ and let $\mu = (\mu_t)_{t \in [0,\infty)}$ be the continuous $\Delta^3_+$-valued path given by
  \begin{equation*}
    \mu_t = \left(\hspace{-4pt}\begin{array}{c}
    \mu^1_t\\
    \mu^2_t\\
    \mu^3_t
    \end{array}\hspace{-4pt}\right) = \left(\hspace{-4pt}\begin{array}{c}
    \frac{1}{3} (1 + \frac{k^{-\lambda}}{3} (1 - \cos t))\\
    \frac{1}{3} (1 + \frac{k^{-\lambda}}{3} \sin t)\\
    \frac{1}{3} (1 + \frac{k^{-\lambda}}{3} (\cos t - 1 - \sin t))
    \end{array}\hspace{-4pt}\right), \qquad t \in [2\pi (k-1),2\pi k),
  \end{equation*}
  for each $k \in \N$. For $\alpha \in (0,1]$ and $K > 0$, let $V^{\ast,K,\alpha}_T$ be the wealth induced by the best retrospectively chosen portfolio over $\mathcal{F}^{2+\alpha,K}$ at time $T$. Then
  \begin{equation*}
    \limsup_{T \to \infty}\frac{\log V^{\ast,K,\alpha}_T}{(1 + \|\mu\|_{p,[0,T]}^2) \xi_T} > 0.
  \end{equation*}
\end{lemma}

\begin{proof}
  Recall that for any portfolio~$\pi$, it follows from Proposition~\ref{prop: log of relative wealth} that
  \begin{equation*}
    \log V^\pi_T = \int_0^T \frac{\pi_s}{\mu_s} \dd \mu_s - \frac{1}{2} \sum_{i,j = 1}^d \int_0^T \frac{\pi^i_s \pi^j_s}{\mu^i_s \mu^j_s} \dd [\mu]^{ij}_s.
  \end{equation*}
  Clearly, since $\mu$ is continuous with bounded variation on every compact interval, we have that $[\mu] = 0$, so that the second term vanishes. For any functionally controlled portfolio $\pi^F \in \mathcal{F}^{2+\alpha,K}$, using the relation
  \begin{equation*}
    \frac{\pi^{F,i}_t}{\mu^i_t} = F^i(\mu_t) + 1 - \sum_{j=1}^d \mu^j_t F^j(\mu_t), \qquad i = 1, \ldots, d,
  \end{equation*}
  together with the fact that $\sum_{i=1}^d \d \mu^i_t = 0$ (since $\sum_{i=1}^d \mu^i_t = 1$), we deduce that
  \begin{equation}\label{eq: a simple expression of log wealth}
    \log V^{\pi^F}_T =  \int_0^T \frac{\pi^F_t}{\mu_t} \dd \mu_t = \sum_{i=1}^d \int_0^T \frac{\pi^{F,i}_t}{\mu^i_t} \dd \mu^i_t = \sum_{i=1}^d \int_0^T F^i(\mu_t) \dd \mu^i_t.
  \end{equation}
  We now choose the function $F \in C^{2+\alpha,K}$ given by
  \begin{equation*}
    F(x) = \left(\hspace{-4pt}\begin{array}{c}
    x_2\\
    0\\
    0
    \end{array}\hspace{-4pt}\right)
  \end{equation*}
  for $x = (x_1,x_2,x_3)^\top \in \overline{\Delta}^3_+$. Substituting this function into \eqref{eq: a simple expression of log wealth}, we have
  \begin{equation*}
    \log V^{\pi^F}_T  = \int_0^T \frac{\pi^F_t}{\mu_t} \dd \mu_t = \sum_{i=1}^3 \int_0^T F^i(\mu_t) \dd \mu^i_t = \int_0^T \mu^2_t \dd \mu^1_t.
  \end{equation*}
  For $n \in \N$, we compute
  \begin{align*}
    \int_0^{2\pi n} \mu^2_t \dd \mu^1_t &= \sum_{k=1}^n \int_{2\pi (k-1)}^{2\pi k} \mu^2_t \dd \mu^1_t = \sum_{k=1}^n \int_0^{2\pi} \frac{1}{3} \bigg(1 + \frac{k^{-\lambda}}{3} \sin t\bigg) \cdot \frac{k^{-\lambda}}{9} \sin t \dd t\\
    &= \sum_{k=1}^n \frac{k^{-2\lambda}}{81} \int_0^{2\pi} \sin^2 t \dd t = \frac{\pi}{81} \sum_{k=1}^n k^{-2\lambda},
  \end{align*}
  and note that
  \begin{equation*}
    \|\mu\|_{p,[0,2\pi n]} \lesssim \bigg(\sum_{k=1}^n k^{-\lambda p}\bigg)^{\hspace{-2pt}\frac{1}{p}} < \bigg(\sum_{k=1}^\infty k^{-\lambda p}\bigg)^{\hspace{-2pt}\frac{1}{p}} < \infty
  \end{equation*}
  for every $n \in \N$.

  Writing $A^\mu_{s,t} = [A^{\mu,i,j}_{s,t}]_{i, j = 1, 2, 3} = \int_s^t (\mu_u - \mu_s) \otimes \d \mu_u$ for the canonical rough path lift of $\mu$, and using the monotonicity of trigonometric functions on the intervals $[0, \frac{\pi}{2}]$, $[\frac{\pi}{2}, \pi]$, $[\pi, \frac{3\pi}{2}]$ and $[\frac{3\pi}{2}, 2\pi]$, one can readily check that
  \begin{equation*}
    \|A^\mu\|_{\frac{p}{2},[0,2\pi n]} \sim A^{\mu,2,1}_{0,2\pi n} = \int_0^{2\pi n} \mu^2_t \dd \mu^1_t \sim \sum_{k=1}^n k^{-2\lambda}.
  \end{equation*}
  Recalling that $\xi_T = \|\mu\|_{p,[0,T]} + \|A^\mu\|_{\frac{p}{2},[0,T]}$ (since $[\mu] = 0$), and combining the calculations above, we deduce that, for $T = 2\pi n$,
  \begin{equation*}
    \frac{\log V^{\ast,K,\alpha}_T}{(1 + \|\mu\|_{p,[0,T]}^2) \xi_T} \geq \frac{\log V^{\pi^F}_{2\pi n}}{(1 + \|\mu\|_{p,[0,2\pi n]}^2) \xi_{2\pi n}} \gtrsim \frac{\sum_{k=1}^n k^{-2\lambda}}{1 + \sum_{k=1}^n k^{-2\lambda}} \, \longrightarrow \, 1 \qquad \text{as} \quad n \to \infty,
  \end{equation*}
  where we used the fact that $2\lambda < 1$.
\end{proof}

The example in Lemma~\ref{lemma: asymptotic growth rate is nontrivial} thus shows that for functionally controlled portfolios $\pi^F$ generated by a function $F \in C^{2+\alpha,K}$ which is not necessarily of gradient-type, the asymptotic growth rate $(1 + \|\mu\|_{p,[0,T]}^2) \xi_T$ appearing in Theorem~\ref{thm: Cover's theorem for generalized functionally generated portfolios} is actually sharp, in the sense that the log-relative wealth $\log V^{\pi^F}_T$ and the rate $(1 + \|\mu\|_{p,[0,T]}^2) \xi_T$ grow at the same rate (up to a multiplicative constant) as $T \to \infty$.

\subsection{Functionally controlled portfolios have better performance}

Let us conclude this section by showing that classical functionally generated portfolios of form in~\eqref{eq: canonical function generated portfolio}, which are induced by functions of gradient type, are in general not optimal among the class of functionally controlled portfolios of the form in~\eqref{eq:portfoliogen}.

Let $\mu$ be a continuous $\Delta^d_+$-valued path which, for simplicity, we assume to have finite variation on every bounded interval (and which therefore trivially satisfies Property~\textup{(RIE)}). For any $F \in C^{2+\alpha,K}(\overline{\Delta}^d_+;\R^d)$, we know, as we saw in \eqref{eq: a simple expression of log wealth} above, that for every $T > 0$,
\begin{equation*}
  \log V^{\pi^F}_T = \int_0^T \frac{\pi^F_s}{\mu_s} \dd \mu_s - \frac{1}{2} \sum_{i,j = 1}^d \int_0^T \frac{\pi^{F,i}_s \pi^{F,j}_s}{\mu^i_s \mu^j_s} \dd [\mu]^{ij}_s = \int_0^T F(\mu_s) \dd \mu_s,
\end{equation*}
since the quadratic variation $[\mu]$ vanishes. Suppose now that the generating function~$F$ were of gradient-type, so that $F = \nabla f$ for some suitably smooth real-valued function~$f$. We then have that
\begin{equation*}
  \log V^{\pi^F}_T = \int_0^T \nabla f(\mu_s) \dd \mu_s = f(\mu_T) - f(\mu_0),
\end{equation*}
which implies together with the mean value theorem that
\begin{equation*}
  |\log V^{\pi^F}_T| \leq \|\nabla f\|_\infty |\mu_T - \mu_0| = \|F\|_\infty |\mu_T - \mu_0| \leq 2K,
\end{equation*}
as $\|F\|_\infty \leq K$ and $\mu_T, \mu_0 \in \Delta^d_+$. In particular, we have that
\begin{equation}\label{eq:bounded wealth for gradients}
  \sup_{T \geq 0} \, \log V^{\pi^F}_T \leq 2K < \infty
\end{equation}
for every generating function~$F$ of gradient type.

\smallskip

Now let $\mu$ be the market portfolio given in Lemma~\ref{lemma: asymptotic growth rate is nontrivial}, and let $F(x_1,x_2,x_3) = (x_2,0,0)^\top$, which we note is \emph{not} of gradient type. In the proof of Lemma~\ref{lemma: asymptotic growth rate is nontrivial} we saw, for $T = 2\pi n$ with any $n \in \N$, that $\log V^{\pi^F}_T = \int_0^T \mu^2_t \dd \mu^1_t = \frac{\pi}{81} \sum_{k=1}^n k^{-2\lambda}$ for some positive $\lambda <\frac{1}{2}$. We thus immediately have that
\begin{equation}\label{eq:infinite asymp wealth for non-gradient}
  \limsup_{T \to \infty} \, \log V^{\pi^F}_T = \infty.
\end{equation}
Comparing \eqref{eq:infinite asymp wealth for non-gradient} with \eqref{eq:bounded wealth for gradients}, it is clear that the best retrospectively chosen portfolio over the set of functionally controlled portfolios cannot be of gradient type. Indeed, we infer that among the class of all functionally controlled portfolios, those corresponding to gradient-type generating functions are in general far from being optimal, demonstrating the need to go beyond gradient-type generating functions.

\section{Functionally controlled portfolios in probabilistic models}\label{sec:Functionally controlled protfolios in probabilistic models}

In this section we shall demonstrate some further links between our purely pathwise theory and classical stochastic portfolio theory in a probabilistic setting. In particular, this will allow us to illustrate again the advantages of functionally controlled portfolios, as introduced in Example~\ref{ex: functionally controlled portfolios}, compared to (pathwise) functionally generated portfolios (see Lemma~\ref{lem: functionally generated portfolios are in base set}), as were previously treated in \cite{Schied2018,Cuchiero2019} based on F{\"o}llmer integration.

\subsection{Probabilistic model for the market portfolio}

Whereas in the previous sections we worked in a purely pathwise setting, we now assume that the market portfolio (also known as the market weights process) $\mu = (\mu^1_t, \ldots, \mu_t^d)_{t \in [0,\infty)}$ is described by a time-homogeneous Markovian It{\^o}-diffusion with values in $\Delta_+^d$, of the form
\begin{equation}\label{eq:muP}
  \mu_t = \mu_0 + \int_0^t c(\mu_s) \lambda(\mu_s) \dd s + \int_0^t \sqrt{c(\mu_s)} \dd W_s, \qquad t \in [0,\infty),
\end{equation}
where $\mu_0$ is distributed according to some measure $\rho$ on $\Delta_+^{d}$, $W$ is a $d$-dimensional Brownian motion and $\sqrt{\hspace{1pt}\cdot\hspace{1pt}}$ denotes the matrix square root. We assume that $\mu$ is the canonical process defined on path space ($\Omega, \mathcal{F}, \mathbb{P})$, i.e.~$\Omega = C([0,\infty);\Delta^d_+)$, $\mathcal{F} = \sigma(\mu_t : t \in [0, \infty))$, and $\mathbb{P}$ denotes the law of $\mu$. For the moment $\lambda$ is just assumed to be a Borel measurable function from $\Delta_+^d$ to $\R^d$. Writing $\mathbb{S}_+^d$ for the set of positive semi-definite symmetric matrices, $c \in C(\overline{\Delta}_+^d;\mathbb{S}_+^d)$ is such that
\begin{equation*}
  c(x) \mathbf{1} = 0 \qquad \text{for all} \quad x \in \Delta_+^d.
\end{equation*}
The latter requirement is necessary to guarantee that the process~$\mu$ lies in $\Delta_+^d$. For a complete characterization of stochastic invariance of the closed simplex (under additional regularity conditions on the coefficients~$\lambda$ and~$c$) we refer to \cite[Theorem~2.3]{AbiJaber2017} and the references therein. To ensure that the process stays in the open simplex $\Delta_+^d$, conditions for non-attainment of the boundary are established for instance in \cite[Theorem~5.7]{Filipovic2016}. These conditions build on versions of what is sometimes called ``McKean’s argument'' (see \cite{Mayerhofer2011a} for an overview and further references).

We further suppose that the so-called \emph{structure condition} is satisfied, that is
\begin{equation}\label{eq: condition for NUPBR}
  \int_0^T \lambda^\top (\mu_s) c(\mu_s) \lambda(\mu_s) \dd s < \infty \quad \mathbb{P}\text{-a.s.}, \quad \text{for all} \quad T \in [0,\infty),
\end{equation}
which is equivalent to  ``no unbounded profit with bounded risk'' (NUPBR); see e.g.~Theorem~3.4 in \cite{Hulley2010}.

\begin{remark}
  As (NUPBR) is satisfied due to \eqref{eq: condition for NUPBR}, the sample paths of $\mu$ almost surely satisfy Property \textup{(RIE)} with respect to every $p \in (2,3)$ and a suitable sequence of partitions, cf.~Remark~\ref{remark: financial models satisfy RIE}.
\end{remark}

We further impose the following ergodicity assumption in the spirit of \cite[Section~2.2, Theorem~2.6 and Section~2.2.3, Theorem~2.8]{E:16}, along with an integrability condition on $\lambda$.

\begin{assumption}\label{ass:1}
  We assume that the market portfolio $\mu$, given by the dynamics in \eqref{eq:muP}, is an ergodic process with stationary measure $\rho$ on $\Delta_+^d$. That is, we suppose that $\rho p_t = \rho$ for every $t \in [0,\infty)$, where here $(p_t)_{t \in [0,\infty)}$ denotes the transition probability of $\mu$. Furthermore, we suppose that $\lambda \in L^2(\Delta^d_+,\rho;\mathbb{R}^d)$.
\end{assumption}

Note that the assumption that $\rho$ is a stationary measure implies that the shift semigroup $\Theta_t (\omega) = \omega(t + \cdot)$, $t \in [0, \infty)$, $\omega \in \Omega$, preserves the measure $\mathbb{P}$, in the sense that $\mathbb{P} \circ \Theta_t^{-1} = \mathbb{P}$. Hence, the ``ergodic theorem in continuous time'' (see \cite[Section~2.2, Theorem~2.6, Theorem~2.8]{E:16}) can be applied.

\smallskip

While on the pathwise market $\Omega_p$ the portfolios were given by $\mu$-controlled paths $(\pi,\pi') \in \mathcal{V}^q_{\mu}$ (recall Definition~\ref{def: controlled path}), in the present semimartingale setting we consider a portfolio $\pi$ to be an element of the set $\Pi$ of all predictable processes $\pi$ taking values in $\Delta^d$, such that the It{\^o} integral
\begin{equation*}
  \int_0^T \frac{\pi_s}{\mu_s} \dd \mu_s = \int_0^T \sum_{i=1}^d \frac{\pi^i_s}{\mu_s^i} \dd \mu_s^i
\end{equation*}
is well-defined for every $T \in [0,\infty)$. As established in \cite[Section~4.2.3]{Cuchiero2019}, for $\pi \in \Pi$, the relative wealth process (recall \eqref{eq:defn relative wealth process}) can be written in the usual form, that is
\begin{align}\label{eq:portfolio}
  V^\pi_T &= \exp \bigg(\int_0^T \frac{\pi_s}{\mu_s} \dd \mu_s - \frac{1}{2} \int_0^T \sum_{i,j=1}^d \frac{\pi^i_s \pi^j_s}{\mu_s^i \mu_s^j} c^{ij}(\mu_s) \dd s\bigg), \qquad T \in [0,\infty).
\end{align}

\begin{remark}
  Note that if $(\pi,\pi^\prime)$ is an adapted process with sample paths which are almost surely $\mu$-controlled paths, then it is predictable, and under Property \textup{(RIE)} the rough integral interpretation of $\int_0^T \frac{\pi_s}{\mu_s} \dd \mu_s$ coincides almost surely with the It{\^o} integral interpretation. Indeed, the rough integral can be approximated by left-point Riemann sums (see Theorem~\ref{thm: Ito integral for smooth transformed RIE path}), while the It{\^o} integral can be approximated by the same Riemann sums in probability (see e.g.~\cite[Theorem~II.21]{Protter2004}). Moreover, as established in Proposition~\ref{prop: log of relative wealth}, the identity in \eqref{eq:portfolio} holds even in a pathwise setting.
\end{remark}

\subsection{The log-optimal portfolio and equivalence of its asymptotic growth rate with Cover's universal and the best retrospectively chosen portfolio}\label{subsec:assgrowth}

The results in this section will illustrate that in the presence of an appropriate probabilistic structure the asymptotic growth rate can be significantly improved for scenarios outside a null set.

\smallskip

For a given $T > 0$, the \emph{log-optimal portfolio} $\widehat{\pi}$ is the maximizer of the optimization problem
\begin{equation}\label{eq:definition log-optimal portfolio}
  \sup_{\pi \in \Pi} \mathbb{E} [\log V^\pi_T].
\end{equation}
We write
\begin{equation*}
  \widehat{V}_T := V_T^{\widehat{\pi}}
\end{equation*}
for the corresponding wealth process. As shown in \cite[Section~4.2.3]{Cuchiero2019}, if $\mu$ satisfies the dynamics in \eqref{eq:muP}, then $\widehat{\pi} = (\widehat{\pi}^1, \dots, \widehat{\pi}^d)$ can be expressed as
\begin{align}\label{eq:logopt}
  \widehat{\pi}_t^i = \mu^i_{t} \bigg(\lambda^i(\mu_t) + 1 - \sum_{j=1}^d \mu^j_{t} \lambda^j(\mu_t)\bigg), \qquad t \in [0,\infty),
\end{align}
and, due to \eqref{eq:portfolio}, the expected value of the log-optimal portfolio satisfies
\begin{align}\label{eq:expectedlogopt}
  \mathbb{E} [\log \widehat{V}_T] = \sup_{\pi \in \Pi} \mathbb{E} [\log V^\pi_T] = \frac{1}{2} \mathbb{E} \bigg[\int_0^T \lambda^\top (\mu_s) c(\mu_s) \lambda(\mu_s) \dd s\bigg].
\end{align}
We suppose that the log-optimal portfolio has finite maximal expected utility and require thus additionally to \eqref{eq: condition for NUPBR} that 
\begin{equation*}
  \mathbb{E} \bigg[\int_0^T \lambda^\top (\mu_s) c(\mu_s) \lambda(\mu_s) \dd s\bigg]< \infty.
\end{equation*}
From the expression in \eqref{eq:logopt}, we see immediately that the log-optimal portfolio $\widehat{\pi}$ belongs to the class of functionally controlled portfolios, as defined in Example~\ref{ex: functionally controlled portfolios}, whenever $\lambda$ is sufficiently smooth. In general, however, it does \emph{not} belong to the smaller class of functionally generated portfolios, as we will see in Section~\ref{subsec:comp}.

\smallskip

In \eqref{eq:definition log-optimal portfolio}, the supremum is taken over \emph{all} predictable strategies in $\Pi$. However, since the optimizer is actually of the form in \eqref{eq:logopt}, we can also take the supremum in~\eqref{eq:definition log-optimal portfolio} over a smaller set. Indeed, it is sufficient to consider (functionally controlled) portfolios of the form
\begin{equation}\label{eq:genportLinfty}
  (\pi^F_t)^i = \mu_t^i \bigg(F^i(\mu_t) + 1 - \sum_{j=1}^d \mu_t^j F^j(\mu_t)\bigg),
\end{equation}
for functions $F$ in the space $L^2(\Delta_+^d, \rho; \R^d)$. 

Clearly, any portfolio $\pi^F$ of the form in \eqref{eq:genportLinfty} can itself be considered as a function $\pi^F \in L^2(\Delta_+^d, \rho;\R^d)$ which maps $x \mapsto \pi^F(x)$, where
\begin{equation}\label{eq:pi^F(x) defn}
  [\pi^F(x)]^i = x^i \bigg(F^i(x) + 1 - \sum_{j=1}^d x^j F^j(x)\bigg),
\end{equation}
with the corresponding portfolio then being given by $t \mapsto \pi^F(\mu_t)$.

In the current probabilistic setting we establish the following equivalence of the asymptotic growth rates of the log-optimal, best retrospectively chosen and the universal portfolio based on functionally controlled portfolios of the form in \eqref{eq:genportLinfty}, which can be viewed as a generalization of \cite[Theorem~4.12]{Cuchiero2019} for non-functionally generated portfolios.

\begin{theorem}\label{thm:complogopt}
  Let $\mu$ be a market weights process with the dynamics in \eqref{eq:muP}.
  \begin{enumerate}
    \item[(i)] Suppose that $\mu$ and $\lambda$ satisfy Assumption~\ref{ass:1}, and that $c \in C(\overline{\Delta}_+^d;\mathbb{S}^d_+)$. Let $m$ be a probability measure on $L^2(\Delta_+^d,\rho; \mathbb{R}^d)$ such that $\lambda \in \operatorname{supp}(m)$. Define the universal portfolio $\pi^{\nu}$ analogously to \eqref{eq: universal portfolio1} but with $\nu$ being the pushforward measure of $m$ under the mapping $F \, \mapsto \, \pi^F$ with $\pi^F$ as in \eqref{eq:pi^F(x) defn}, cf.~\cite[Section~4.2.2]{Cuchiero2019}. Suppose that there exists an integrable random variable $w$ such that, for each $T > 0$, the growth rate of the universal portfolio satisfies
    \begin{equation}\label{eq:technical}
      \frac{1}{T} \log V^{\pi^\nu}_T \geq -w.
    \end{equation}
    We then have that
    \begin{equation}\label{eq: asymptotic rate for log optimal and universal}
      \liminf_{T\to \infty} \frac{1}{T} \log V^{\pi^\nu}_T = \lim_{T \to \infty} \frac{1}{T} \log \widehat{V}_T = \widehat{L}, \qquad \mathbb{P}\text{-a.s.},
    \end{equation}
    where $\widehat{L} $ is given by
    \begin{equation*}
      \widehat{L} := \frac{1}{2} \int_{\Delta^d_+} \lambda^{\top}(x) c(x) \lambda(x) \, \rho(\d x).
    \end{equation*}
    \item[(ii)] Suppose that
    \begin{equation}\label{eq: smooth conditions}
      \lambda \in C^{3}_b(\overline{\Delta}_+^d;\R^d), \quad \text{and} \quad \sqrt{c} \in C^{3}_b(\overline{\Delta}_+^d;\mathbb{S}_+^d).
    \end{equation}
    With the same notation as in Section~\ref{subsection universal portfolios functionally controlled}, let $m$ be a probability measure on $C^{2+\alpha,K}$ with full support, and let $\nu = \Phi_\ast m$ be the pushforward measure on $\mathcal{F}^{2+\alpha,K}$ of $m$ under the map $\Phi$ given in Lemma~\ref{lem: continuity embedding of function spaces}. Let $\pi^\nu$ be the universal portfolio as defined in \eqref{eq: universal portfolio1}, and let $V^{\ast,K,\alpha}$ be the wealth process of the best retrospectively chosen portfolio, as in \eqref{eq: best retrospective portfolio1}. We then have that
    \begin{equation}\label{eq: asymptotic rate for universal and best, stochastic}
      \lim_{T \to \infty} \frac{1}{T} \Big(\log V^{\ast,K,\alpha}_T - \log V^{\pi^\nu}_T\Big) = 0, \qquad \mathbb{P}\text{-a.s.}
    \end{equation}
    \item[(iii)] Suppose that $\mu, \lambda$ and $c$ satisfy both Assumption~\ref{ass:1} and \eqref{eq: smooth conditions}, and that $K > 0$ is sufficiently large to ensure that $\lambda \in C^{2+\alpha,K}$. Let $m, \nu, \pi^\nu$ and $V^{\ast,K,\alpha}$ be as in part~(ii) above. Then
    \begin{equation}\label{eq: equivalence for three convergence rates}
      \liminf_{T \to \infty} \frac{1}{T} \log V^{\pi^\nu}_T = \liminf_{T \to \infty} \frac{1}{T} \log V^{\ast,K,\alpha}_T = \lim_{T \to \infty} \frac{1}{T} \log \widehat{V}_T = \widehat{L}, \qquad \mathbb{P}\text{-a.s.}
    \end{equation}
  \end{enumerate}
\end{theorem}

\begin{remark}
Note that the assumption of ergodicity in Assumption~\ref{ass:1} is only needed for assertions (i) and (iii). The equivalence of the asymptotic growth rates of the best retrospectively chosen and  Cover's universal portfolio, as established in part (ii), holds for all Brownian driven SDEs with sufficiently smooth coefficients.
\end{remark}

As preparation for the proof of Theorem~\ref{thm:complogopt}, we need the following technical lemma, which is an adaptation of \cite[Lemma~3.1]{Hubalek2002}.

\begin{lemma}\label{lem:universaldomination}
  Let $(f_n)_{n \in \mathbb{N}}$ be a sequence of non-negative measurable functions on some topological space $\mathcal{A}$, such that the map $a \mapsto \liminf_{n \to \infty} f_n(a)$ is continuous at some point $\widehat{a} \in \mathcal{A}$. Let $\nu$ be a probability measure on $\mathcal{A}$ with $\widehat{a} \in \operatorname{supp}(\nu)$. Then
  \begin{equation*}
    \liminf_{n \to \infty} f_n(\widehat{a}) \leq \liminf_{n \to \infty} \bigg(\int_{\mathcal{A}} f_n^n(a) \, \nu(\d a)\bigg)^{\frac{1}{n}}.
  \end{equation*}
\end{lemma}

\begin{proof}
  Let $g \geq 0$ be bounded measurable function such that $\int_{\mathcal{A}} g(a) \, \nu(\d a) = 1$. By Fatou's lemma and H{\"o}lder's inequality,
  \begin{align*}
    \int_{\mathcal{A}} &\liminf_{n \to \infty} f_n(a) g(a) \, \nu(\d a) \leq \liminf_{n \to \infty} \int_{\mathcal{A}} f_n(a) g(a) \, \nu(\d a)\\
    &\leq \liminf_{n \to \infty} \bigg(\int_{\mathcal{A}} f_n^n(a) \, \nu(\d a)\bigg)^{\frac{1}{n}} \bigg(\int_{\mathcal{A}} g^{\frac{n}{n-1}}(a) \, \nu(\d a)\bigg)^{\frac{n-1}{n}} = \liminf_{n \to \infty} \bigg(\int_{\mathcal{A}} f_n^n(a) \, \nu(\d a)\bigg)^{\frac{1}{n}},
  \end{align*}
  where the last equality follows from the fact that $\lim_{n \to \infty} \int_{\mathcal{A}} g^{\frac{n}{n-1}} \, \nu(\d a) = \int_{\mathcal{A}} g(a) \, \nu(\d a)$ by the dominated convergence theorem. Since $g$ was arbitrary, $\widehat{a}$ lies in the support of $\nu$, and $\liminf_{n \to \infty} f_n$ is continuous at $\widehat{a}$, we deduce the result.
\end{proof}

\begin{proof}[Proof of Theorem~\ref{thm:complogopt}]
  \emph{Part~(i):}
  By the conditions on $\lambda$ and $c$, and the fact that we consider portfolios of the form in \eqref{eq:genportLinfty} with $F \in L^2(\Delta^d_+, \rho; \mathbb{R}^d)$, we see that the assumptions of \cite[Theorem~4.9]{Cuchiero2019} are satisfied. Thus, for each $F \in L^2(\Delta^d_+, \rho; \mathbb{R}^d)$, we have that
  \begin{equation}\label{eq:log wealth converges}
    \lim_{T \to \infty} \frac{1}{T} \log V^{\pi^F}_T = L^{\pi^F}, \quad \mathbb{P}\text{-a.s.},
  \end{equation}
  where
  \begin{equation*}
    L^{\pi^F} := \int_{\Delta^d_+} \bigg(\frac{\pi^F(x)}{x}\bigg)^{\hspace{-2pt}\top} c(x) \lambda(x) \, \rho(\d x) - \frac{1}{2} \int_{\Delta^d_+} \bigg(\frac{\pi^F(x)}{x}\bigg)^{\hspace{-2pt}\top} c(x) \bigg(\frac{\pi^F(x)}{x}\bigg) \, \rho(\d x).
  \end{equation*}
  Taking the supremum over $F \in L^2(\Delta^d_+, \rho; \mathbb{R}^d)$, we find that
  \begin{equation*}
    \sup_{F \in L^2(\Delta^d_+, \rho; \mathbb{R}^d)} L^{\pi^F} = L^{\pi^\lambda} = \widehat{L}.
  \end{equation*}
  Recalling~\eqref{eq:logopt} and \eqref{eq:log wealth converges}, it follows that, $\mathbb{P}$-a.s.,
  \begin{equation}\label{eq:log V hat conv L hat}
    \lim_{T \to \infty} \frac{1}{T} \log \widehat{V}_T = \lim_{T \to \infty} \frac{1}{T} \log V^{\pi^\lambda}_T = L^{\pi^\lambda} = \widehat{L}.
  \end{equation}
  Note that the map
  \begin{equation*}
    F \, \mapsto \, \exp(L^{\pi^{F}}) = \lim_{T \to \infty} \Big(V^{\pi^F}_T\Big)^{\frac{1}{T}}
  \end{equation*}
  is continuous with respect to the $L^2(\Delta^d_+, \rho; \mathbb{R}^d)$-norm. Thus, applying Lemma~\ref{lem:universaldomination} with $f_T(F) = (V^{\pi^F}_T)^{\frac{1}{T}}$, and recalling Lemma~\ref{lemma wealth of universal portfolio}, we deduce that
  \begin{equation}\label{eq:pathwiselim1}
    \lim_{T \to \infty} \frac{1}{T} \log V^{\pi^\lambda}_T \leq \liminf_{T \to \infty} \frac{1}{T} \log V^{\pi^\nu}_T, \quad \mathbb{P}\text{-a.s.}
  \end{equation}
  On the other hand, by the definition of the log-optimal portfolio,
  \begin{equation}\label{eq:expectdom1}
    \mathbb{E} [\log V^{\pi^\nu}_T] \leq \mathbb{E} [\log \widehat{V}_T].
  \end{equation}
  By \eqref{eq:expectedlogopt} and the ergodicity of the process $\mu$, we have that
  \begin{equation}\label{eq:expected log-opt wealth converges}
    \lim_{T \to \infty} \frac{1}{T} \mathbb{E} [\log \widehat{V}_T] = \widehat{L}.
  \end{equation}
  By Fatou's lemma (which we may apply by the condition in \eqref{eq:technical}), \eqref{eq:expectdom1}, \eqref{eq:expected log-opt wealth converges}, \eqref{eq:log V hat conv L hat} and \eqref{eq:pathwiselim1}, we then have that, $\mathbb{P}$-a.s.,
  \begin{align*}
    \mathbb{E} \bigg[\liminf_{T \to \infty} \frac{1}{T} \log V^{\pi^\nu}_T\bigg] &\leq \liminf_{T \to \infty} \frac{1}{T} \mathbb{E} [\log V^{\pi^\nu}_T] \leq \liminf_{T \to \infty} \frac{1}{T} \mathbb{E} [\log \widehat{V}_T]\\
    &= \widehat{L} = \lim_{T \to \infty} \frac{1}{T} \log \widehat{V}_T \leq \liminf_{T \to \infty} \frac{1}{T} \log V^{\pi^\nu}_T,
  \end{align*}
  from which the result \eqref{eq: asymptotic rate for log optimal and universal} follows.

  \smallskip

  \emph{Part~(ii):}
  The process $\mu$ is assumed to satisfy the It\^o SDE \eqref{eq:muP}, but since the vector fields $\lambda(\cdot) c(\cdot)$ and $\sqrt{c}(\cdot)$ are in $C^3$ with bounded derivatives, $\mu$ also coincides almost surely with the unique solution of the rough differential equation
  \begin{equation*}
    \mu_t = \mu_0 + \int_0^t c(\mu_s) \lambda(\mu_s) \dd s + \int_0^t \sqrt{c(\mu_s)} \dd \bW_s,
  \end{equation*}
  driven by the standard It\^o-rough path lift $\bW = (W,\W)$ of $W$ (see e.g.~\cite{Friz2020}). By standard rough path estimates (see e.g.~\cite[(11.10)]{Friz2020}), for each $k \in \N$, we may deduce an estimate of the form
  \begin{equation*}
    \|\mu\|_{p,[k,k+1]} \lesssim 1 + (\|\bW\|_{p,[k,k+1]} \vee \|\bW\|_{p,[k,k+1]}^p),
  \end{equation*}
  where $\|\bW\|_{p,[k,k+1]} := \|W\|_{{p,[k,k+1]}} + \|\W\|_{{\frac{p}{2},[k,k+1]}}^{\frac{1}{2}}$, and the implied multiplicative constant is independent of $k$ and $T$. Using the bound in \eqref{eq:est int of controlled paths}, a similar estimate can be inferred for the rough path lift $A^\mu$ of $\mu$, defined as in \eqref{eq:canonical lift of Z}. Writing $\text{tr}(\cdot)$ for the trace operator, it also follows from Lemma~\ref{lem: Ito isometry for rough paths} and the boundedness of $c$ that
  \begin{equation*}
    \sum_{i=1}^d [\mu]^{ii}_{k,k+1} = \text{tr} \bigg(\int_k^{k+1} c(\mu_t)  \dd [\bW]_t\bigg) = \int_k^{k+1} \text{tr} (c(\mu_t)) \dd t \lesssim 1,
  \end{equation*}
  where we used that $[\bW]_t = t I_d$ as shown e.g.~in \cite[Example 5.9]{Friz2020}.
  We therefore deduce the existence of a polynomial $g$ such that
  \begin{equation}\label{eq:clock bounded by polynomial}
    (1 + \|\mu\|_{p,[k,k+1]}^2) \xi_{k,k+1} \leq g(\|\bW\|_{p,[k,k+1]})
  \end{equation}
  for every $k \in \N$, with $\xi_{k,k+1}$ defined as in Theorem~\ref{thm: Cover's theorem for generalized functionally generated portfolios}.

  Since Brownian motion is a L\'evy process, the random variables $g(\|\bW\|_{p,[k,k+1]})$, $k \in \N$, are independent and identically distributed. Moreover, by the enhanced Burkholder--Davis--Gundy inequality\footnote{Strictly speaking, the enhanced BDG inequality was proved for geometric rough paths constructed via Stratonovich integration. However, since $[\bW]_t = t I_d$, it is easy to see that it also holds for the It\^o lift $\bW$.} (see \cite[Theorem~14.12]{Friz2010}) applied to each of the monomials comprising $g$, we have that $\E[g(\|\bW\|_{p,[0,1]})] < \infty$. Thus, by the strong law of large numbers, we have that, almost surely,
  \begin{equation}\label{eq:SLLN convergence of sum of poly}
    \frac{1}{T} \sum_{k=0}^{\lceil T \rceil - 1} g(\|\bW\|_{p,[k,k+1]}) \, \longrightarrow \, \E[g(\|\bW\|_{p,[0,1]})] \qquad \text{as} \quad T \, \longrightarrow \, \infty.
  \end{equation}
  From \eqref{eq:clock bounded by polynomial}, \eqref{eq:SLLN convergence of sum of poly} and the result of part~(ii) of Theorem~\ref{thm: Cover's theorem for generalized functionally generated portfolios}, we then deduce that, almost surely,
  \begin{align*}
    &\limsup_{T \to \infty} \frac{1}{T} \Big(\log V^{\ast,K,\alpha}_T - \log V^{\pi^\nu}_T\Big)\\
    &\leq \limsup_{T \to \infty} \frac{\sum_{k=0}^{\lceil T \rceil - 1} g(\|\bW\|_{p,[k,k+1]})}{T} \cdot \frac{\log V^{\ast,K,\alpha}_T - \log V^{\pi^\nu}_T}{\sum_{k=0}^{\lceil T \rceil - 1} (1 + \|\mu\|_{p,[k,k+1]}^2) \xi_{k,k+1}} = 0,
  \end{align*}
  which immediately implies \eqref{eq: asymptotic rate for universal and best, stochastic}.

  \smallskip

  \emph{Part~(iii):}
  We have from part~(ii) that \eqref{eq: asymptotic rate for universal and best, stochastic} holds. It is straightforward to check that the result of part~(i) also holds when we restrict to portfolios generated by functions $F \in C^{2+\alpha,K}$. Thus, it suffices to verify the technical condition in \eqref{eq:technical}, since then part~(i) implies that \eqref{eq: asymptotic rate for log optimal and universal} holds, which, combined with \eqref{eq: asymptotic rate for universal and best, stochastic}, gives \eqref{eq: equivalence for three convergence rates}.

  To this end, we first note that, similarly to the proof of part~(ii) above, we may deduce that there exists a polynomial $g$ such that, for any $F \in C^{2+\alpha,K}$,
  \begin{equation*}
    |\log V^{\pi^F}_T| \leq \|F\|_{C^2} \sum_{k=0}^{\lceil T \rceil - 1} g(\|\bW\|_{p,[k,k+1]})
  \end{equation*}
  for all $T > 0$. In particular, we have that
  \begin{equation*}
    \log V^{\pi^F}_T \geq - K \sum_{k=0}^{\lceil T \rceil - 1} g(\|\bW\|_{p,[k,k+1]}).
  \end{equation*}
  Since, by Lemma~\ref{lemma wealth of universal portfolio}, $V^{\pi^\nu}_T = \int_{C^{2+\alpha,K}} V^{\pi^F}_T \dd m(F)$, and using Jensen's inequality, we then have
  \begin{equation*}
    \frac{1}{T} \log V^{\pi^\nu}_T \geq \frac{1}{T} \int_{C^{2+\alpha,K}} \log V^{\pi^F}_T \dd m(F) \geq -\frac{K}{T} \sum_{k=0}^{\lceil T \rceil - 1} g(\|\bW\|_{p,[k,k+1]}),
  \end{equation*}
and, again by the strong law of large numbers, \eqref{eq:SLLN convergence of sum of poly} holds almost surely. It is also straightforward to verify that
  \begin{equation*}
    \bigg(\frac{1}{T} \sum_{k=0}^{\lceil T \rceil - 1} g(\|\bW\|_{p,[k,k+1]})\bigg)^2 \leq \frac{\lceil T \rceil}{T^2} \sum_{k=0}^{\lceil T \rceil - 1} g(\|\bW\|_{p,[k,k+1]})^2,
  \end{equation*}
  so that, for all $T > 1$,
  \begin{equation*}
    \E\bigg[\bigg(\frac{1}{T} \sum_{k=0}^{\lceil T \rceil - 1} g(\|\bW\|_{p,[k,k+1]})\bigg)^2\bigg] \leq \frac{\lceil T \rceil^2}{T^2} \E[g(\|\bW\|_{p,[0,1]})^2] \leq 4\E[g(\|\bW\|_{p,[0,1]})^2] < \infty.
  \end{equation*}
  We deduce that the family $\frac{1}{T} \sum_{k=0}^{\lceil T \rceil - 1} g(\|\bW\|_{p,[k,k+1]})$ for $T > 1$ is bounded in $L^2(\Omega,\mathbb{P})$, and therefore uniformly integrable. Thus, $\frac{1}{T} \sum_{k=0}^{\lceil T \rceil - 1} g(\|\bW\|_{p,[k,k+1]}) \to \E[g(\|\bW\|_{p,[0,1]})]$ as $T \to \infty$ both almost surely and in $L^1(\Omega,\mathbb{P})$. It follows that
  \begin{equation*}
    \frac{1}{T} \log V^{\pi^\nu}_T \geq -w_T,
  \end{equation*}
  for some random variables $w_T$, $T > 0$, which converge as $T \to \infty$ to an integrable random variable $w$ almost surely and in $L^1(\Omega,\mathbb{P})$. Although weaker than the condition in \eqref{eq:technical}, it is straightforward to verify that this condition suffices, as it is sufficient for the application of Fatou's lemma in the proof of part~(i).
\end{proof}


\subsection{Comparison of functionally controlled and functionally generated portfolios}\label{subsec:comp}

Recall that, as we observed from the expression in \eqref{eq:logopt}, the log-optimal portfolio $\widehat{\pi}$ belongs to the class of functionally controlled portfolios, provided that the drift characteristic $\lambda$---as introduced in the model \eqref{eq:muP}---is sufficiently smooth. In fact, the log-optimal portfolio $\widehat{\pi}$ is known to be even a (classical) functionally generated portfolio if $\lambda$ can be written in the gradient form
\begin{equation*}
  \lambda(x) = \nabla \log G(x) = \frac{\nabla G(x)}{G(x)}, \qquad x \in \Delta_+^d,
\end{equation*}
for some differentiable function $G \colon \Delta^d_+ \to \mathbb{R}_+$; see \cite[Proposition~4.7]{Cuchiero2019}.

\smallskip

Considering again the stochastic model in \eqref{eq:muP}, we shall show in this section that the log-optimal portfolio may genuinely \emph{not} be a functionally generated portfolio, but still a functionally controlled one, in cases when $\lambda$ is not of the above gradient type. We will then illustrate numerically that the  difference between the true log-optimal portfolio and an approximate ``best'' portfolio based on a class of gradient type trading strategies can be substantial. This demonstrates that such extensions beyond classical functionally generated portfolios are crucial.

\smallskip

Let us consider a so-called volatility stabilized market model of the form in \eqref{eq:muP}, where, for some $\gamma > 0$, the diffusion matrix is given by
\begin{equation*}
  c^{ij}(\mu) := \gamma \mu^i (\delta_{ij} - \mu^j), \qquad i, j = 1, \ldots, d,
\end{equation*}
where $\delta_{ij}$ is the Kronecker delta, and the drift is given by
\begin{equation*}
  c(\mu) \lambda(\mu) = B \mu,
\end{equation*}
where $B \in \mathbb{R}^{d \times d}$ is defined by $B^{ij} := \frac{1 + \alpha}{2} (1 - \delta_{ij} d)$ for some $\alpha > \gamma - 1$. In the context of stochastic portfolio theory these models were first considered in \cite{Fernholz2005}. The condition $ \alpha > \gamma -1$ assures non-attainment of the boundary, as proved in \cite[Proposition~5.7]{Cuchiero2019a}, i.e.~the process $\mu$ takes values in $\Delta^d_+$.

We can solve this linear system for $\lambda$, and find as general solution
\begin{equation*}
  \lambda^i(\mu) = \frac{1+\alpha}{2 \gamma \mu^i} + C, \qquad i = 1, \ldots, d,
\end{equation*}
for an arbitrary $C \in \mathbb{R}$. Note that this is well-defined as $\mu$ always stays within the interior of the unit simplex $\Delta_+^d$ due to the condition $\alpha > \gamma - 1$. We now define the function $f^{\alpha} \colon \mathbb{R}^d_+ \to \mathbb{R}$ by
\begin{align}\label{eq:funcgen}
  f^{\alpha}(x):= \frac{1 + \alpha}{2 \gamma} \sum_{i=1}^d \log(x^i) + C \sum_{i=1}^d x^i.
\end{align}
Then $\partial_i f^{\alpha}(x) = (1 + \alpha)/(2 \gamma x^i) + C$ for $i = 1, \ldots, d$, so that
\begin{equation*}
  \lambda(x) = \nabla f^{\alpha}(x) = \nabla \log G(x), \qquad x \in \Delta_+^d,
\end{equation*}
where $G(x) := \exp (f^\alpha(x))$. Hence, in this volatility stabilized model the log-optimal portfolio $\widehat{\pi}$ can be realized as a functionally generated portfolio. It follows from \eqref{eq:expectedlogopt} that
\begin{equation*}
  \sup_{\pi \in \Pi} \mathbb{E} [\log V^\pi_T] = \frac{(1 + \alpha)^2}{8 \gamma} \bigg(\mathbb{E} \bigg[\int_0^T \sum_{i=1}^d \frac{1}{\mu^i_s} \dd s\bigg] - d^2 T\bigg).
\end{equation*}

A generalization of this model is a polynomial model with the same diffusion matrix (for some fixed $\gamma$), but a more general drift matrix $B$ just satisfying $B^{jj}=-\sum_{i \neq j} B^{ij}$ and $B^{ij} \geq 0$ for $i \neq j$ (see \cite[Definition~4.9]{Cuchiero2019a}). In this case $\lambda$ is in general no longer of gradient type. To see this, let $d = 3$, and
\begin{align}\label{eq:genB}
  B = \begin{pmatrix}
  -p & q  & r\\
  p & -q & 0\\
  0 & 0 & -r
  \end{pmatrix}
\end{align}
for $p, q , r >0$ such that $2 \min(p,q,r)- \gamma \geq 0$, where the latter condition is imposed to guarantee non-attainment of the boundary (see \cite[Propostion~5.7]{Cuchiero2019a}). We refer also to \cite[Theorem~5.1]{Cuchiero2019a} for the relation to (NUPBR) and relative arbitrages.

The solution $\lambda$ of $c(x) \lambda(x) = B x$ is now found to be
\begin{align*}
  \lambda^1(x) &= \frac{1}{\gamma} \bigg(r - p + q \frac{x^2}{x^1} + r \frac{x^3}{x^1}\bigg) + C,\\
  \lambda^2(x) &= \frac{1}{\gamma} \bigg(r - q + p \frac{x^1}{x^2}\bigg) + C,\\
  \lambda^3(x) &= C,
\end{align*}
which cannot be realized as a gradient, for instance since $\frac{\partial \lambda^3}{\partial x^1} \neq \frac{\partial \lambda^1}{\partial x^3}$.

Let us now compare the log-optimal portfolio
\begin{equation*}
  (\widehat{\pi}_t)^i = \mu^i_t \bigg(\lambda^i(\mu_t) + 1 - \sum_{j=1}^d \mu^j_t \lambda^j(\mu_t)\bigg)
\end{equation*}
with the functionally generated portfolio
\begin{equation*}
  (\pi_t^{\alpha}) = \mu^i_t \bigg(\partial_i f^{\alpha}(\mu_t) + 1 - \sum_{j=1}^d \mu^j_t \partial_j f^{\alpha}(\mu_t)\bigg),
\end{equation*}
with $f^{\alpha}$ as defined in \eqref{eq:funcgen}. We seek the value of $\alpha$ which optimizes
\begin{equation*}
  \sup_\alpha \mathbb{E} [\log V^{\pi^\alpha}_T].
\end{equation*}
By \eqref{eq:muP} and \eqref{eq:portfolio}, we have that
\begin{align*}
  \mathbb{E} &[\log V^{\pi^\alpha}_T] = \mathbb{E} \bigg[\int_0^T \nabla^\top \hspace{-1.5pt} f^{\alpha}(\mu_s) B \mu_s \dd s - \frac{1}{2} \int_0^T \nabla^\top \hspace{-1.5pt} f^{\alpha}(\mu_s) c(\mu_s) \nabla f^{\alpha}(\mu_s) \dd s\bigg]\\
  &= \frac{1+\alpha}{2 \gamma} \mathbb{E} \bigg[\int_0^T \bigg(\frac{1}{\mu^1_s}, \ldots, \frac{1}{\mu^d_s}\bigg) B \mu_s \dd s\bigg] - \frac{(1+\alpha)^2}{8 \gamma} \bigg(\mathbb{E} \bigg[\int_0^T \sum_{i=1}^d \frac{1}{\mu^i_s} \dd s\bigg] - d^2 T\bigg).
\end{align*}
Since this expression is concave in $\alpha$, we find the optimizer $\alpha^\ast$ to be given by
\begin{equation*}
  \alpha^\ast = \frac{2 \mathbb{E} \Big[\int_0^T \Big(\frac{1}{\mu^1_s}, \ldots, \frac{1}{\mu^d_s}\Big) B \mu_s \dd s\Big]}{\mathbb{E} \Big[\int_0^T \sum_{i=1}^d \frac{1}{\mu^i_s} \dd s\Big] - d^2 T} - 1.
\end{equation*}
Note that if $B$ is the drift matrix of a volatility stabilized market model with parameter $\alpha$, the right-hand side yields exactly $\alpha$, and we find the correct log-optimal portfolio. However, when we take $\pi^{\alpha^\ast}$ as an approximate portfolio, for instance in the case of $B$ being of the form \eqref{eq:genB}, this leads to Figure~\ref{fig1}. There, with the parameters $p = 0.15$, $q = 0.3$, $r = 0.2$, the functions $t \mapsto \mathbb{E} [\log \widehat{V}_t]$ (blue) and $t \mapsto \mathbb{E} [\log V^{\pi^{\alpha^\ast}}_t]$ (orange) are plotted, where the expected value is computed via a Monte Carlo simulation. This shows a significantly better performance of the log-optimal portfolio and, thus, illustrates a clear benefit from going beyond functionally generated portfolios in stochastic portfolio theory.

\appendix

\section{On the rough path foundation}\label{sec:rough integration}

In this appendix we collect some results regarding rough integration, including its associativity and a Fubini type theorem. While such elementary results are well-known for stochastic It{\^o} integration and other classical theories of integration, the presented results seem to be novel in the context of rough path theory and are essential for the model-free portfolio theory developed in the previous sections.

\smallskip

Throughout this section we will consider a general $p$-rough path $\bX = (X,\X)$---that is, we will \emph{not} impose Property (RIE)---and, as usual, we will assume that $p, q$ and $r$ satisfy Assumption~\ref{ass: parameters}, so that in particular $1 < p/2 \leq r < p \leq q < \infty$.

\subsection{Products of controlled paths}

As a first step towards the associativity of rough integration, we show that the product of two controlled paths is again a controlled path; see \cite[Corollary~7.4]{Friz2020} for a similar result in a H{\"o}lder-rough path setting.

\begin{lemma}\label{lem: product of controlled rough paths}
  Let $X \in C^{p\textup{-var}}([0,T];\R^d)$. The product operator $\Pi$, given by
  \begin{align*}
    \crpXq([0,T];\R^d) \times \crpXq([0,T];\R^d) &\rightarrow \crpXq([0,T];\R^d),\\
    ((F,F'), (G,G')) &\mapsto (FG,(FG)'),
  \end{align*}
where $(FG)^i := F^i G^i$ and $((FG)')^{ij} := (F')^{ij} G^i + F^i (G')^{ij}$ for every $1 \leq i, j \leq d$, is a continuous bilinear map, and comes with the estimate
  \begin{equation}\label{eq:B continuous bilinear est}
    \|(F,F')(G,G')\|_{\crpXq} \leq C(1 + \|X\|_{p})^2 \|F,F'\|_{\crpXq} \|G,G'\|_{\crpXq},
  \end{equation}
  where the constant $C$ depends on $p, q, r$ and the dimension $d$. We call $\Pi((F,F'),(G,G'))$ the product of $(F,F')$ and $(G,G')$, which we sometimes denote simply by $FG$.
\end{lemma}

\begin{proof}
  It is clear from its definition that $\Pi$ is a bilinear map. Suppose $(F,F'), (G,G') \in \crpXq$. For all $1 \leq i,j \leq d$ and $(s,t) \in \Delta_{[0,T]}$, we have
  \begin{align}
    \|(FG)'\|_q &\lesssim \|F'\|_q \|G\|_\infty + \|F'\|_\infty \|G\|_q + \|F\|_q \|G'\|_\infty + \|F\|_\infty \|G'\|_q\nonumber\\
    &\lesssim (\|F\|_\infty + \|F\|_q + \|F'\|_\infty + \|F'\|_q)(\|G\|_\infty + \|G\|_q + \|G'\|_\infty + \|G'\|_q)\label{eq:est (FG)' q-var}\\
    &\lesssim (1 + \|X\|_p)^2 \|F,F'\|_{\crpXq} \|G,G'\|_{\crpXq}.\nonumber
  \end{align}
  To identify the remainder $R^{FG}$, we compute
  \begin{align*}
    (FG)^i_{s,t} &= F^i_{s,t} G^i_s + F^i_s G^i_{s,t} + F^i_{s,t} G^i_{s,t}\\
    &= \bigg(\sum_{j=1}^d (F')^{ij}_s X^j_{s,t} + (R^F)^i_{s,t}\bigg) G^i_s + F^i_s \bigg(\sum_{j=1}^d (G')^{ij}_s X^j_{s,t} + (R^G)^i_{s,t}\bigg) + F^i_{s,t} G^i_{s,t}\\
    &= \sum_{j=1}^d \Big((F')^{ij}_s G^i_s + F^i_s (G')^{ij}_s\Big) X^j_{s,t} + (R^F)^i_{s,t} G^i_s + F^i_s (R^G)^i_{s,t} + F^i_{s,t} G^i_{s,t}\\
    &= \sum_{j=1}^d ((FG)')^{ij}_s X^j_{s,t} + (R^{FG})^i_{s,t},
  \end{align*}
  where $(R^{FG})^i_{s,t} := (R^F)^i_{s,t} G^i_s + F^i_s (R^G)^i_{s,t} + F^i_{s,t} G^i_{s,t}$. Using the fact that $2r \geq p$, we then estimate
  \begin{align}
    \|R^{FG}\|_r &\lesssim \|R^F\|_r \|G\|_\infty + \|F\|_\infty \|R^G\|_r + \|F\|_{2r} \|G\|_{2r}\nonumber\\
    &\lesssim (1 + \|X\|_p) \Big(\|R^F\|_r  \|G,G'\|_{\crpXq} + \|F,F'\|_{\crpXq} \|R^G\|_r\Big) + \|F\|_p \|G\|_p\label{eq:est R^FG r-var}\\
    &\lesssim (1 + \|X\|_p)^2 \|F,F'\|_{\crpXq} \|G,G'\|_{\crpXq}.\nonumber
  \end{align}
  The estimate \eqref{eq:B continuous bilinear est} then follows from \eqref{eq:est (FG)' q-var} and \eqref{eq:est R^FG r-var}.
\end{proof}

\subsection{Associativity of rough integration}\label{subsec: associativity of rough integration}

The following proposition provides an associativity result for rough integration.

\begin{proposition}\label{prop: associativity of rough integration}
  Let $\bX = (X,\X)$ be a $p$-rough path and let $(Y,Y'), (F,F'), (G,G') \in \crpXq$ be controlled paths. Then, the pair $(Z,Z') := (\int_0^\cdot F_u \,\d G_u,FG') \in \crpXq$, and we have that
  \begin{equation*}
    \int_0^\cdot Y_u \,\d Z_u = \int_0^\cdot Y_u F_u \,\d G_u,
  \end{equation*}
  where on the left-hand side we have the integral of $(Y,Y')$ against $(Z,Z')$, and on the right-hand side we have the integral of $(YF,(YF)')$ against $(G,G')$, each defined in the sense of Lemma~\ref{lem: int of controlled paths exists}.
\end{proposition}

\begin{proof}
  The fact that $(Z,Z') \in \crpXp$ follows from the estimate in \eqref{eq:est int of controlled paths} combined with the relation $G_{s,t} = G'_s X_{s,t} + R^G_{s,t}$. It also follows from \eqref{eq:est int of controlled paths} that the function $H^{\int F \,\d G}$, defined by
  \begin{equation*}
    Z_{s,t} = \int_s^t F_u \,\d G_u = F_s G_{s,t} + F'_s G'_s \X_{s,t} + H^{\int F \,\d G}_{s,t}
  \end{equation*}
  for $(s,t) \in \Delta_T$, has finite $\hat{p}$-variation for some $\hat{p} < 1$, and we can thus conclude that $\lim_{|\mathcal{P}| \to 0} \sum_{[s,t] \in \mathcal{P}} |H^{\int F \,\d G}_{s,t}| = 0$. We similarly obtain
  \begin{align*}
    \int_s^t Y_u \,\d Z_u &= Y_s Z_{s,t} + Y'_s Z'_s \X_{s,t} + H^{\int Y \,\d Z}_{s,t},\\
    \int_s^t Y_u F_u \,\d G_u &= Y_s F_s G_{s,t} + (YF)'_s G'_s \X_{s,t} + H^{\int YF \,\d G}_{s,t},
  \end{align*}
  with
  \begin{equation*}
    \lim_{|\mathcal{P}| \to 0} \sum_{[s,t] \in \mathcal{P}} |H^{\int Y \,\d Z}_{s,t}| = \lim_{|\mathcal{P}| \to 0} \sum_{[s,t] \in \mathcal{P}} |H^{\int YF \,\d G}_{s,t}| = 0.
  \end{equation*}
  Noting that $(YF)' = YF' + Y'F$, we then calculate
  \begin{align*}
    \int_s^t Y_u \,\d Z_u &= Y_s Z_{s,t} + Y'_s Z'_s \X_{s,t} + H^{\int Y \,\d Z}_{s,t}\\
    &= Y_s \Big(F_s G_{s,t} + F'_s G'_s \X_{s,t} + H^{\int F \,\d G}_{s,t}\Big) + Y'_s F_s G'_s \X_{s,t} + H^{\int Y \,\d Z}_{s,t}\\
    &= Y_s F_s G_{s,t} + (Y_s F'_s + Y'_s F_s) G'_s \X_{s,t} + Y_s H^{\int F \,\d G}_{s,t} + H^{\int Y \,\d Z}_{s,t}\\
    &= \int_s^t Y_u F_u \,\d G_u - H^{\int YF \,\d G}_{s,t} + Y_s H^{\int F \,\d G}_{s,t} + H^{\int Y \,\d Z}_{s,t}.
  \end{align*}
  Taking $\lim_{|\mathcal{P}| \to 0} \sum_{[s,t] \in \mathcal{P}}$ on both sides, we obtain $\int_0^T Y_u \,\d Z_u = \int_0^T Y_u F_u \,\d G_u$.
\end{proof}

\begin{remark}
  Denoting the integration operator by $\bullet$, the result of Proposition~\ref{prop: associativity of rough integration} may be expressed formally as $Y \bullet (F \bullet G) = (YF) \bullet G$. We therefore refer to this result as the \emph{associativity} of rough integration.
\end{remark}

\subsection{The canonical rough path lift of a controlled path}\label{sec Rough paths associated to controlled paths}

Given a $p$-rough path $\bX = (X,\X)$ and a controlled path $(Z,Z') \in \crpXq$, one can use Lemma~\ref{lem: int of controlled paths exists} to enhance $Z$ in a canonical way to a $p$-rough path $\bZ = (Z,\Z)$, where $\Z$ is defined by
\begin{equation}\label{eq:canonical lift of Z}
\Z_{s,t} := \int_s^t Z_u \,\d Z_u - Z_s Z_{s,t}, \qquad \text{for} \quad (s,t) \in \Delta_{[0,T]},
\end{equation}
with the integral defined as in \eqref{eq:int of controlled paths defn}. Indeed, we observe the following.

\begin{lemma}\label{lem: consistency of rough integrals}
  Let $\bX = (X,\X)$ be a $p$-rough path and $(Z,Z') \in \crpXq$ be a controlled path. Then,  $\bZ = (Z,\Z)$, as defined in \eqref{eq:canonical lift of Z}, is a $p$-rough path. Moreover, if $(Y,Y') \in \crpZq$, then $(Y,Y'Z') \in \crpXq$ and
  \begin{equation*}
    \int_0^T Y_u \dd \bZ_u = \int_0^T Y_u \dd Z_u,
  \end{equation*}
  where on the left-hand side we have the rough integral of $(Y,Y')$ against $\bZ$, and on the right-hand side we have the integral of $(Y,Y'Z')$ against $(Z,Z')$ as defined in~\eqref{eq:int of controlled paths defn}.
\end{lemma}

\begin{proof}
  That $\bZ = (Z,\Z)$ is a $p$-rough path follows immediately from Lemma~\ref{lem: int of controlled paths exists}. That $(Y,Y'Z') \in \crpXq$ can be shown in a straightforward manner using the definition of controlled paths. Arguing similarly as in the proof of Proposition~\ref{prop: associativity of rough integration} and using the same notation, we calculate, for $(s,t) \in \Delta_{[0,T]}$,
  \begin{align*}
    \int_s^t Y_u \,\d \bZ_u &= Y_s Z_{s,t} + Y'_s \Z_{s,t} + H^{\int Y \,\d \bZ}_{s,t}\\
    &= Y_s Z_{s,t} + Y'_s \Big(Z'_s Z'_s \X_{s,t} + H^{\int Z \,\d Z}_{s,t}\Big) + H^{\int Y \,\d \bZ}_{s,t}\\
    &= \int_s^t Y_u \,\d Z_u - H^{\int Y \,\d Z}_{s,t} + Y'_s H^{\int Z \,\d Z}_{s,t} + H^{\int Y \,\d \bZ}_{s,t}.
  \end{align*}
  Taking $\lim_{|\mathcal{P}| \to 0} \sum_{[s,t] \in \mathcal{P}}$ on both sides, we obtain $\int_0^T Y_u \,\d \bZ_u = \int_0^T Y_u \,\d Z_u$.
\end{proof}

\subsection{The exponential of a rough path}

Based on the bracket of a rough path (recall Definition~\ref{def: bracket}), one can introduce the rough exponential analogously to the stochastic exponential of It{\^o} calculus.

\begin{lemma}\label{lem: dynamics of rough exponential}
  For a one-dimensional $p$-rough path $\bX = (X,\X)$ (so that in particular $X$ is real-valued) such that $X_0 = 0$, we introduce the rough exponential by
  \begin{equation*}
    V_t := \exp\Big(X_t - \frac{1}{2}[\bX]_t\Big), \qquad t \in [0,T].
  \end{equation*}
  Then $V$ is the unique controlled path in $\crpXp$ satisfying the linear rough differential equation
  \begin{equation}\label{eq: dynamics of rough exponential}
    V_t = 1 + \int_0^t V_u \dd \bX_u,\qquad t \in [0,T],
  \end{equation}
  with Gubinelli derivative $V^\prime = V$.
\end{lemma}

\begin{proof}
  Applying the It{\^o} formula of Proposition~\ref{prop: general Ito formula for rough paths} with $Y = X - \frac{1}{2}[\bX]$, $Y' = 1$ and $f = \exp$, we observe that the Young integrals cancel, so that $V$ does indeed satisfy \eqref{eq: dynamics of rough exponential}. The uniqueness of solutions to \eqref{eq: dynamics of rough exponential} follows from the stability of rough integration, provided in this setting by \cite[Lemma~3.4]{Friz2018}.
\end{proof}

\subsection{A Fubini type theorem for rough integration}

In this subsection we provide a Fubini type theorem for Bochner and rough integrals. A result of this type is mentioned in a H{\"o}lder-rough path setting in \cite[Exercise~4.10]{Friz2020}.

\begin{theorem}\label{thm: rough Fubini}
  Let $\bX = (X,\X)$ be a $p$-rough path, let $\mathcal{A}$ be a measurable subset of $\crpXq$, and let $\nu$ be a probability measure on $\mathcal{A}$. If $\int_{\mathcal{A}} \|K,K'\|_{\crpXq} \,\d \nu < \infty$, then
  \begin{equation*}
    \int_0^T \int_{\mathcal{A}} K_u \,\d \nu \,\d \bX_u = \int_{\mathcal{A}} \int_0^T K_u \,\d \bX_u \,\d \nu.
  \end{equation*}
\end{theorem}

\begin{proof}
  Due to $\int_{\mathcal{A}} \|K,K'\|_{\crpXq} \,\d \nu < \infty$, the controlled path $\int_{\mathcal{A}} (K,K') \,\d \nu \in \crpXq$ exists as a well-defined Bochner integral. For $s < t$, we have
  \begin{align*}
    \int_{\mathcal{A}} \int_s^t K_u \dd \bX_u \,\d \nu - \int_{\mathcal{A}} K_s \,\d \nu \, X_{s,t} - \int_{\mathcal{A}} K'_s \dd \nu \, \X_{s,t}
    = \int_{\mathcal{A}} \bigg(\int_s^t K_u \,\d \bX_u - K_s X_{s,t} - K'_s \X_{s,t}\bigg) \dd \nu
  \end{align*}
  and, by the estimate in \eqref{eq:rough integral estimate},
  \begin{equation}\label{eq:H^K s,t bound}
    \bigg|\int_s^t K_u \,\d \bX_u - K_s X_{s,t} - K'_s \X_{s,t}\bigg| \leq C(\|R^K\|_{r,[s,t]} \|X\|_{p,[s,t]} + \|K'\|_{q,[s,t]} \|\X\|_{\frac{p}{2},[s,t]}).
  \end{equation}
  Since $1/r + 1/p > 1$, there exists a $\hat{p} > p$ such that $1/r + 1/\hat{p} = 1$. By H{\"o}lder's inequality, for any partition $\mathcal{P}$ of $[0,T]$, we have
  \begin{align*}
    \int_{\mathcal{A}} \sum_{[s,t] \in \mathcal{P}} \|R^K\|_{r,[s,t]} \|X\|_{p,[s,t]} \,\d \nu &\leq \int_{\mathcal{A}} \bigg(\sum_{[s,t] \in \mathcal{P}} \|R^K\|_{r,[s,t]}^r\bigg)^{\frac{1}{r}} \bigg(\sum_{[s,t] \in \mathcal{P}} \|X\|_{p,[s,t]}^{\hat{p}}\bigg)^{\frac{1}{\hat{p}}} \,\d \nu\\
    &\leq \int_{\mathcal{A}} \|R^K\|_{r,[0,T]} \,\d \nu \, \|X\|_{p,[0,T]}^{\frac{p}{\hat{p}}} \Big(\max_{[s,t] \in \mathcal{P}} \|X\|_{p,[s,t]}^{\frac{\hat{p} - p}{\hat{p}}}\Big).
  \end{align*}
  Since $\int_{\mathcal{A}} \|R^K\|_{r,[0,T]} \,\d \nu \leq \int_{\mathcal{A}} \|K,K'\|_{\crpXq} \,\d \nu < \infty$, and since $(s,t) \mapsto \|X\|_{p,[s,t]}$ is uniformly continuous, we deduce, treating the second term on the right-hand side of \eqref{eq:H^K s,t bound} similarly, that
  \begin{equation*}
    \lim_{|\mathcal{P}| \to 0} \sum_{[s,t] \in \mathcal{P}} \int_{\mathcal{A}}  \bigg(\int_s^t K_u \,\d \bX_u - K_s X_{s,t} - K'_s \X_{s,t}\bigg) \,\d \nu = 0.
  \end{equation*}
  Thus, we obtain
  \begin{align*}
    \int_{\mathcal{A}} \int_0^T K_u \,\d \bX_u \,\d \nu 
    &= \lim_{|\mathcal{P}| \to 0} \sum_{[s,t] \in \mathcal{P}} \int_{\mathcal{A}} \int_s^t K_u \,\d \bX_u \,\d \nu\\  
    &= \lim_{|\mathcal{P}| \to 0} \sum_{[s,t] \in \mathcal{P}} \int_{\mathcal{A}} K_s \,\d \nu \, X_{s,t} + \int_{\mathcal{A}} K'_s \,\d \nu \, \X_{s,t} = \int_0^T \int_{\mathcal{A}} K_u \,\d \nu \,\d \bX_u.
  \end{align*}
\end{proof}

\section{Rough path theory assuming Property (RIE)}\label{sec:appendix on RIE}

In this section we provide additional results concerning rough path theory assuming Property \textup{(RIE)}, and, in particular, we give a proof of Theorem~\ref{thm: Ito integral for smooth transformed RIE path}. As usual we adopt Assumption~\ref{ass: parameters}.

\subsection{On the bracket of a rough path}

We begin with some properties of the bracket of a rough path, introduced in Definition~\ref{def: bracket}.

\begin{lemma}\label{lem: Ito isometry for rough paths}
  Let $\bX = (X,\X)$ be a $p$-rough path and let $(K,K') \in \crpXq$. Recall from Proposition~\ref{prop: associativity of rough integration} that $(Z,Z') := (\int_0^\cdot K_u \,\d \bX_u,K) \in \crpXq$. Let $\bZ = (Z,\Z)$ be the canonical rough path lift of $Z$, as defined in \eqref{eq:canonical lift of Z}, so that in particular the bracket $[\bZ]$ of $\bZ$ exists. Then
  \begin{equation*}
    [\bZ] = \int_0^\cdot (K_u \otimes K_u) \,\d [\bX]_u,
  \end{equation*}
  where the right-hand side is defined as a Young integral.
\end{lemma}

\begin{proof}
  Since $[\bX]$ has finite $p/2$-variation, the integral
  \begin{equation*}
    \int_0^T (K_u \otimes K_u) \,\d [\bX]_u = \lim_{|\mathcal{P}| \to 0} \sum_{[s,t] \in \mathcal{P}} (K_s \otimes K_s) [\bX]_{s,t} 
  \end{equation*}
  exists as a Young integral. In the following we shall abuse notation slightly by writing $H_{s,t} = o(|t - s|)$ whenever a function $H$ satisfies $\lim_{|\mathcal{P}| \to 0} \sum_{[s,t] \in \mathcal{P}} |H_{s,t}| = 0$. We have
  \begin{align*}
    [\bZ]_{s,t} &= Z_{s,t} \otimes Z_{s,t} - 2\textup{Sym}(\Z_{s,t})\\
    &= (K_s X_{s,t} + K'_s \X_{s,t}) \otimes (K_s X_{s,t} + K'_s \X_{s,t}) - 2(Z'_s \otimes Z'_s)\textup{Sym}(\X_{s,t}) + o(|t - s|)\\
    &= (K_s X_{s,t}) \otimes (K_s X_{s,t}) - 2(K_s \otimes K_s)\textup{Sym}(\X_{s,t}) + o(|t - s|)\\
    &= (K_s \otimes K_s) [\bX]_{s,t} + o(|t - s|).
  \end{align*}
  Taking $\lim_{|\mathcal{P}| \to 0} \sum_{[s,t] \in \mathcal{P}}$ on both sides, we obtain $[\bZ]_T = \int_0^T (K_u \otimes K_u) \,\d [\bX]_u$.
\end{proof}

\begin{proposition}\label{prop: property of bracket of rough paths}
  Suppose that $S\in C([0,T];\R^d)$ satisfies \textup{(RIE)} with respect to $p$ and $(\mathcal{P}^n)_{n \in \N}$. Let $\bS = (S,\S)$ be the associated rough path as defined in \eqref{eq:defn A st}. Let $(K,K') \in \crpSq$ and $(Z,Z') = (\int_0^\cdot K_u \dd \bS_u,K) \in \crpSq$. Let $\bZ = (Z,\Z)$ be the canonical rough path lift of $Z$ as defined in \eqref{eq:canonical lift of Z}, so that in particular the bracket $[\bZ]$ of $\bZ$ exists. Then the following hold:
  \begin{enumerate}
    \item[(i)] The bracket $[\bZ]$ has finite total variation, and is given by
    \begin{equation*}
      [\bZ]_t = \lim_{n \to \infty} \sum_{k=0}^{N_n-1} Z_{t^n_k \wedge t,t^n_{k+1} \wedge t} \otimes Z_{t^n_k \wedge t,t^n_{k+1} \wedge t}, \qquad t \in [0,T].
   \end{equation*}
   \item[(ii)] Let $\Gamma$ be a continuous path of finite $p/2$-variation. Then the path $Y := Z + \Gamma$ admits a canonical rough path lift $\mathbf{Y} = (Y,\mathbb{Y})$, such that
   \begin{equation}\label{eq:bracket of Y}
     [\mathbf{Y}]_t = [\bZ]_t = \lim_{n \to \infty} \sum_{k=0}^{N_n-1} Y_{t^n_k \wedge t,t^n_{k+1} \wedge t} \otimes Y_{t^n_k \wedge t,t^n_{k+1} \wedge t}, \qquad t \in [0,T].
    \end{equation}
  \end{enumerate}
\end{proposition}

\begin{proof}
  (i) Since, by Lemma~\ref{lem: property of bracket of rough paths}, $[\bS]$ has finite variation, it follows from Lemma~\ref{lem: Ito isometry for rough paths} that the same is true of $[\bZ]$. By the estimate in \eqref{eq:rough integral estimate}, we know that $Z_{s,t} = K_s S_{s,t} + K'_s \S_{s,t} + H_{s,t}$ for some $H$ satisfying $\lim_{|\mathcal{P}| \to 0} \sum_{[s,t] \in \mathcal{P}} |H_{s,t}| = 0$. It follows that
  \begin{align*}
    \lim_{n \to \infty} \sum_{k=0}^{N_n-1} Z_{t^n_k \wedge t,t^n_{k+1} \wedge t} \otimes Z_{t^n_k \wedge t,t^n_{k+1} \wedge t} &= \lim_{n \to \infty} \sum_{k=0}^{N_n-1} (K_{t^n_k \wedge t} S_{t^n_k \wedge t,t^n_{k+1} \wedge t}) \otimes (K_{t^n_k \wedge t} S_{t^n_k \wedge t,t^n_{k+1} \wedge t})\\
    &= \int_0^t (K_u \otimes K_u) \,\d [\bS]_u = [\bZ]_t.
  \end{align*}

  (ii) Since $\Gamma$ has finite $p/2$-variation, the Young integrals $\int_s^t Z_{s,u} \otimes \d \Gamma_u$, $\int_s^t \Gamma_{s,u} \otimes \d Z_u$ and $\int_s^t \Gamma_{s,u} \otimes \d \Gamma_u$ are well-defined, and the function $\mathbb{Y}$, defined by
  \begin{equation*}
    \mathbb{Y}_{s,t} = \Z_{s,t} + \int_s^t Z_{s,u} \otimes \d \Gamma_u + \int_s^t \Gamma_{s,u} \otimes \d Z_u + \int_s^t \Gamma_{s,u} \otimes \d \Gamma_u,
  \end{equation*}
  also has finite $p/2$-variation. It follows that $\mathbf{Y} = (Y,\mathbb{Y})$ is a $p$-rough path. The equality $[\mathbf{Y}]_t = [\bZ]_t$ follows easily from the integration by parts formula for Young integrals. The second equality in \eqref{eq:bracket of Y} follows by a similar argument to the one in the proof of part~(i).
\end{proof}

\subsection{Proof -- the rough integral as a limit of Riemann sums}

\begin{proof}[Proof of Theorem~\ref{thm: Ito integral for smooth transformed RIE path}]
  Let $(Y,Y') \in \crpSq$. Recalling the It{\^o} formula for rough paths (Proposition~\ref{prop: general Ito formula for rough paths}), it follows from the associativity of Young and rough integrals (recall Proposition~\ref{prop: associativity of rough integration}) that
  \begin{equation*}
    \int_0^t Y_u \,\d f(S)_u = \int_0^t Y_u \D f(S_u) \,\d \bS_u + \frac{1}{2} \int_0^t Y_u \D^2 f(S_u) \,\d [\bS]_u.
  \end{equation*}
  By \cite[Theorem~4.19]{Perkowski2016}, we have
  \begin{equation*}
    \int_0^t Y_u \D f(S_u) \dd \bS_u = \lim_{n \to \infty} \sum_{k=0}^{N_n-1} Y_{t^n_k} \D f(S_{t^n_k})S_{t^n_k \wedge t,t^n_{k+1} \wedge t},
  \end{equation*}
  the convergence being uniform in $t \in [0,T]$. By \cite[Lemma~5.11]{Friz2020}, we have the pointwise convergence
  \begin{equation}\label{eq:convergence to integral against bracket}
    \lim_{n \to \infty} \sum_{k=0}^{N_n-1} Y_{t^n_k} \D^2 f(S_{t^n_k}) S_{t^n_k \wedge t,t^n_{k+1} \wedge t}^{\otimes 2} = \int_0^t Y_u \D^2 f(S_u) \,\d [\bS]_u.
  \end{equation}
  Recalling P{\'o}lya's theorem (see e.g.~\cite{Rao1962}), which asserts that pointwise convergence of distribution functions on $\R$ to a continuous limit implies the uniformity of this convergence, we see from the proof of \cite[Lemma~5.11]{Friz2020} that the convergence in \eqref{eq:convergence to integral against bracket} also holds uniformly for $t \in [0,T]$. Thus, we obtain
  \begin{equation}\label{eq: uniform convergence of rough integral}
    \int_0^t Y_u \,\d f(S)_u = \lim_{n \to \infty} \sum_{k=0}^{N_n-1} \Big( Y_{t^n_k} \D f(S_{t^n_k}) S_{t^n_k \wedge t,t^n_{k+1} \wedge t} + \frac{1}{2} Y_{t^n_k} \D^2 f(S_{t^n_k}) S_{t^n_k \wedge t,t^n_{k+1} \wedge t}^{\otimes 2}\Big),
  \end{equation}
  where the convergence is uniform in $t \in [0,T]$. For every $n$ and $k$, we have, by Taylor expansion,
  \begin{align}\label{eq:Taylor expand f(S)} 
    \begin{split}
    &Y_{t^n_k} f(S)_{t^n_k \wedge t,t^n_{k+1} \wedge t} \\
    &\quad= Y_{t^n_k} \D f(S_{t^n_k}) S_{t^n_k \wedge t,t^n_{k+1} \wedge t} + \frac{1}{2} Y_{t^n_k} \D^2 f(S_{t^n_k}) S_{t^n_k \wedge t,t^n_{k+1} \wedge t}^{\otimes 2} + Y_{t^n_k} R_{t^n_k \wedge t,t^n_{k+1} \wedge t},
    \end{split}
  \end{align}
  where
  \begin{equation*}
    R_{u,v} := \int_0^1 \int_0^1 \Big(\D^2 f(S_u + r_1 r_2 S_{u,v}) - \D^2 f(S_u)\Big) S_{u,v}^{\otimes 2} \, r_1 \dd r_2 \,\d r_1.
  \end{equation*}
  Since $f \in C^{p + \epsilon}$, we have that $|R_{u,v}| \lesssim |S_{u,v}|^{p + \epsilon}$, from which we see that $R$ has finite $p/(p + \epsilon)$-variation. Since $p/(p + \epsilon) < 1$, it follows that
  \begin{equation*}
    \lim_{n \to \infty} \sum_{k=0}^{N_n-1} Y_{t^n_k} R_{t^n_k \wedge t,t^n_{k+1} \wedge t} = 0,
  \end{equation*}
  where the convergence is uniform in $t \in [0,T]$. Thus, taking $\lim_{n \to \infty} \sum_{k=0}^{N_n-1}$ in \eqref{eq:Taylor expand f(S)} and substituting into \eqref{eq: uniform convergence of rough integral}, we deduce the result.
\end{proof}

\bibliography{quellen}{}

\def\cprime{$'$}
\providecommand{\bysame}{\leavevmode\hbox to3em{\hrulefill}\thinspace}
\providecommand{\MR}{\relax\ifhmode\unskip\space\fi MR }
\providecommand{\MRhref}[2]{%
  \href{http://www.ams.org/mathscinet-getitem?mr=#1}{#2}
}
\providecommand{\href}[2]{#2}
\begin{thebibliography}{CSW19}

\bibitem[ABI19]{AbiJaber2017}
Eduardo {Abi Jaber}, Bruno Bouchard, and Camille Illand, \emph{Stochastic
  invariance of closed sets with non-{L}ipschitz coefficients}, Stochastic
  Processes and their Applications \textbf{129} (2019), no.~5, 1726--1748.

\bibitem[Ana20]{Ananova2020}
Anna Ananova, \emph{Rough differential equations with path-dependent
  coefficients}, Preprint arXiv:2001.10688 (2020).

\bibitem[CC22]{CC:22}
Henry Chiu and Rama Cont, \emph{Causal functional calculus}, Transactions of
  the London Mathematical Society \textbf{9} (2022), no.~1, 237--269.

\bibitem[CF10]{Cont2010}
Rama Cont and David-Antoine Fourni{\'e}, \emph{{Change of variable formulas for
  non-anticipative functionals on path space}}, J. Funct. Anal. \textbf{259}
  (2010), no.~4, 1043--1072.

\bibitem[CG98]{Chistyakov1998}
V.~V. Chistyakov and O.~E. Galkin, \emph{On maps of bounded {$p$}-variation
  with {$p>1$}}, Positivity \textbf{2} (1998), no.~1, 19--45.

\bibitem[CO96]{Cover1996}
T.~M. Cover and E.~Ordentlich, \emph{Universal portfolios with side
  information}, IEEE Transactions on Information Theory \textbf{42} (1996),
  no.~2, 348--363.

\bibitem[Cov91]{Cover1991}
Thomas~M. Cover, \emph{Universal portfolios}, Math. Finance \textbf{1} (1991),
  no.~1, 1--29.

\bibitem[CP19]{Cont2019}
Rama Cont and Nicolas Perkowski, \emph{Pathwise integration and change of
  variable formulas for continuous paths with arbitrary regularity}, Trans.
  Amer. Math. Soc. Ser. B \textbf{6} (2019), 161--186.

\bibitem[CSW19]{Cuchiero2019}
Christa Cuchiero, Walter Schachermayer, and Ting-Kam~Leonard Wong,
  \emph{Cover's universal portfolio, stochastic portfolio theory, and the
  numéraire portfolio}, Mathematical Finance \textbf{29} (2019), no.~3,
  773--803.

\bibitem[Cuc19]{Cuchiero2019a}
Christa Cuchiero, \emph{Polynomial processes in stochastic portfolio theory},
  Stochastic processes and their applications \textbf{129} (2019), no.~5,
  1829--1872.

\bibitem[CW22]{Campbell2022}
Steven Campbell and Ting-Kam~Leonard Wong, \emph{Functional portfolio
  optimization in stochastic portfolio theory}, SIAM J. Financial Math.
  \textbf{13} (2022), no.~2, 576--618.

\bibitem[CZ93]{Chopra1993}
Vijay~K. Chopra and William~T. Ziemba, \emph{The effect of errors in means,
  variances, and covariances on optimal portfolio choice}, The Journal of
  Portfolio Management \textbf{19} (1993), no.~2, 6--11.

\bibitem[dF40]{deFinetti1940}
Bruno de~Finetti, \emph{Il problema dei "{P}ieni". {G}iornale dell’
  {I}stituto {I}taliano degli {A}ttuari 11, 1--88; translation ({B}arone, {L}.
  (2006)): {T}he problem of full-risk insurances. {C}hapter~{I}. {T}he risk
  within a single accounting period.}, Journal of Investment Management
  \textbf{4(3)} (1940), 19--43.

\bibitem[DGU07]{DeMiguel2007}
Victor DeMiguel, Lorenzo Garlappi, and Raman Uppal, \emph{{O}ptimal {V}ersus
  {N}aive {D}iversification: {H}ow {I}nefficient is the 1/{N} {P}ortfolio
  {S}trategy?}, The Review of Financial Studies \textbf{22} (2007), no.~5,
  1915--1953.

\bibitem[Dup19]{Dupire2019}
Bruno Dupire, \emph{Functional {I}t{\^o} calculus}, Quant. Finance \textbf{19}
  (2019), no.~5, 721--729.

\bibitem[Ebe16]{E:16}
Andreas Eberle, \emph{Markov processes}, Lecture Notes at University of Bonn
  (2016).

\bibitem[Fer99]{Fernholz1999}
Robert Fernholz, \emph{Portfolio {G}enerating {F}unctions}, pp.~344--367, River
  Edge, NJ. World Scientific, 1999.

\bibitem[Fer01]{Fernholz2001}
\bysame, \emph{Equity portfolios generated by functions of ranked market
  weights}, Finance Stoch. \textbf{5} (2001), no.~4, 469--486.

\bibitem[Fer02]{Fernholz2002}
E.~Robert Fernholz, \emph{Stochastic portfolio theory}, Springer, 2002.

\bibitem[FH20]{Friz2020}
Peter~K. Friz and Martin Hairer, \emph{A course on rough paths with an
  introduction to regularity structures}, 2nd ed., Universitext, Springer,
  Cham, 2020.

\bibitem[FK05]{Fernholz2005}
Robert Fernholz and Ioannis Karatzas, \emph{Relative arbitrage in
  volatility-stabilized markets}, Annals of Finance \textbf{1} (2005), no.~2,
  149--177.

\bibitem[FK11]{Fernholz2011}
Daniel Fernholz and Ioannis Karatzas, \emph{{Optimal arbitrage under model
  uncertainty}}, The Annals of Applied Probability \textbf{21} (2011), no.~6,
  2191 -- 2225.

\bibitem[FL16]{Filipovic2016}
Damir Filipovi{\'c} and Martin Larsson, \emph{Polynomial diffusions and
  applications in finance}, Finance and Stochastics \textbf{20} (2016), no.~4,
  931--972.

\bibitem[F{\"o}l81]{Follmer1981}
H.~F{\"o}llmer, \emph{Calcul d'{I}t\^{o} sans probabilit\'{e}s}, Seminar on
  {P}robability, {XV} ({U}niv. {S}trasbourg, {S}trasbourg, 1979/1980), Lecture
  Notes in Math., vol. 850, Springer, Berlin, 1981, pp.~143--150.

\bibitem[FV10]{Friz2010}
Peter~K. Friz and Nicolas~B. Victoir, \emph{{Multidimensional stochastic
  processes as rough paths. Theory and applications}}, Cambridge University
  Press, 2010.

\bibitem[FZ18]{Friz2018}
Peter~K. Friz and Huilin Zhang, \emph{Differential equations driven by rough
  paths with jumps}, J. Differential Equations \textbf{264} (2018), no.~10,
  6226--6301.

\bibitem[GR13]{Guidolin2013}
Massimo Guidolin and Francesca Rinaldi, \emph{Ambiguity in asset pricing and
  portfolio choice: a review of the literature}, Theory and Decision
  \textbf{74} (2013), no.~2, 183--217.

\bibitem[HKT02]{Hubalek2002}
Friedrich Hubalek, Irene Klein, and Josef Teichmann, \emph{A general proof of
  the {D}ybvig-{I}ngersoll-{R}oss theorem: long forward rates can never fall},
  Math. Finance \textbf{12} (2002), no.~4, 447--451.

\bibitem[HS10]{Hulley2010}
Hardy Hulley and Martin Schweizer, \emph{{$\rm M^6$}---on minimal market models
  and minimal martingale measures}, Contemporary quantitative finance,
  Springer, Berlin, 2010, pp.~35--51.

\bibitem[IL22]{Itkin2020}
David Itkin and Martin Larsson, \emph{Robust asymptotic growth in stochastic
  portfolio theory under long-only constraints}, Math. Finance \textbf{32}
  (2022), no.~1, 114--171.

\bibitem[IP15]{Imkeller2015}
Peter Imkeller and Nicolas Perkowski, \emph{The existence of dominating local
  martingale measures}, Finance Stoch. \textbf{19} (2015), no.~4, 685--717.

\bibitem[Jam92]{Jamshidian1992}
Farshid Jamshidian, \emph{Asymptotically {O}ptimal {P}ortolios}, Mathematical
  Finance \textbf{2} (1992), no.~2, 131--150.

\bibitem[KF09]{Karatzas2009}
Ioannis Karatzas and Robert Fernholz, \emph{Stochastic portfolio theory: an
  overview}, Special Volume: Mathematical Modeling and Numerical Methods in
  Finance (Alain Bensoussan and Qiang Zhang, eds.), Handbook of Numerical
  Analysis, vol.~15, Elsevier, 2009, pp.~89--167.

\bibitem[KK07]{Karatzas2007}
Ioannis Karatzas and Constantinos Kardaras, \emph{{The num\'{e}raire portfolio
  in semimartingale financial models}}, Finance Stoch. \textbf{11} (2007),
  no.~4, 447--493.

\bibitem[KK20]{Karatzas2020}
Ioannis Karatzas and Donghan Kim, \emph{Trading strategies generated pathwise
  by functions of market weights}, Finance Stoch. \textbf{24} (2020), no.~2,
  423--463.

\bibitem[KR12]{KardarasRobertson2012}
Constantinos Kardaras and Scott Robertson, \emph{Robust maximization of
  asymptotic growth}, The Annals of Applied Probability \textbf{22} (2012),
  no.~4, 1576--1610.

\bibitem[KR17]{Karatzas2017}
Ioannis Karatzas and Johannes Ruf, \emph{Trading strategies generated by
  {L}yapunov functions}, Finance Stoch. \textbf{21} (2017), no.~3, 753--787.

\bibitem[KR21]{Kardaras2018}
Constantinos Kardaras and Scott Robertson, \emph{Ergodic robust maximization of
  asymptotic growth}, Ann. Appl. Probab. \textbf{31} (2021), no.~4, 1787--1819.

\bibitem[LCL07]{Lyons2007}
Terry~J. Lyons, Michael Caruana, and Thierry L\'{e}vy, \emph{{Differential
  equations driven by rough paths}}, {Lecture Notes in Mathematics}, vol. 1908,
  Springer, Berlin, 2007.

\bibitem[Lej12]{Lejay2014a}
Antoine Lejay, \emph{{Global solutions to rough differential equations with
  unbounded vector fields}}, {S\'{e}minaire de {P}robabilit\'{e}s {XLIV}},
  {Lecture Notes in Math.}, vol. 2046, Springer, Heidelberg, 2012,
  pp.~215--246. \MR{2953350}

\bibitem[LH14]{Li2014}
Bin Li and Steven C.~H. Hoi, \emph{Online {P}ortfolio {S}election: {A}
  {S}urvey}, ACM Comput. Surv. \textbf{46} (2014), no.~3.

\bibitem[LQ02]{Lyons2002}
Terry Lyons and Zhongmin Qian, \emph{{System control and rough paths}}, Oxford
  University Press, 2002.

\bibitem[Mar59]{Markowitz1959}
Harry Markowitz, \emph{Portfolio selection: efficient diversification of
  investments}, no.~16, New Haven: Yale University Press, c1959, 1970
  printing., 1959.

\bibitem[MPS11]{Mayerhofer2011a}
Eberhard Mayerhofer, Oliver Pfaffel, and Robert Stelzer, \emph{On strong
  solutions for positive definite jump diffusions}, Stochastic Processes and
  their Applications \textbf{121} (2011), no.~9, 2072--2086.

\bibitem[PP16]{Perkowski2016}
Nicolas Perkowski and David~J. Pr\"{o}mel, \emph{Pathwise stochastic integrals
  for model free finance}, Bernoulli \textbf{22} (2016), no.~4, 2486--2520.

\bibitem[Pro04]{Protter2004}
Philip~E. Protter, \emph{{Stochastic integration and differential equations}},
  2nd ed., Springer, 2004.

\bibitem[PW07]{Pflug2007}
Georg Pflug and David Wozabal, \emph{Ambiguity in portfolio selection},
  Quantitative Finance \textbf{7} (2007), no.~4, 435--442.

\bibitem[PW16]{Pal2016}
Soumik Pal and Ting-Kam~Leonard Wong, \emph{The geometry of relative
  arbitrage}, Math. Financ. Econ. \textbf{10} (2016), no.~3, 263--293.

\bibitem[Rao62]{Rao1962}
R.~Rao, \emph{{R}elations between weak and uniform convergence of measures with
  applications}, The Annals of Mathematical Statistics \textbf{33} (1962).

\bibitem[RX19]{Ruf2019}
Johannes Ruf and Kangjianan Xie, \emph{Generalised {L}yapunov functions and
  functionally generated trading strategies}, Appl. Math. Finance \textbf{26}
  (2019), no.~4, 293--327.

\bibitem[SSV18]{Schied2018}
Alexander Schied, Leo Speiser, and Iryna Voloshchenko, \emph{Model-free
  portfolio theory and its functional master formula}, SIAM J. Financial Math.
  \textbf{9} (2018), no.~3, 1074--1101.

\bibitem[Str14]{Strong2014}
Winslow Strong, \emph{Generalizations of functionally generated portfolios with
  applications to statistical arbitrage}, SIAM J. Financial Math. \textbf{5}
  (2014), no.~1, 472--492.

\bibitem[SV16]{Schied2016}
Alexander Schied and Iryna Voloshchenko, \emph{Pathwise no-arbitrage in a class
  of delta hedging strategies}, Probab. Uncertain. Quant. Risk \textbf{1}
  (2016), Paper No. 3, 25.

\bibitem[Won15]{Wong2015}
Ting-Kam~Leonard Wong, \emph{Universal portfolios in stochastic portfolio
  theory}, Preprint arXiv:1510.02808 (2015).

\end{thebibliography}
\bibliographystyle{amsalpha}

\end{document}